\documentclass[a4paper, 11pt]{article}

\usepackage[utf8]{inputenc}
\usepackage[english]{babel}
\usepackage{amsthm, amsmath, amssymb}
\usepackage[T1]{fontenc}
\usepackage{fullpage}
\usepackage{lipsum}
\usepackage{authblk}
\usepackage{url} 
\usepackage{hyperref}
\usepackage{graphicx}
\usepackage{mathtools}
\usepackage{bbm}
\usepackage{tikz-cd,tikz-3dplot}
\usepackage{enumitem}
\usepackage{soul}
\usepackage{mathrsfs}
\usepackage{subcaption}
\usepackage{float}
\usepackage{longtable}
\usepackage{array}

\usepackage[maxbibnames=99, style=numeric]{biblatex} 

\newtheorem{assumption}{Assumption}[section]
\newtheorem{definition}[assumption]{Definition}
\newtheorem{theorem}[assumption]{Theorem}
\newtheorem{proposition}[assumption]{Proposition}
\newtheorem{corollary}[assumption]{Corollary}
\newtheorem{lemma}[assumption]{Lemma}

\theoremstyle{remark}

\theoremstyle{remark}
\newtheorem{remark}[assumption]{Remark}

\newcommand{\Acal}{\mathcal{A}}

\newcommand{\Ccal}{\mathcal{C}}

\newcommand{\Ical}{\mathcal{I}}
\newcommand{\Jcal}{\mathcal{J}}

\newcommand{\Scal}{\mathcal{S}}

\newcommand{\Wcal}{\mathcal{W}}

\newcommand{\Prob}{\mathscr{P}}
\newcommand{\Mart}{\mathscr{M}}

\newcommand{\EE}{\mathbb{E}}

\newcommand{\NN}{\mathbb{N}}

\newcommand{\PP}{\mathbb{P}}

\newcommand{\RR}{\mathbb{R}}

\newcommand{\mean}{{\textup{mean}}}
\newcommand{\diam}{{\textup{diam}}}

\newcommand{\CDF}{{\textup{CDF}}}
\newcommand{\supp}{{\textup{supp}}}

\newcommand{\adj}{{\textup{adj}}}
\newcommand{\diag}{{\textup{diag}}}
\newcommand{\iadd}{{i^{*}}}

\newcommand{\interior}{{\textup{int}}}

\numberwithin{equation}{section}

\makeatletter
\renewcommand{\p@enumii}{}            

\makeatother

\newcounter{stepimg}
\renewcommand{\thestepimg}{\alph{stepimg}}

\newcommand{\stepfig}[2]{
  \refstepcounter{figure}\label{#1}\setcounter{stepimg}{0}
  \begin{longtable}{|@{\hspace{.5em}}
  >{\centering\arraybackslash}p{.32\textwidth}@{\hspace{0.5cm}}
  >{\centering\arraybackslash}p{.32\textwidth}@{\hspace{0.5cm}}
  >{\centering\arraybackslash}p{.32\textwidth}
  @{\hspace{1em}}|}
  \hline
  \multicolumn{3}{|@{\hspace{.5em}}l@{\hspace{.5em}}|}{
    \rule{0pt}{1.5em}\textbf{Figure \thefigure: #2}
  }\\[0.5em]
  \hline
  \noalign{\vskip .25em}
}

\newcommand{\closestepfig}{
  \hline
  \end{longtable}
}

\newcommand{\I}[2]{
  \refstepcounter{stepimg}
  \begin{minipage}[t]{1\linewidth}
    \centering
    \includegraphics[width=\linewidth]{#1}\\[-.1em]
    {\small\textbf{(\thestepimg)} #2}
  \end{minipage}
}

\newcommand{\R}[6]{
  \I{#1}{#2}&\I{#3}{#4}&\I{#5}{#6}\\[1.2em]
}

\begin{document}

\title{Existence of $q$-Bass martingales in the semidiscrete setting}
\author{Beatrice Acciaio\thanks{Department of Mathematics, ETH Z\"{u}rich, Switzerland. \emph{beatrice.acciaio@math.ethz.ch}}~~and~Antonio Marini\thanks{Department of Mathematics, ETH Z\"{u}rich, Switzerland. \emph{antonio.marini@math.ethz.ch}}}

\maketitle

\begin{abstract}
The class of $q$-Bass martingales provides a natural answer to a central question in martingale optimal transport: how to construct martingales with prescribed initial and terminal marginals whose transition kernel remains as close as possible to a given reference measure $q$. We prove the existence of $q$-Bass martingales when the initial marginal is supported on finitely many atoms, and establish uniqueness, up to an additive translation constant, of the associated Bass measure.
Our approach is geometric and relies on the analysis of a suitable parametrization of convex polygonal chains.
\end{abstract}

\section{Introduction}
The concept of $q$-Bass martingale finds its roots in Bass’s classical solution to the Skorokhod embedding problem (SEP), introduced in \cite{Ba83}. Given a centered distribution $\nu$, the SEP consists in finding stopping times $\tau$ such that a Brownian motion stopped at $\tau$ has distribution $\nu$. The idea of Bass is that of transporting a Gaussian random variable into the target distribution $\nu$ via a monotone map, then take conditional expectations in order to build a martingale out of it, and finally perform a time change via the Dambis–Dubins–Schwarz theorem thus finding a solution $\tau$ to the SEP.
Thinking of the above Gaussian distribution as the distribution at some time $T$ of a Brownian motion, this construction can be interpreted as a way of stretching it in order to match a target distribution  while preserving the martingale structure.

An extension of this construction appeared more recently in martingale optimal transport through the notion of stretched Brownian motion. In \cite{BaBeHuKa20}, stretched Brownian motion is introduced as the optimizer of the Martingale Benamou--Brenier formula, namely as the martingale with prescribed initial and terminal marginals which stays as close as possible to Brownian motion. This point of view has since led to a substantial literature. The dual formulation and the structural characterization of optimizers were developed further in \cite{BaBeScTs23}, while the associated Bass functional was studied in \cite{BaScTs23}. On the computational side, Conze and Henry-Labordère introduced in \cite{CoHe21} a fixed-point iteration for financial applications, leading to the Bass local volatility model. Its convergence and its use for calibration in dimension one were investigated in \cite{AcMaPa23}. More recently, \cite{HaJoLoObPa26} extended this approach to a multidimensional setting and referred to the resulting procedure as the martingale Sinkhorn algorithm, which generalizes the fixed-point iteration and whose convergence is proved by exploiting the dual formulation.
In this recent literature, the term Bass martingale is often used not only for the original construction with trivial initial law, but also as a synonym for the standard stretched Brownian motion construction, where the reference Brownian motion may start from a non-trivial initial distribution chosen so as to satisfy the prescribed marginal constraints.

In the present article, we follow this terminology. More precisely, the usual Bass martingale corresponds to the construction in Definition~\ref{def:q-Bass} when the reference measure is the standard Gaussian law. Here we study the same construction for a general reference measure $q$. This is the point of view initiated by Tschiderer in \cite{Ts24}, where the Gaussian reference law in the martingale Benamou--Brenier problem is replaced by a general measure $q$, the corresponding dual problem is derived, and the resulting optimizer is described through the notion of a $q$-Bass martingale.

\paragraph{The $q$-Bass martingale.}
The definition below is a reformulation, in the one-dimensional setting, of the definition of $q$-Bass martingale given by Tschiderer in \cite{Ts24}. Tschiderer’s definition is formulated in the $n$-dimensional setting in terms of dual convex potentials, which provides a general but rather abstract description.  In dimension one, however, one can exploit the explicit structure of the optimal transport maps for quadratic cost: by Brenier’s theorem, these maps are given by compositions of quantile and cumulative distribution functions. This allows us to avoid the use of dual potentials and to give a more concrete formulation. We denote by $\Mart(\mu,\nu)$ the set of martingale measures between two distributions $\mu$ and $\nu$. Moreover, given measurable functions $f,g$ and a distribution $\xi$, we set $f\ast g(x)=\int f(x-y)g(y)dy$ and $\xi \star f(x)=\int f(x+y)\xi(dy)$, for $x\in\RR$. See the Notations paragraph below for more details.

\begin{definition}[$q$-Bass martingale] \label{def:q-Bass}
	Let $\mu, \nu \in \Prob_1(\RR)$, $q \in \Prob(\RR)$ such that $q\ll \lambda$, and
	\begin{equation}\label{eq.mmpi}
		(M_0,M_1)\sim\pi\in\Mart(\mu,\nu).
	\end{equation}
	We say that $\pi$ is a $q$-Bass martingale if there exists a measure $\alpha \in \Prob(\RR)$ such that
	\begin{equation}
	\label{eq:q-Bass}
		M_0= (q \star T)(X_0), \qquad M_1=T(X_0+Q), \qquad \text{and} \qquad q \star T \text{ is increasing $\alpha$-a.s.},
	\end{equation}
	where $X_0 \sim \alpha$, $Q\sim q$, $X_0$ is independent of $Q$, and $T=Q_\nu \circ F_{\alpha \ast q}$. Such a  measure $\alpha$ is called Bass measure.
\end{definition}
In particular, any $q$-Bass martingale from $\mu$ to $\nu$ satisfies the following diagram.
\[
\begin{tikzcd}[row sep=0.7in, column sep = 1.4in]
  M_0 \sim \mu \arrow[r, "\text{\normalsize $q$-Bass martingale}"]  & M_1 \sim \nu  \\
  X_0 \arrow[u, "\text{\normalsize $q \star T$}"] \sim \alpha \arrow[r] & X_0+Q \sim \alpha \ast q  \arrow[u, "\text{\normalsize $T$}"]
\end{tikzcd}
\]
From \eqref{eq.mmpi}, a necessary condition for the existence of a $q$-Bass martingale from $\mu$ to $\nu$ is the convex order relation $\mu \preceq_c \nu$; see Definition \ref{def:cvx-order-irreducibility}. Indeed, Strassen's theorem \cite{St65} states that $\Mart(\mu,\nu)$ is non-empty if and only if $\mu$ is dominated by $\nu$ in convex order.

The condition in Definition~\ref{def:q-Bass} that $q \star T$ be increasing $\alpha$-a.s. is equivalent to the strict convexity of the potential $q \star \widehat v$ appearing in \cite[Definition~1.4]{Ts24}. Indeed, in our notation, $T = \nabla \widehat v$.
Moreover, when $q$ is symmetric with respect to $0$ (for instance, in the Gaussian case $q=\gamma$, where $\gamma$ is the standard Gaussian distribution), the operations $\ast$ and $\star$ coincide.

\begin{remark}[On the monotonicity of $q \star T$]
    Proposition~\ref{prop:q-star-T-well-posed} below guarantees that $q \star T$ is well defined for every $\nu \in \Prob_1(\mathbb{R})$ and every $q \in \Prob(\mathbb{R})$ satisfying $q \ll \lambda$. In general, however, whether $q \star T$ is $\alpha$-a.s. strictly monotone depends on the choice of $\alpha$. This can be seen from the following example. Let $\mu=\delta_0$, $\nu=\tfrac12\delta_{-1}+\tfrac12\delta_1$, and $q=\mathrm{Unif}_{[-2,-1]\cup[1,2]}$. 
If $\alpha=\mathrm{Unif}_{[-1/2,1/2]}$, then the transport map is $T(x)=-\mathbf{1}_{(-\infty,0)}(x)+\mathbf{1}_{[0,\infty)}(x)$ and $(q\star T)(\supp(\alpha))=\{0\}$. Hence, $q\star T$ is not increasing $\alpha$-a.s. 
On the other hand, if $\alpha=\delta_0$, then the transport map is again $T(x)=-\mathbf{1}_{(-\infty,0)}(x)+\mathbf{1}_{[0,\infty)}(x)$, but now $\supp(\alpha)=\{0\}$, so $q\star T$ is trivially increasing on $\supp(\alpha)$, and therefore increasing $\alpha$-a.s.
\end{remark}

\begin{proposition}
\label{prop:q-star-T-well-posed}
	Let $\nu \in \Prob_1(\RR)$ and $\alpha, q \in \Prob(\RR)$ such that $q \ll \lambda$. Then the map $q\star T$ in \eqref{eq:q-Bass} is $\alpha$-a.e. well-defined.
\end{proposition}
The proof is postponed to Appendix~\ref{app.post}.

\paragraph{Optimality of the $q$-Bass martingale.} Assuming that $\mu, \nu, q \in \Prob_2(\RR)$, it has been shown by Tschiderer in \cite{Ts24} that a $q$-Bass martingale from $\mu$ to $\nu$, if it exists, is the unique optimizer of the problem 
\begin{equation}
	\label{def:WT^q_S}
		WT^q_\Scal(\mu, \nu) = \sup_{\pi \in \Mart(\mu, \nu)} \Scal(\pi), \qquad \Scal(\pi)=\int_\RR   \mu(dx) \sup_{p \in \Pi(\pi_x, q)} \int_\RR  m b p(dm, db);
\end{equation}
see \cite{BaBeHuKa20} for the classical Bass case ($q=\gamma$).
In particular, under the assumption that $\mu$ and $\nu$ have finite second moment, solving $WT^q_\Scal(\mu, \nu)$ is equivalent to solving
\begin{equation}
	\label{def:WT^q_I}
		WT^q_\Ical(\mu, \nu) = \inf_{\pi \in \Mart(\mu, \nu)} \Ical(\pi), \qquad \Ical(\pi)=\int_\RR    \Wcal_2^2(\pi_x, q) \mu(dx),
\end{equation}
where $\Wcal_2$ denotes the $2$-Wasserstein distance. The fact that the $q$-Bass martingale is a solution to this problem is what best illustrates its characteristic of being the martingale whose transition kernel remains as close as possible to the measure $q$ while matching the prescribed marginals.

Hasenbichler et al. showed in \cite{HaJoLoObPa26} that $WT^q_\Scal(\mu, \nu)$ admits a $q$-Bass martingale as optimizer when $\mu, \nu \in \Prob_p(\RR)$ for $p>1$ and $q$ is normally distributed.

\paragraph{The fixed-point equations.} As shown by Conze and Henry-Labordère in \cite{CoHe21} for the case $q=\gamma$, the problem of computing a $q$-Bass martingale can be formulated in terms of the fixed points of a suitable operator. The same formulation extends to absolutely continuous reference measures with full support. More precisely, for $F\in\CDF$, set
\begin{equation}
\label{def:fixed_point_operator}
\Acal_qF:=F_\mu\circ\bigl(q\star(Q_\nu\circ(q\ast F))\bigr).
\end{equation}
A direct adaptation of \cite[Theorem~2.1]{CoHe21} shows that, if $\supp(q)=\RR$ and $\nu$ is not a Dirac measure, then a $q$-Bass martingale from $\mu$ to $\nu$ exists if and only if there exists a distribution $\alpha\in\Prob(\RR)$ such that
\begin{equation}
\label{eq:fixed_point_eq}
F_\alpha=\Acal_qF_\alpha.
\end{equation}
This fixed-point characterization explains why the Bass distribution $\alpha$ is also referred to in the literature as the fixed-point distribution.

For a general absolutely continuous reference measure $q$, however, the map $q\star T$ need not be globally strictly increasing, so that the characterization in terms of cumulative distribution functions may fail. In this case, the appropriate formulation is given by the quantile equation introduced in \cite{AcMaPa23}, supplemented by the monotonicity condition in Definition~\ref{def:q-Bass}. We thus obtain the following characterization, valid for arbitrary absolutely continuous reference measures.

\begin{theorem}[Quantile characterization]
\label{rmk:fixed-point-syst}
Let $\mu,\nu\in\Prob_1(\RR)$ and let $q\in\Prob(\RR)$ be such that $q\ll\lambda$. Then a $q$-Bass martingale from $\mu$ to $\nu$ exists if and only if there exists a distribution $\alpha\in\Prob(\RR)$ such that
\begin{equation}
\label{eq:fixed-point_equiv}
\int_\RR Q_\nu\left(\int_0^1F_q\bigl(Q_\alpha(u)-Q_\alpha(v)+z\bigr)\,dv\right)\rho_q(z)\,dz
=Q_\mu(u), \qquad \text{$\lambda$-a.e. $u\in(0,1)$}
\end{equation}
and $S_\alpha := q \star T_\alpha$ is increasing $\alpha$-a.s., where $T_\alpha = Q_\nu\circ F_{\alpha\ast q}$.
\end{theorem}

\begin{proof}
For every $x,z\in\RR$,
\[
F_{\alpha\ast q}(x+z)=\int_0^1F_q\bigl(x-Q_\alpha(v)+z\bigr)\,dv,
\]
and therefore the left-hand side of \eqref{eq:fixed-point_equiv} is precisely $S_\alpha(Q_\alpha(u))$. Assume first that a $q$-Bass martingale exists with Bass distribution $\alpha$. Then $(S_\alpha)_\#\alpha=\mu$ and $S_\alpha$ is increasing $\alpha$-a.s. Hence $S_\alpha\circ Q_\alpha=Q_\mu$ $\lambda$-a.e. on $(0,1)$, which yields \eqref{eq:fixed-point_equiv}.

Conversely, suppose that $\alpha$ satisfies \eqref{eq:fixed-point_equiv} and that $S_\alpha$ is increasing $\alpha$-a.s. If $U\sim\text{Unif}_{[0,1]}$, then $S_\alpha(Q_\alpha(U))=Q_\mu(U)$ a.s., and therefore $(S_\alpha)_\#\alpha=\mu$. Let $X_0\sim\alpha$ and $Y\sim q$ be independent, and define $M_0:=S_\alpha(X_0)$ and $M_1:=T_\alpha(X_0+Y)$. Since $\alpha\ast q$ is atomless, $(T_\alpha)_\#(\alpha\ast q)=\nu$, and hence $M_1\sim\nu$. Moreover, $\EE[M_1\mid X_0]=(q\star T_\alpha)(X_0)=S_\alpha(X_0)=M_0$. Thus $(M_0,M_1)$ is a $q$-Bass martingale from $\mu$ to $\nu$.
\end{proof}

\paragraph{An equivalent characterization of convex order.}In the present work, we characterize convex order in terms of integrated quantile functions.

\begin{definition}[Convex order and irreducibility]\label{def:cvx-order-irreducibility}
Let $\mu,\nu\in\Prob_1(\RR)$. We say that $\mu$ is dominated by $\nu$ in
convex order, and write $\mu\preceq_c \nu$, if
\[
	\int_{\RR}\varphi\,d\mu
	\leq
	\int_{\RR}\varphi\,d\nu
\]
for every convex function $\varphi:\RR\to\RR$ such that both integrals are
well-defined.

For $\rho\in\Prob_1(\RR)$, define its potential function by
\[
	C_\rho(x):=\int_{\RR}(y-x)^+\,\rho(dy),
	\qquad x\in\RR.
\]
We say that the pair $(\mu,\nu)$ is
irreducible if $\mu \preceq_c \nu$ and the set
\[
	I:=\{x\in\RR:C_\mu(x)<C_\nu(x)\}
\]
is an interval and satisfies $\mu(I)=1$.
\end{definition}

\begin{remark}[Potential-function characterization of the convex order] \label{rmk:convex-order}
By the potential-function characterization of the convex order, see \cite[Section 2.2]{BeJu16}, for $\mu,\nu\in\Prob_1(\RR)$ one has $\mu \preceq_c \nu$ if and only if $C_\mu(x)\leq C_\nu(x)$ for all $x \in \RR$  and $\mean(\mu)=\mean(\nu)$. 
\end{remark}

\begin{definition}
\label{def:integ_quant}
	Let $\eta \in \Prob_1(\RR)$. The integrated quantile function of $\eta$ is the function $U_\eta:[0,1] \rightarrow \RR$ given by
	\[
		U_\eta(p)= \int_0^p Q_\eta(u) du.
	\]
\end{definition}
Crucial in our analysis will be the following characterization of convex order and irreducibility in terms of the integrated quantile functions.
\begin{proposition}
\label{prop:convex_order_char}
	Let $\eta, \eta' \in \Prob_1(\RR)$. Then the following statements hold:
	\begin{enumerate}[label = (\roman*)]
		\item \label{pot1} $U_\eta$ is well-defined in $[0,1]$ and convex, with $U_\eta(0)=0$ and $U_\eta(1)=\mean(\eta)$.
		\item \label{pot2} The distribution $\eta$ is dominated in convex order by $\eta'$ if and only if 
		\[
			U_\eta(1) = U_{\eta'}(1), \qquad \text{and} \qquad U_\eta(p) \geq U_{\eta'}(p), \quad \text{for all } p \in (0,1).
		\]
		\item \label{pot3} The pair $(\eta, \eta')$ is irreducible if and only if 
		\[
			U_\eta(1) = U_{\eta'}(1), \qquad \text{and} \qquad U_\eta(p) > U_{\eta'}(p), \quad \text{for all } p \in (0,1).
		\]
	\end{enumerate}
\end{proposition}

\begin{figure}[h]
     \centering
          \begin{subfigure}[t]{0.45\textwidth}
         \centering
         \includegraphics[width=\textwidth]{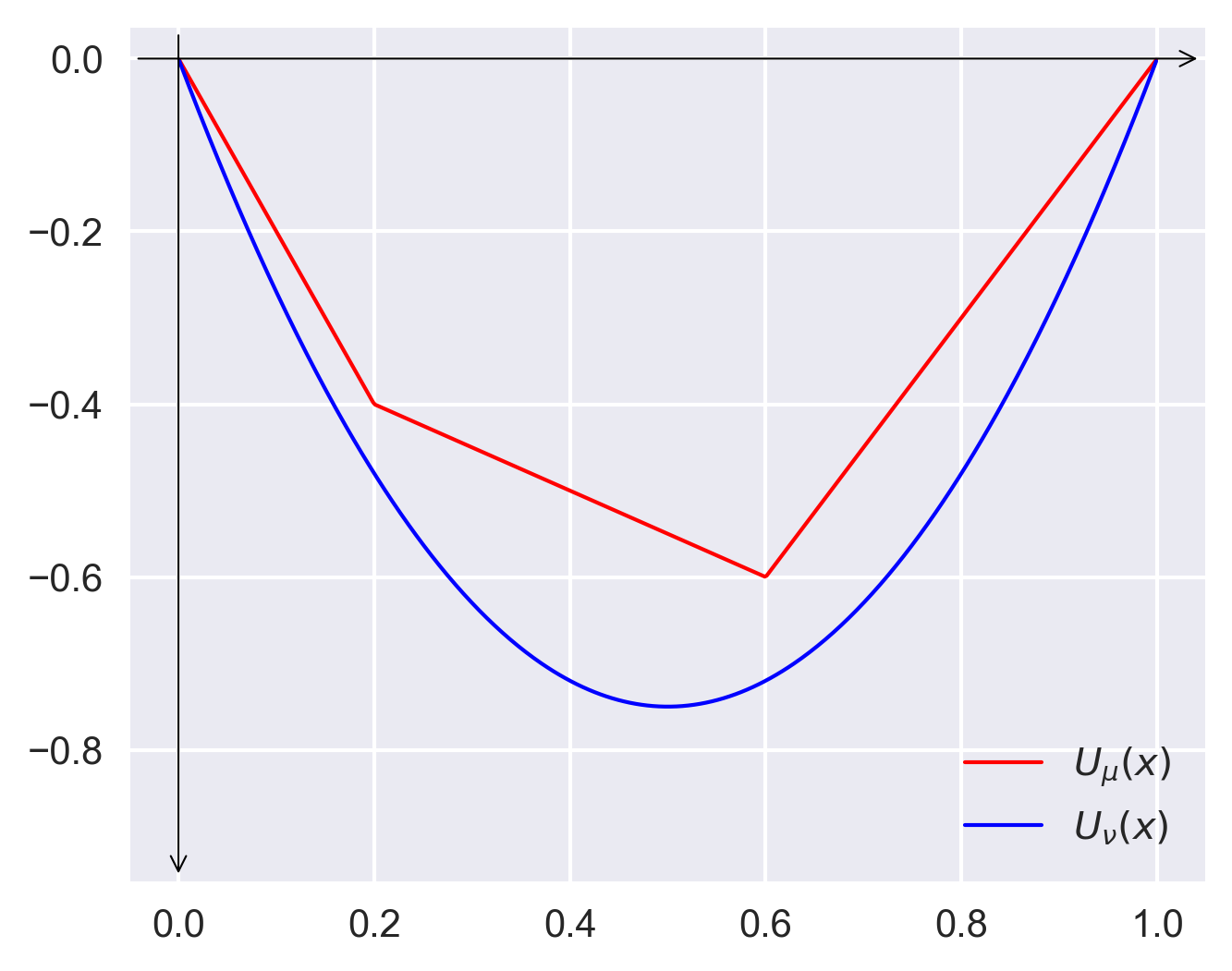}
         \captionsetup{justification=centering}
         \caption{The pair $(\mu, \nu)$ is irreducible\\ ($\mu = 0.2\delta_{-2}+0.4\delta_{-0.5}+0.4\delta_{1.5}$ , $\nu = \text{Unif}_{[-3,3]}$).}
          \label{fig:integrated_quantile_1}

     \end{subfigure}
     \hfill
     \begin{subfigure}[t]{0.45\textwidth}
         \centering
         \includegraphics[width=\textwidth]{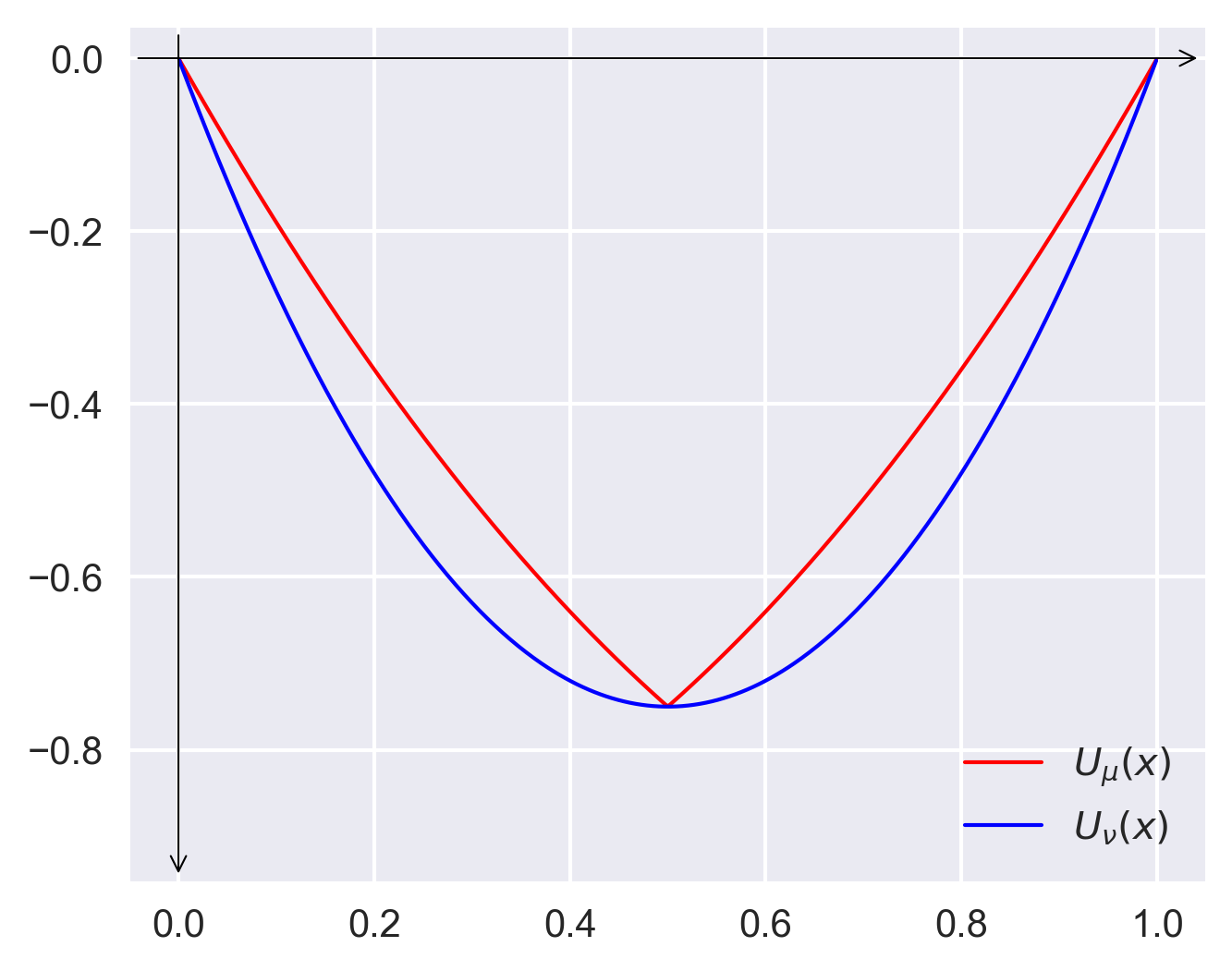}
         \captionsetup{justification=centering}
         \caption{$\mu \preceq_c \nu$, but $(\mu, \nu)$ is not irreducible ($\mu = \text{Unif}_{[-2,-1]\cup[1,2]}$ , $\nu = \text{Unif}_{[-3,3]}$).}
        \label{fig:integrated_quantile_2}

     \end{subfigure}
     \caption{Convex order and irreducibility in terms of integrated quantile functions.}
\end{figure}
The proof is postponed to Appendix~\ref{app.post}.

\subsection{Main contributions}
This section presents the main contributions of the paper; the proofs are deferred to Section~\ref{sec:existence}. 
 Throughout this section, the results are formulated in the semidiscrete setting, meaning that $\mu$ is assumed to be supported on finitely many atoms. This assumption is crucial in the proofs, as it allows the existence problem for $q$-Bass martingales to be linked to the geometry of convex polygonal chains.

\begin{theorem}[Existence of the $q$-Bass martingale for $\mu$ finitely supported]
\label{thm:q-bass-existence}
	Let $\mu, \nu \in \Prob_1(\RR)$, and let $q \in \Prob(\RR)$ be such that $q \ll \lambda$. If $(\mu, \nu)$ is irreducible and $\mu$ is supported on finitely many atoms, there exists a $q$-Bass martingale from $\mu$ to $\nu$.
\end{theorem}

The proof of this result is entirely geometric and relies on the connection between solutions of \eqref{eq:fixed-point_equiv} and convex curves dominating $U_\nu$, in the spirit of Proposition~\ref{prop:convex_order_char}. Since no optimization problem is involved, it is not surprising that no assumptions on $\mu$ and $\nu$ are needed beyond the finiteness of the first moment, which is required for the construction of a martingale. 
This result suggests that the $q$-Bass martingale plays, within the martingale setting, the same role as monotone transport maps do in Monge--Kantorovich transport with squared Euclidean distance: even when the value of the optimization problem is $+\infty$, one can still identify a canonical optimal way to transport one measure into the other.

We now introduce  additional regularity conditions on the problem that are used in many arguments of the present paper, and finally to establish uniqueness of the solution to \eqref{eq:fixed-point_equiv}. In particular, Assumption~\ref{ass:A3} is equivalent to the condition introduced in \cite{HaJoLoObPa26} for the existence of a $\gamma$-Bass martingale attaining the value of $WT^q_\Scal(\mu,\nu)$. 

\begin{assumption}[Regularity of the distributions]\label{ass:technical-assumptions}
Let $\nu \in \Prob_1(\RR)$ and $q \in \Prob(\RR)$ with $\mean(\nu)=0$. We assume that
\begin{enumerate}[label=(A\arabic*)]
	\item \label{ass:A1} the probability measures $\nu$ and $q$ are absolutely continuous w.r.t. Lebesgue;
	\item \label{ass:A2} the supports of $\nu$ and $q$ are intervals,  and their densities are strictly positive $\lambda$-a.e. on the interior of their respective support;
	\item \label{ass:A3} there exist $a>1$ and $b>\frac{a}{a-1}$ such that $\nu \in \Prob_a(\RR)$ and $q \in \Prob_b(\RR)$;
	\item \label{ass:A4} there exists $\varsigma >1$ such that $Q_\nu' \in L^\varsigma(0,1)$ and $\rho_q \in L^{\max\left(\frac{2\varsigma-1}{\varsigma-1}, \frac{a}{a-1} \right)}(\RR)$.
\end{enumerate}
\end{assumption}
\begin{remark}
	In the above assumption $\rho_q \in L^1(\RR)$, thus the interpolation inequality yields 
	\[
		\rho_q \in L^{\max\left(\frac{2\varsigma-1}{\varsigma-1}, \frac{a}{a-1} \right)}(\RR) \implies \rho_q \in L^{\frac{2\varsigma-1}{\varsigma-1}}(\RR) \cap L^{\frac{a}{a-1}}(\RR).
	\]
\end{remark}

\begin{theorem}[Uniqueness of the Bass measure for $\mu$ finitely supported]
	\label{thm:uniqueness-fixed-point-sol}
	Let $\mu, \nu \in \Prob_1(\RR)$ and $q \in \Prob(\RR)$ be such that Assumption~\ref{ass:technical-assumptions} holds. If the pair $(\mu,\nu)$ is irreducible and $\mu$ is supported on finitely many atoms, the  Bass measure is unique up to translation.
\end{theorem}

\paragraph{Why the semidiscrete setting?}
The semidiscrete setting assumed in this paper allows us to show existence of the $q$-Bass martingale by relying on geometric arguments. 
This offers a simpler perspective on the problem and, crucially, provides the foundation to prove existence of the $q$-Bass martingale in the general setting, which we do in the forthcoming paper \cite{AcMa26_general}. 

The semidiscrete setting may also be of interest in its own right, when the initial marginal is not meant to approximate a smooth distribution, but rather to encode a finite set of relevant scenarios. This is the case when most of the uncertainty at some date $T_1$ is absorbed by the realization of a specific event, such as a policy decision, a regulatory announcement, or any other event with a finite number of relevant outcomes. To illustrate this, let $0<T_1<T_2$ and let $X$ be the discounted forward price of an asset under a risk-neutral measure. Suppose that, at time $T_1$, the market-relevant information is summarized by a finite-valued random variable
$Z \in \{1,\dots,n\}$,
where each value of $Z$ corresponds to one possible scenario. For instance, in the case of a central-bank decision, the scenarios may represent different policy outcomes, without specifying the exact size or nature of the decision. If $p_i$ is the risk-neutral probability of scenario $i$, and $x_i$ is the corresponding post-event forward level, then the event-date marginal may be modeled as
\[
    \mu=\mathcal L(X_{T_1}) = \sum_{i=1}^n p_i\delta_{x_i}.
\]
This should be understood as a scenario-based marginal rather than as a numerical discretization of a continuous law.
After $T_1$, the asset is again exposed to ordinary market risk, and the later marginal $\nu=\mathcal L(X_{T_2})$ may be inferred, for example, from vanilla option prices via the Breeden--Litzenberger formula. In particular, $\nu$ may be absolutely continuous. The calibration problem is then to find conditional laws
\[
    \pi_i(dy) = \mathcal L(X_{T_2}\in dy\mid X_{T_1}=x_i), \qquad i=1,\dots,n,
\]
such that $\int y\,\pi_i(dy)=x_i$ and $\sum_{i=1}^n p_i\pi_i=\nu$.
Equivalently, one seeks a martingale coupling between the atomic event-date marginal $\mu$ and the continuous later marginal $\nu$, necessarily with $\mu\preceq_c\nu$.

\paragraph{Notations.}
We write $\Prob(\RR)$ for the probability measures on $\RR$, $\Prob_p(\RR)$ for the subset of probability measures with finite $p$-moment, $p\in[1,\infty)$, and $\Prob_\infty(\RR)$ for the subset of probability measures with compact support. We denote by $\gamma$ the standard Gaussian distribution and by $\phi$ its density. Moreover, we write $\lambda$ for the Lebesgue measure on $\RR$. For $\xi\in\Prob(\RR)$, the notation $\xi\ll\lambda$ means that $\xi$ is absolutely continuous with respect to the Lebesgue measure. In this case, we denote by $\rho_\xi$ its density. 
For  $\xi\in\Prob(\RR)$, we denote by $\supp(\xi)$ its support and by $F_\xi$ its cumulative distribution function. We denote by $Q_\xi:(0,1)\to\RR$ its left-continuous quantile function, defined by $Q_\xi(u):=\inf\{y\in\RR:\ F_\xi(y)\ge u\}$ for $u\in(0,1)$. By convention, we extend $Q_\xi$ to $[0,1]$ by setting $Q_\xi(0):=\inf(\supp(\xi))$ and $Q_\xi(1):=\sup(\supp(\xi))$. For any $\xi\in\Prob_1(\RR)$, we write $\mean(\xi):=\int y\,\xi(dy)$ for its mean.
We use CDF as abbreviation for cumulative distribution function and, with an abuse of notation, we also denote by CDF the set of all cumulative distribution functions on $\RR$.

For $\mu,\nu\in\Prob(\RR)$, we denote by  $\Pi(\mu,\nu)$ the subset of $\Prob(\RR\times\RR)$ of measures with first marginal $\mu$ and second marginal $\nu$. The elements of $\Pi(\mu,\nu)$ are called couplings of $\mu$ and $\nu$. 
For $p\in[1,\infty)$, the $p$-Wasserstein distance between two probability measures $\xi,\zeta\in\Prob_p(\RR)$ is given by
\[
\Wcal_p(\xi,\zeta) :=\inf_{\pi\in\Pi(\xi,\zeta)}\left(\int |x-y|^p\pi(dx,dy)\right)^{1/p}.
\]
We use $\Mart(\mu,\nu)$ for the subset of $\Pi(\mu,\nu)$ containing the measures $\pi$ such that $\mbox{mean}(\pi_x)=x\; \mu\text{-a.e.}$, where $\pi_x$ is the regular conditional disintegration of $\pi$ w.r.t. $\mu$: $\pi(dx,dy)=\mu(dx)\pi_x(dy)$. The elements of $\Mart(\mu,\nu)$ are called martingale couplings of $\mu$ and $\nu$.  

The push-forward measure of $\xi \in \Prob (\RR)$ through a measurable map $T: \RR \rightarrow \RR$, denoted by $T_\# \xi$, is the probability measure such that $T_\#\xi(A)=\xi(T^{-1}(A))$, for any Borel set  $A$ of $\RR$.
For $\xi,\zeta\in\Prob(\RR)$, we write $\xi\ast\zeta$ for the probability measure representing their convolution, so that $\xi\ast\zeta(A)=\int 1_A(x+y)d\xi(x)d\zeta(y)$, for any $A$ Borel set of $\RR$. For two measurable functions $f,g$, their convolution is the function given by $f\ast g(x)=\int f(x-y)g(y)dy$, $x\in\RR$. Moreover, the convolution of $f$ and $\xi$ is the function defined as $\xi \ast f(x)=\int f(x-y)\xi(dy)$, $x\in\RR$. Similarly, we define $\xi \star f(x)=\int f(x+y)\xi(dy)$, $x\in\RR$.
For $A\subseteq\RR$ and $p\in[1,\infty)$, with $L^p(A)$ we denote the set of $\lambda$-measurable functions $f:A\to\RR$ such that $|f|^p$ is $\lambda$-integrable, and with $L^\infty(A)$ we mean the set of $\lambda$-essentially bounded functions $f:A\to\RR$. With an abuse of notation, for $a,b\in\RR$ and $p\in[1,\infty]$, we write $L^p(a,b):=L^p((a,b))$, and analogously for $(a,b], [a,b), [a,b]$.
We equip $L^p(A)$ with convergence w.r.t.\ the usual $\lambda$-$L^p$ norm $\|\cdot\|_p$, with $\|\cdot\|_\infty$ denoting the $\lambda$-essential supremum norm.

\section{The $n$-atomic $q$-Bass maps}
\label{sec:n-atom}

In this section, we introduce and study a finite-dimensional map which, in the $n$-atomic setting, provides an equivalent formulation of the problem \eqref{eq:fixed-point_equiv}. 
From this point on, unless otherwise stated, we work under the standing assumption that $\mu$ has finite support, in the sense of Definition~\ref{def.natom}.

For simplicity, throughout this section we fix $n \ge 2$ and $p_1,\dots,p_n>0$ such that $\sum_{i=1}^n p_i=1$. We also set
\[
p_i^*:=\sum_{j=1}^i p_j,\qquad i=1,\dots,n.
\]

\begin{definition}[$n$-atomic distribution]\label{def.natom}
	Let $\mu \in \Prob(\RR)$. We say that $\mu$ is an $n$-atomic distribution if there exist $x_1, \dots, x_n \in \RR$ such that $x_1 < \dots < x_n$ and 
	\begin{equation}
	\label{n_atoms_repr}
		\mu = \sum_{j=1}^n p_j \delta_{x_j}.
	\end{equation}
\end{definition}
\begin{remark}\label{rem.fpn}
If $\mu \in \Prob(\RR)$ is an $n$-atomic distribution with representation \eqref{n_atoms_repr}, equation \eqref{eq:fixed-point_equiv} writes as
\begin{equation}
\label{eq:fixed_point_atoms1}
	\int_\RR Q_\nu \left (\sum_{j=1}^n p_j F_q(y_i-y_j + z) \right) \rho_q(z) dz = x_i, \quad \text{for any } i \in \{1, \dots, n\},
\end{equation} 
and any  Bass measure is of the form
\begin{equation}
\label{eq:fixed-point-n-atomic}
	\alpha = \sum_{j=1}^n p_j \delta_{y_j}, \qquad \text{where} \quad y_1 < y_2 < ... < y_n.
\end{equation}
Indeed, when $\mu$ is $n$-atomic, its quantile function $Q_\mu$ is piecewise constant. Therefore, the left-hand side of \eqref{eq:fixed-point_equiv} must be constant on each corresponding interval. Consequently, the quantile function $Q_\alpha$ of any  Bass measure is piecewise constant on the same intervals as $Q_\mu$, leading to \eqref{eq:fixed-point-n-atomic}. Moreover, since the right-hand side of \eqref{eq:fixed-point_equiv} is constant on each such interval, \eqref{eq:fixed_point_eq} is equivalent to the system of $n$ equations given in \eqref{eq:fixed_point_atoms1}.
\end{remark}
In view of the above remark, when $\mu$ is an $n$-atomic distribution, studying  the existence of a $q$-Bass martingale from $\mu$ to $\nu$ is equivalent to studying the solutions of \eqref{eq:fixed_point_atoms1}. To this end, we introduce the notion of $n$-atomic $q$-Bass map with respect to the weights $p_1,\dots,p_n$ and terminal distribution $\nu$.

\begin{definition}[$n$-atomic $q$-Bass map]
\label{def:n-atom_q-Bass_map}
Let $\nu \in \Prob_1(\RR), q \in \Prob(\RR)$ such that $q \ll \lambda$, and define the function $g:\RR^{n-1}_{\geq 0}\times\RR\to\RR$ by
\[
g(h,z)=\sum_{j=1}^n p_j F_q\left(z-\sum_{k=1}^{j-1}h_k\right),
\qquad h=(h_1,\dots,h_{n-1})\in\RR^{n-1}_{\geq 0},\ z\in\RR.
\]
The map $f:\RR^{n-1}_{\geq 0}\to\RR^{n-1}$ defined componentwise by
\begin{equation}
\label{def:f_function}
	f_i(h_1, \dots, h_{n-1})
	=
	\int_\RR Q_\nu \big(g(h,z)\big)\,
	\sum_{j=1}^{i} p_j \rho_q \left( z - \sum_{k=1}^{j-1} h_k \right)\, dz,
	\qquad i\in\{1,\dots,n-1\},
\end{equation}
is called the \emph{$n$-atomic $q$-Bass map} with respect to the weights $p_1,\dots,p_n$ and terminal distribution $\nu$.
\end{definition}
In what follows, whenever the weights  are clear from the context, we do not explicitly indicate the dependence of $f$ on them and simply refer to $f$ as the \emph{$n$-atomic $q$-Bass map with respect to $\nu$}.

\begin{remark}\label{rmk:one-coordinate-q-Bass-map}
For later use, we record the form of the maps $g$ and $f$ when only one component of the argument is non-zero. Fix $i\in\{1,\dots,n-1\}$ and denote by $e_i$ the $i$-th vector of the canonical basis of $\RR^{n-1}$. Then, for every $t\geq 0$ and $z\in\RR$,
\[
g(te_i,z)=p_i^* F_q(z)+(1-p_i^*)F_q(z-t).
\]
Consequently,
\[
f_i(te_i)=p_i^* \int_\RR Q_\nu\!\left(p_i^* F_q(z)+(1-p_i^*)F_q(z-t)\right)\rho_q(z)\,dz.
\]
In particular, $t\mapsto g(te_i,z)$ is non-increasing for every fixed $z\in\RR$. If $\nu$ and $q$ satisfy Assumption~\ref{ass:technical-assumptions}, then Proposition~\ref{prop:regulatiry_f} below implies that each component $f_i$ is non-increasing in each variable separately, with all the other variables kept fixed. In particular, for every fixed $(h_j)_{j\neq i}$, the map
\[
t\mapsto f_i(h_1,\dots,h_{i-1},t,h_{i+1},\dots,h_{n-1})
\]
is decreasing on $[0,\diam(\supp(q)))$.
\end{remark}

\begin{proposition}[The $n$-atomic $q$-Bass maps are well-defined]
\label{prop:q-Bass-map-well-defined}
Let $\nu \in \Prob_1(\RR)$ and let $q \in \Prob(\RR)$ be such that $q \ll \lambda$.
Then the $n$-atomic $q$-Bass map $f$ with respect to $\nu$ is well-defined.\end{proposition}

\begin{proof}
Fix $i \in \{1,\dots,n-1\}$. We first show that $f_i$ is well defined. For every $h\in\RR^{n-1}_{\geq 0}$,
\begin{align*}
|f_i(h)|
&\leq \int_\RR \big| Q_\nu \big(g(h,z)\big)\big|\, \sum_{j=1}^{i} p_j \rho_q \left( z - \sum_{\ell=1}^{j-1} h_\ell \right)\, dz \\
&\leq \int_\RR \big| Q_\nu \big(g(h,z)\big)\big|\, \frac{\partial g(h,z)}{\partial z}\, dz = \int_0^1 |Q_\nu(u)|\,du < \infty,
\end{align*}
where the last equality follows from the change of variable $u=g(h,z)$.
\end{proof}

Since, if $y \in \RR^n$ solves \eqref{eq:fixed_point_atoms1}, then $y+c$ is a solution to \eqref{eq:fixed_point_atoms1} for any $c\in \RR$, Proposition~\ref{prop:dimension_reduction} can be used to transform \eqref{eq:fixed_point_atoms1} into a system of $n-1$ equations in $n-1$ variables.

\begin{proposition}[Dimension reduction]
\label{prop:dimension_reduction}
Let $\mu,\nu\in\Prob_1(\RR)$ and $q \in \Prob(\RR)$ be such that $q \ll \lambda$, and assume that $\mu$ is $n$-atomic with representation \eqref{n_atoms_repr}. If \eqref{eq:fixed_point_atoms1} admits a solution, then $\mean(\mu)=\mean(\nu)$. Moreover, $y\in\RR^n$ solves \eqref{eq:fixed_point_atoms1} if and only if $h\in\RR^{n-1}_{>0}$, defined by $h_i:=y_{i+1}-y_i$ for $i=1,\dots,n-1$, solves
\begin{equation}
\label{eq:fixed_point_atoms2}
f_i(h)=x_i^*,\qquad i=1,\dots,n-1,
\end{equation}
where $x_i^*:=\sum_{j=1}^i p_jx_j$ and $f$ is the $n$-atomic $q$-Bass map with respect to $\nu$.
\end{proposition}

\begin{proof}
Let $y\in\RR^n$ and set $h_i:=y_{i+1}-y_i$ for $i=1,\dots,n-1$. Then, for any $i,j\in\{1,\dots,n\}$,
\[
y_i-y_j=\sum_{k=1}^{i-1}h_k-\sum_{k=1}^{j-1}h_k .
\]
Substituting this identity into \eqref{eq:fixed_point_atoms1} and performing, for each fixed $i$, the change of variables
$z\mapsto z+\sum_{k=1}^{i-1}h_k$, we obtain the equivalent system
\begin{equation}\label{eq:sys_g3}
\int_{\RR} Q_\nu\big(g(h,z)\big)\,\rho_q\!\left(z-\sum_{k=1}^{i-1}h_k\right)\,dz = x_i,
\qquad i=1,\dots,n.
\end{equation}
Now fix $i\in\{1,\dots,n\}$. Multiply the $j$-th equation in \eqref{eq:sys_g3} by $p_j$ and sum over $j=1,\dots,i$ to get
\[
\int_{\RR} Q_\nu\big(g(h,z)\big)\,\sum_{j=1}^{i} p_j\,\rho_q\!\left(z-\sum_{k=1}^{j-1}h_k\right)\,dz = x_i^* .
\]
For $i=1,\dots,n-1$, the left-hand side is $f_i(h)$ by \eqref{def:f_function}, hence \eqref{eq:fixed_point_atoms2}. For $i=n$, note that $\partial_z g(h,z)=\sum_{j=1}^{n} p_j\,\rho_q\!\left(z-\sum_{k=1}^{j-1}h_k\right)$, so the last identity reads
\[
\mean(\nu) = \int_{\RR} Q_\nu\big(g(h,z)\big)\,\partial_z g(h,z)\,dz = x_n^* = \mean(\mu).
\]
\end{proof}

To study the solutions of \eqref{eq:fixed_point_atoms2}, we first investigate the regularity properties of the $n$-atomic $q$-Bass maps.

\subsection{Regularity of the $n$-atomic $q$-Bass maps}
In this subsection, we study the regularity of $n$-atomic $q$-Bass maps. In particular, under Assumption~\ref{ass:technical-assumptions}, we show that every $n$-atomic $q$-Bass map admits a potential function; see Proposition~\ref{prop:regulatiry_f}.

\begin{definition}[Potential of the $n$-atomic $q$-Bass map]
\label{def:potential-f}
Let $\nu, q \in \Prob(\RR)$ satisfy Assumption~\ref{ass:technical-assumptions}, and set $D := [0,\diam(\supp(q)))^{n-1}$. Let $f : D \to \RR^{n-1}$ be the corresponding $n$-atomic $q$-Bass map. Define $v : D \times \RR \to \RR$ by $v(h,z) := U_\nu\big(g(h,z)\big)$ and $V : D \to \RR$ by 
\[
V(h) := \int_\RR v(h,z)\,dz.
\]
We call $V$ the \emph{potential} of the $n$-atomic $q$-Bass map $f$.
\end{definition}

\begin{lemma}\label{lem:psi-continuous}
\label{lemma:C2-q-Bass-map}
Let $\nu, q \in \Prob(\RR)$ such that Assumption~\ref{ass:technical-assumptions} holds, let $D = [0, \diam(\supp(q)))^{n-1}$, $1\le j<\ell\le n$, and define the map $\psi_{j\ell}:D\to[0,\infty)$  by
\[
\psi_{j\ell}(h)
:= p_j p_\ell \int_\RR Q_\nu'(g(h,z))\,
\rho_q\left(z-\sum_{k=1}^{j-1}h_k\right)\rho_q\left(z-\sum_{k=1}^{\ell-1}h_k\right)\,dz \ge 0.
\]
If $\psi_{j\ell}$ is well-defined, then it is continuous.
\end{lemma}
\begin{proof}
Let $(\zeta_m)_{m\in\NN}\subseteq C_0([0,1])$ be such that $\zeta_m\to Q_\nu'$ in $L^\varsigma(0,1)$. For $m\in\NN$ define
\[
\psi_{j\ell}^{(m)}(h):=p_jp_\ell\int_\RR \zeta_m(g(h,z))\,
\rho_q\left(z-\sum_{k=1}^{j-1}h_k\right)\rho_q\left(z-\sum_{k=1}^{\ell-1}h_k\right)\,dz.
\]
Fix $h,h'\in D$ and write
\[
|\psi_{j\ell}(h)-\psi_{j\ell}(h')|
\le |\psi_{j\ell}(h)-\psi_{j\ell}^{(m)}(h)|
+|\psi_{j\ell}^{(m)}(h)-\psi_{j\ell}^{(m)}(h')|
+|\psi_{j\ell}^{(m)}(h')-\psi_{j\ell}(h')|.
\]
For the first and third terms, a change of variables $u=g(h,z)$ (respectively $u=g(h',z)$) and Hölder's inequality as above yield
\[
|\psi_{j\ell}(h)-\psi_{j\ell}^{(m)}(h)|
\le p_\ell\,\|Q_\nu'-\zeta_m\|_{L^\varsigma(0,1)}\,
\|\rho_q\|_{L^{\frac{2\varsigma-1}{\varsigma-1}}(\RR)},
\]
and the same bound holds with $h'$ in place of $h$. Hence these terms can be made arbitrarily small uniformly in $h,h'$ by choosing $m$ large.

Fix now $m$. Since $\zeta_m$ is bounded, after the change of variables
$z\mapsto z+\sum_{k=1}^{\ell-1}h_k$ and $z\mapsto z+\sum_{k=1}^{\ell-1}h'_k$ we obtain
\begin{align*}
|\psi_{j\ell}^{(m)}(h)-\psi_{j\ell}^{(m)}(h')|
&\le 2p_jp_\ell\|\zeta_m\|_\infty
\int_\RR \left|
\rho_q\left(z-\sum_{k=1}^{j-1}h_k+\sum_{k=1}^{\ell-1}h_k\right)
-\rho_q\left(z-\sum_{k=1}^{j-1}h'_k+\sum_{k=1}^{\ell-1}h'_k\right)
\right|\rho_q(z)\,dz \\
&\le 2p_jp_\ell\|\zeta_m\|_\infty\,
\|T_{t}\rho_q-T_{t'}\rho_q\|_{L^2(\RR)}\,\|\rho_q\|_{L^2(\RR)},
\end{align*}
where $T_a f=f(\cdot+a)$ and
\[
t=-\sum_{k=1}^{j-1}h_k+\sum_{k=1}^{\ell-1}h_k,
\qquad
t'=-\sum_{k=1}^{j-1}h'_k+\sum_{k=1}^{\ell-1}h'_k.
\]
Since $\rho_q\in L^{\frac{2\varsigma-1}{\varsigma-1}}(\RR)$ and $\frac{2\varsigma-1}{\varsigma-1}>2$, we have $\rho_q\in L^2(\RR)$, and translations are continuous in $L^2(\RR)$. Thus $|\psi_{j\ell}^{(m)}(h)-\psi_{j\ell}^{(m)}(h')|\to 0$ as $h'\to h$ for each fixed $m$. Combining the estimates and letting $m\to\infty$ yields the continuity of $\psi_{j\ell}$.
\end{proof}

\begin{proposition}
\label{prop:regulatiry_f}
Let $\nu, q \in \Prob(\RR)$ satisfy Assumption~\ref{ass:technical-assumptions}, let $D := [0,\diam(\supp(q)))^{n-1}$, let $f$ be the corresponding $n$-atomic $q$-Bass map, and let $V : D \to \RR$ and $v : D \times \RR \to \RR$ be the maps defined in Definition~\ref{def:potential-f}.

Then $V$ is strictly concave on $D$, its restriction to $\interior(D)$ belongs to $C^2(\interior(D))$, and its first- and second-order partial derivatives extend continuously to $D$. Moreover, $\nabla V=f$ and the map $f|_{\interior(D)}:\interior(D)\rightarrow f(\interior(D))$ is a $C^1$-diffeomorphism. More specifically, for every $1\leq r\leq s\leq n-1$,
\begin{equation}
\label{eq:secondDeriv}
\frac{\partial^2 V(h)}{\partial h_s\,\partial h_r} = -\sum_{j=1}^r\sum_{\ell=s+1}^{n}\psi_{j\ell}(h),
\end{equation}
where $\psi_{j\ell}:D \rightarrow \RR$ are continuous non-negative functions, and $\psi_{j(j+1)}$ is strictly positive on $D$ for every $j=1,\dots,n-1$.
\end{proposition}
The proof is postponed to Appendix~\ref{app.post}.

\begin{remark}
\label{rmk:JF_prop}
Proposition~\ref{prop:regulatiry_f} yields, for every $h\in D$,
\[
-D^2V(h)=\sum_{1\le j<\ell\le n} \psi_{j\ell}(h)\, v^{j\ell}(v^{j\ell})^T,
\qquad \psi_{j\ell}(h)\ge 0,
\]
with $\psi_{j(j+1)}(h)>0$ for $j=1,\dots,n-1$, and $v^{j\ell}\in\RR^{n-1}$, where $V$ is the potential function of an $n$-atomic $q$-Bass map.
Let $I\subseteq \{1,\dots,n-1\}$ be such that $|I|=d>1$, and let
\[
H_I(h):= -\bigl(D^2V(h)\bigr)_{I,I}\in \RR^{d\times d}
\]
be the corresponding positive definite principal submatrix of $-D^2V(h)$. Taking principal submatrices in the decomposition above gives
\[
H_I(h)=\sum_{1\le j<\ell\le n} \psi_{j\ell}(h)\,\bigl(v^{j\ell}_I\bigr)\bigl(v^{j\ell}_I\bigr)^T,
\]
where $v^{j\ell}_I\in\RR^{d}$ denotes the restriction to the coordinates in $I$.
For each pair $(j,\ell)$, the restricted vector $v^{j\ell}_I$ is either $0$ or the indicator of a discrete interval in $\{1,\dots,d\}$, hence it equals $v^{ab}\in\RR^{d}$ for some $1\le a<b\le d+1$ after relabelling.
Grouping identical vectors, one finds coefficients $a^{(I)}_{ab}(h)\ge 0$ such that
\[
H_I(h) = \sum_{1\le a<b\le d+1} a^{(I)}_{ab}(h)\, v^{ab}(v^{ab})^T.
\]
Moreover, since $\psi_{j(j+1)}(h)>0$ and all the coefficients in the decomposition are non-negative, it follows that $a^{(I)}_{a(a+1)}(h) > 0$, for all $1 \leq a\leq d$.
Therefore Lemma~\ref{adjLemma} applies to every principal submatrix $H_I(h)$, and in particular for all $i\neq j$,
\[
\frac{|\adj(H_I(h))_{ij}|}{\adj(H_I(h))_{jj}}<1.
\]
\end{remark}

In the next result, we use sets of indices to keep track of which coordinates satisfy certain properties either at the level of the domain or at the level of the image of $f$. More precisely, the subscript ``dom'' refers to coordinates of the arguments $h,\widetilde h$, while the subscript ``cdom'' refers to coordinates of their images $f(h),f(\widetilde h)$ in the codomain.

\begin{proposition}\label{prop:non-expans}
Let $\nu,q\in\Prob(\RR)$ such that Assumption~\ref{ass:technical-assumptions} holds, let $D = [0, \diam(\supp(q)))^{n-1}$ and let $f$ be the corresponding $n$-atomic $q$-Bass map.
Let $h,\widetilde h\in D$ and define
\[
I_{\mathrm{dom}}=\{i\in\{1,\dots,n-1\}:h_i\neq \widetilde h_i\},
\qquad
I_{\mathrm{cdom}}=\{i\in\{1,\dots,n-1\}:f_i(h)=f_i(\widetilde h)\}.
\]
Assume that $I_{\mathrm{cdom}}\subseteq I_{\mathrm{dom}}$ and that $I_{\mathrm{dom}}\setminus I_{\mathrm{cdom}}=\{\iadd\}$ for some $\iadd\in\{1,\dots,n-1\}$.
Then $f_\iadd(h)\neq f_\iadd(\widetilde h)$ and, for every $i\neq \iadd$,
\[
\frac{|f_i(h)-f_i(\widetilde h)|}{|f_\iadd(h)-f_\iadd(\widetilde h)|}<1.
\]
More precisely, set $u:=h-\widetilde h$ and
\[
\bar J:=\int_0^1 Jf(\widetilde h+t u)\,dt.
\]
If $i\notin I_{\mathrm{cdom}}$ and $I:=I_{\mathrm{dom}}\cup\{i\}$, then
\[
\frac{|f_i(h)-f_i(\widetilde h)|}{|f_\iadd(h)-f_\iadd(\widetilde h)|}
=
\left|\frac{\adj\!\left(\bar J_{I,I}\right)_{\iadd i}}{\adj\!\left(\bar J_{I,I}\right)_{ii}}\right|
<1,
\]
whereas if $i\in I_{\mathrm{cdom}}$ the ratio is $0$.
\end{proposition}

\begin{proof}
Since $f=\nabla V$ on $D$ and $V\in C^2(D)$ by Proposition~\ref{prop:regulatiry_f}, the Jacobian $Jf$ is continuous on $D$. Let $u:=h-\widetilde h$ and consider the segment $\widetilde h+t u$, $t\in[0,1]$, which lies in $D$. By the fundamental theorem of calculus,
\[
f(h)-f(\widetilde h)
=\int_0^1 \frac{d}{dt}f(\widetilde h+t u)\,dt
=\int_0^1 Jf(\widetilde h+t u)\,u\,dt
=\bar J\,u,
\]
where $\bar J:=\int_0^1 Jf(\widetilde h+t u)\,dt$. In particular, since $u_j=0$ for $j\notin I_{\mathrm{dom}}$, for every $i\in\{1,\dots,n-1\}$,
\begin{equation}\label{eq:FTC_component}
f_i(h)-f_i(\widetilde h)=\sum_{j\in I_{\mathrm{dom}}}\bar J_{ij}\,u_j.
\end{equation}

By assumption $I_{\mathrm{dom}}\setminus I_{\mathrm{cdom}}=\{\iadd\}$, hence $\iadd\notin I_{\mathrm{cdom}}$ and therefore $f_\iadd(h)\neq f_\iadd(\widetilde h)$.
For every $i\in I_{\mathrm{cdom}}$, we have $f_i(h)-f_i(\widetilde h)=0$, and \eqref{eq:FTC_component} yields
\[
\left(\bar J_{I_{\mathrm{cdom}},I_{\mathrm{dom}}}\right)u_{I_{\mathrm{dom}}}=0.
\]

Set $d:=|I_{\mathrm{dom}}|$. Since $I_{\mathrm{dom}}=I_{\mathrm{cdom}}\cup\{\iadd\}$, the matrix $\bar J_{I_{\mathrm{cdom}},I_{\mathrm{dom}}}$ has size $(d-1)\times d$. Moreover, the principal submatrix $\bar J_{I_{\mathrm{dom}},I_{\mathrm{dom}}}$ is nonsingular because $Jf(\widetilde h+t u)$ is negative definite for every $t\in[0,1]$ by Proposition~\ref{prop:regulatiry_f}.
Therefore $\bar J_{I_{\mathrm{cdom}},I_{\mathrm{dom}}}$ has full row rank $d-1$, and its kernel is one-dimensional. Consequently, there exists $c\in\RR$ such that for every $j\in I_{\mathrm{dom}}$,
\begin{equation}\label{eq:kernel-minors}
u_j
=
c\cdot(-1)^{|\{k \in I_{\mathrm{dom}}: k<j\}|}\det \left( \bar J_{I_{\mathrm{cdom}}, I_{\mathrm{dom}} \setminus \{j\}}\right).
\end{equation}

Substituting \eqref{eq:kernel-minors} into \eqref{eq:FTC_component} and applying Laplace expansion along the row indexed by $i$ gives, for every $i\notin I_{\mathrm{cdom}}$,
\[
f_i(h)-f_i(\widetilde h)
=
c\,\det\!\left(\bar J_{I_{\mathrm{cdom}}\cup\{i\},I_{\mathrm{dom}}}\right).
\]
In particular, using $I_{\mathrm{cdom}}\cup\{\iadd\}=I_{\mathrm{dom}}$,
\[
f_\iadd(h)-f_\iadd(\widetilde h)
=
c\,\det\!\left(\bar J_{I_{\mathrm{dom}},I_{\mathrm{dom}}}\right).
\]
Hence, if $i\notin I_{\mathrm{cdom}}$,
\begin{equation}\label{eq:ratio-det-avgJ}
\frac{|f_i(h)-f_i(\widetilde h)|}{|f_\iadd(h)-f_\iadd(\widetilde h)|}
=
\left|\frac{\det\!\left(\bar J_{I_{\mathrm{cdom}}\cup\{i\},I_{\mathrm{dom}}}\right)}
{\det\!\left(\bar J_{I_{\mathrm{dom}},I_{\mathrm{dom}}}\right)}\right|.
\end{equation}
If $i\in I_{\mathrm{cdom}}$, then $f_i(h)=f_i(\widetilde h)$ and the ratio is $0$, so we may assume $i\notin I_{\mathrm{cdom}}$ from now on. Under the standing hypothesis $I_{\mathrm{dom}}\setminus I_{\mathrm{cdom}}=\{\iadd\}$, this implies $i\notin I_{\mathrm{dom}}$, and we set $I:=I_{\mathrm{dom}}\cup\{i\}$.

Note that $I_{\mathrm{cdom}}\cup\{i\}=I\setminus\{\iadd\}$ and $I_{\mathrm{dom}}=I\setminus\{i\}$. With the convention $\adj(A)_{pq}=(-1)^{p+q}\det(A_{\{1,\dots,|I|\}\setminus\{p\},\{1,\dots,|I|\}\setminus\{q\}})$ applied to the principal submatrix $A=\bar J_{I,I}$, identity \eqref{eq:ratio-det-avgJ} becomes
\[
\left|\frac{\det\!\left(\bar J_{I\setminus\{\iadd\},\,I\setminus\{i\}}\right)}{\det\!\left(\bar J_{I\setminus\{i\},\,I\setminus\{i\}}\right)}\right|
=
\left|\frac{\adj\!\left(\bar J_{I,I}\right)_{\iadd i}}{\adj\!\left(\bar J_{I,I}\right)_{ii}}\right|.
\]

Finally, by Remark \ref{rmk:JF_prop}, every principal submatrix of $-Jf(h)$, $h\in D$, has the decomposition required in Lemma~\ref{adjLemma}. Since this structure is preserved under integration, the same holds for $-\bar J_{I,I}=\int_0^1 \bigl(-Jf(\widetilde h+t u)_{I,I}\bigr)\,dt$, and $-\bar J_{I,I}$ is symmetric positive definite as noted above. Hence Lemma~\ref{adjLemma} applies to $-\bar J_{I,I}$ (and therefore also to $\bar J_{I,I}$), yielding
\[
\left|\frac{\adj\!\left(\bar J_{I,I}\right)_{\iadd i}}{\adj\!\left(\bar J_{I,I}\right)_{ii}}\right|<1.
\]
\end{proof}

Finally, we conclude this section by stating that $n$-atomic $q$-Bass maps are stable with respect to the reference measure q and the terminal distribution $\nu$, and continuous, even without Assumption~\ref{ass:technical-assumptions}. The proof is postponed to Appendix~\ref{app.post}.

\begin{proposition}[Continuity and stability of the $n$-atomic $q$-Bass map]
\label{prop:q-Bass-map-stability}
Let $\nu \in \Prob_1(\RR)$ and let $q\in \Prob(\RR)$ be such that $q \ll \lambda$. Let $(\nu_k)_{k\in\NN}\subseteq \Prob_1(\RR)$ and $(q_k)_{k\in\NN}\subseteq \Prob(\RR)$ be such that
\[
    Q_{\nu_k}\to Q_\nu \quad\text{pointwise a.e.}, \qquad |Q_{\nu_k}|\le H \text{ for some } H\in L^1(0,1),
\]
and
\[
    q_k \ll \lambda, \qquad \rho_{q_k}\to \rho_q \quad\text{in } L^1(\RR).
\]
Let $f$ be the $n$-atomic $q$-Bass map with respect to $\nu$, $f^{(k)}$ be the $n$-atomic $q_k$-Bass map with respect to $\nu_k$ for $k\in\NN$, and $(h^{(k)})_{k\in\NN}\subseteq \RR_{\ge0}^{n-1}$. Then there exists a subsequence $(h^{(k_m)})_{m\in\NN}$, such that one of the following alternatives holds:
\begin{enumerate}[label=(\roman*)]
    \item there exists $h\in\RR_{\ge0}^{n-1}$ such that
    \[
        h^{(k_m)}\to h \qquad\text{and}\qquad f_j^{(k_m)}(h^{(k_m)})\to f_j(h) \quad\text{for every } j\in\{1,\dots,n-1\};
    \]
    \item there exists $i\in\{1,\dots,n-1\}$ such that
    \[
        h_i^{(k_m)}\to+\infty \qquad\text{and}\qquad f_i^{(k_m)}(h^{(k_m)})\to U_\nu(p_i^*).
    \]
\end{enumerate}
In particular, whenever $h^{(k)}\to h$ in $\RR_{\ge0}^{n-1}$, one has
\[
    f_j^{(k)}(h^{(k)})\to f_j(h) \qquad \text{for every } j\in\{1,\dots,n-1\},
\]
and any $n$-atomic $q$-Bass map is continuous.
\end{proposition}

\section{Existence and uniqueness of the Bass measure}
\label{sec:existence}
In this section, we study existence and uniqueness of the Bass measure to \eqref{eq:fixed_point_eq} when $\mu$ is supported on finitely many atoms.
Proposition~\ref{prop:dimension_reduction} shows that, when $\mu$ is $n$-atomic, solving system \eqref{eq:fixed-point_equiv} is equivalent to solving the reduced system \eqref{eq:fixed_point_atoms2} for the associated $n$-atomic $q$-Bass map $f$. In this section, we provide a geometric interpretation of \eqref{eq:fixed_point_atoms2} and use it to prove the existence of a solution when the pair $(\mu,\nu)$ is irreducible.  Since irreducibility implies convex order, $\mu$ and $\nu$ have the same mean. Translating both marginals by this common mean does not affect either irreducibility or the existence of a $q$-Bass martingale, so throughout this section we assume that $\mean(\mu)=\mean(\nu)=0$. We also assume that $\nu\neq\delta_0$, since otherwise $\mu\preceq_c\nu$ would imply $\mu=\delta_0$, contradicting the standing assumption that $\mu$ has $n\geq2$ distinct atoms.

\begin{definition}[Convex polygonal chain]
A convex polygonal chain on $[0,1]$ is the graph $\Ccal\subseteq\RR^2$
of a continuous convex piecewise affine function $\varphi_\Ccal:[0,1]\to\RR$.
We write its vertices as
\[
V_i(\Ccal)=(t_i,y_i), \qquad i=0,\dots,m,
\]
where $0=t_0<t_1<\cdots<t_m=1$. Thus, $\Ccal$ is obtained by joining each pair of consecutive vertices $V_{i-1}(\Ccal)$ and $V_i(\Ccal)$ by a line segment.
\end{definition}

If $\mu\in\Prob(\RR)$ is an $n$-atomic distribution, then its quantile function is piecewise constant with exactly $n$ distinct values. As a consequence, the integrated quantile function $U_\mu$ is piecewise affine on $[0,1]$, and the graph of $U_\mu$ is a convex polygonal chain starting at $(0,0)$ and ending at $(1,\mean(\mu))$. 
For a visual intuition, see $U_\mu$ in Figure \ref{fig:integrated_quantile_1}. This polygonal representation is particularly useful because convex order admits a simple characterization in terms of integrated quantiles (Proposition~\ref{prop:convex_order_char}): the condition $\mu\preceq_c \nu$ can be read as the graph of $U_\mu$ lying above the graph of $U_\nu$, with matching endpoints.

We begin by characterizing those polygonal chains which arise from an $n$-atomic law with prescribed weights $(p_1,\dots,p_n)$ and dominate $U_\nu$.

\begin{definition}
\label{def:lambda_nu}
Let $\nu\in\Prob_1(\RR)$, and let
\[
\Gamma=\bigl\{(p,U_\nu(p)):\,p\in[0,1]\bigr\}.
\]
We denote by $\Lambda_\nu^{p_1,\dots,p_n}$ the family of all convex
polygonal chains $\Ccal\subseteq\RR^2$ whose vertices are of the form
\[
V_0(\Ccal)=(0,0), \qquad
V_i(\Ccal)=(p_i^*,y_i), \quad i=1,\dots,n-1, \qquad
V_n(\Ccal)=(1,0),
\]
and satisfying the following conditions:
\begin{enumerate}[label=(\roman*)]
    \item $\Ccal$ strictly dominates $\Gamma$, namely $\varphi_\Ccal(p)>U_\nu(p)$ for every $p\in(0,1)$;
    \item no three consecutive vertices of $\Ccal$ are collinear.
\end{enumerate}
\end{definition}

The next lemma records the fact that an $n$-atomic law $\mu$ such that $(\mu, \nu)$ is irreducible produces a chain belonging to $\Lambda_\nu^{p_1,\dots,p_n}$.

\begin{lemma}\label{lemma:n_atoms-cvx-order}
Let $\mu, \nu \in \Prob_1(\RR)$ such that the distribution $\mu$ is an $n$-atomic distribution with representation \eqref{n_atoms_repr} and the pair $(\mu, \nu)$ is irreducible. Then the graph of $U_\mu$ is a convex polygonal chain in $\Lambda_\nu^{p_1, \dots, p_n}$.
\end{lemma}

\begin{proof}
It follows immediately from Proposition~\ref{prop:convex_order_char}.
\end{proof}

We now introduce the analogous family of chains generated by the $q$-Bass map $f$. This is the geometric counterpart of the reduced system \eqref{eq:fixed_point_atoms2}: prescribing the values $f_i(h)$ at the abscissas $p_i^*$ means prescribing the intermediate vertices of a polygonal chain.

\begin{definition}\label{def:Lambda_f}
Let $\nu \in \Prob_1(\RR)$, $q \in \Prob(\RR)$ such that $q \ll \lambda$, and let $f$ be the $n$-atomic $q$-Bass map.
We denote by $\Lambda_f$ the collection of convex polygonal chains $\Ccal$ in $\RR^2$ with vertices
\[
(0,0),\ (p_1^*, f_1(h)),\ \dots,\ (p_{n-1}^*, f_{n-1}(h)),\ (1,0),
\]
for some $h \in (0,\diam(\supp(q)))^{n-1}$.
In this case we say that $\Ccal$ is generated by $f(h)$.
\end{definition}

The main result of this section identifies the two geometric families introduced in Definition~\ref{def:lambda_nu} and Definition~\ref{def:Lambda_f}: the convex polygonal chains strictly dominating $U_\nu$ with the prescribed abscissas can always be generated by the $n$-atomic $q$-Bass map and, under Assumptions \ref{ass:A1}--\ref{ass:A2}, these two families coincide. In particular, this provides a geometric existence argument for solutions to \eqref{eq:fixed_point_atoms2}.

\begin{theorem}\label{thm:mainThm-n_atoms}
Let $\nu\in\Prob_1(\RR)$, $q\in \Prob(\RR)$ such that $q \ll \lambda$, and let $f$ be the $n$-atomic $q$-Bass map. Then
 \[
\Lambda^{p_1, \dots, p_n}_\nu \subseteq \Lambda_f.
\] 
Moreover, under Assumptions \ref{ass:A1}--\ref{ass:A2}, we have $\Lambda^{p_1, \dots, p_n}_\nu = \Lambda_f$.
\end{theorem}

Combining Proposition~\ref{prop:dimension_reduction} with the geometric characterization above yields the existence of a  Bass measure as soon as $U_\mu$ belongs to $\Lambda_\nu^{p_1,\dots,p_n}$, which is guaranteed by Lemma~\ref{lemma:n_atoms-cvx-order} under irreducibility.

We split the proof of Theorem~\ref{thm:mainThm-n_atoms} into two parts. The first inclusion is proved in Corollary~\ref{cor:mainThm-n_atoms-part2}, and the opposite one under Assumptions \ref{ass:A1}--\ref{ass:A2} is shown in Proposition~\ref{prop:mainThm-n_atoms-part1}.

Since the abscissae of the intermediate vertices are fixed ($p_1^*,\ldots,p_{n-1}^*$), every $\Ccal\in\Lambda_\nu^{p_1,\dots,p_n}$ is uniquely determined by the ordinates of such vertices. We shall therefore identify $\Ccal$ with the vector $(\Ccal_1,\dots,\Ccal_{n-1}):=(y_1,\dots,y_{n-1})\in\RR^{n-1}$. With this convention, $\Ccal\in\Lambda_\nu^{p_1,\dots,p_n}$ means that the convex polygonal chain with vertices $(0,0)$, $(p_i^*,\Ccal_i)$, $i=1,\dots,n-1$, and $(1,0)$ belongs to $\Lambda_\nu^{p_1,\dots,p_n}$. Similarly, if $f=(f_1,\dots,f_{n-1})$ and $h\in(0,\diam(\supp(q)))^{n-1}$, then $f(h)\in\Lambda_\nu^{p_1,\dots,p_n}$ means that the convex polygonal chain with intermediate vertices $(p_i^*,f_i(h))$, $i=1,\dots,n-1$, belongs to $\Lambda_\nu^{p_1,\dots,p_n}$.

\begin{proposition}
\label{prop:mainThm-n_atoms-part1}
Let $h\in\RR^{n-1}_{\ge 0}$, $\nu\in\Prob_1(\RR)$, $q\in \Prob(\RR)$, such that $q \ll \lambda$ and  Assumptions \ref{ass:A1}--\ref{ass:A2} hold. Let $f$ be the corresponding $n$-atomic $q$-Bass map and fix $i \in \{1, \dots, n-1\}$. 
Then $h_i = 0$ if and only if
\begin{equation}
\label{eq:colinearity}
	f_i(h) = f_{i-1}(h)+\frac{f_{i+1}(h)-f_{i-1}(h)}{p_{i+1}^*-p_{i-1}^*}p_i,
\end{equation}
where we set $f_0 = f_n = 0$.
Additionally, $h_i \geq \diam(\supp(q))$ if and only if
\begin{equation}
\label{eq:reduc_polygonal}
	f_i(h)= U_\nu(p^*_i).
\end{equation}
In particular,
\[
	\Lambda_f \subseteq \Lambda^{p_1, \dots, p_n}_\nu.
\]
\end{proposition}

\begin{proof}
Fix $h\in\RR^{n-1}_{\ge 0}$ and let $\Ccal$ be the convex polygonal chain generated by the vertices $(0,0)$, $(p_1^*,f_1(h))$, $\dots$ , $(1,0)$. For each $i\in\{1,\dots,n\}$, denote by $\widetilde x_i$ the slope of the $i$-th edge of $\Ccal$, namely
\[
	\widetilde x_i
	=\frac{f_i(h)-f_{i-1}(h)}{p_i}
	=\int_\RR Q_\nu (g(h,z)) \rho_q \left( z - \sum_{j=1}^{i-1} h_j \right) dz.
\]
Convexity of $\Ccal$ amounts to showing that these slopes are non-decreasing with $i$.

To compare two consecutive slopes, we write, by change of variables,
\begin{equation}
\label{ineq:coeff-ang-ineq}
	\widetilde x_{i+1}
	= \int_\RR Q_\nu (g(h,z+h_i)) \rho_q \left( z - \sum_{j=1}^{i-1} h_j \right) dz.
\end{equation}
Now observe that, for any fixed $h\in\RR^{n-1}_{>0}$, the map $z\mapsto g(h,z)$ is non-decreasing. Since $h_i\ge 0$, we have $z+h_i\ge z$ and therefore
\[
	g(h,z+h_i)\ge g(h,z)
	\qquad \text{for every } z\in\RR.
\]
Applying $Q_\nu$ preserves this inequality by monotonicity of the quantile map, and multiplying by the non-negative density
$\rho_q\!\left( z - \sum_{j=1}^{i-1} h_j \right)$ preserves it as well. Integrating in $z$ yields $\widetilde x_{i+1}\ge \widetilde x_i$ and proves that $\Ccal$ is convex.

The previous argument also shows that equality holds whenever $h_i=0$. Moreover, since $\nu \ll \lambda$, the only step at which strictness may fail is the comparison $g(h,z+h_i)\ge g(h,z)$, which is an equality for all $z$ if and only if $h_i=0$. Hence, under this additional assumption, if $\widetilde x_{i+1}= \widetilde x_i$, then $h_i =0$.
Geometrically, $\widetilde x_i=\widetilde x_{i+1}$ means that two consecutive edges of $\Ccal$ have the same slope, i.e.\ the three consecutive vertices $\bigl(p_{i-1}^*,f_{i-1}(h)\bigr)$, $\bigl(p_i^*,f_i(h)\bigr)$, $\bigl(p_{i+1}^*,f_{i+1}(h)\bigr)$ are collinear. Writing the
condition that the middle point lies on the segment joining the other two gives exactly \eqref{eq:colinearity}.

Since $\Ccal$ is convex and its vertices have abscissas $0, p_1^*, \cdots, p_{n-1}^*, 1$, the condition $\Ccal\in\Lambda_\nu^{p_1,\dots,p_n}$
is equivalent to requiring that
\[
	f_i(h)> U_\nu(p_i^*),\qquad i\in\{1,\dots,n-1\}.
\]

To verify the above comparison, we use the defining formula for $f_i(h)$ and the fact that, for each fixed $i$, the contribution of the terms of the function $g$
with indices $j>i$ is non-negative. Dropping these terms yields the lower bound
\begin{equation}
\label{ineq:reduc_ineq}
	f_i(h) \geq \int_\RR Q_\nu \left( \sum_{j=1}^i p_j F_q \left( z - \sum_{k=1}^{j-1} h_k \right) \right)\sum_{j=1}^i p_j \rho_q \left( z - \sum_{k=1}^{j-1} h_k \right)dz = U_\nu(p^*_i),
\end{equation}
which is the desired inequality. In particular, the inequality \eqref{ineq:reduc_ineq} becomes strict exactly when the discarded sum
\[
	\sum_{j=i+1}^n p_j F_q \left( z - \sum_{k=1}^{j-1} h_k \right)
\]
is strictly positive on  the support of $\sum_{j=1}^i p_j \rho_q \left( z - \sum_{k=1}^{j-1} h_k \right)dz$, which is equivalent to requiring that the shift $h_i$ does not exceed the diameter of the support of $q$. Consequently, \eqref{eq:reduc_polygonal} holds if and only if $h_i \geq \diam(\supp(q))$. Putting the two parts together, we conclude that, if $h \in (0, \diam(\supp(q)))^{n-1}$, then $\Ccal\in\Lambda_\nu^{p_1,\dots,p_n}$. Hence, $\Lambda_f\subseteq\Lambda_\nu^{p_1,\dots,p_n}$.
\end{proof}

The proof of the inclusion $\Lambda_\nu^{p_1, \dots, p_n} \subseteq \Lambda_f$ relies on the following result, whose proof is given in Section~\ref{sec:subdle-implication}. 
This result establishes that, for any prescribed non-empty set of indices, the corresponding components of a vector $\Ccal\in \Lambda_\nu^{p_1,\dots,p_n}$ can be matched by the $n$-atomic $q$-Bass map $f$, while the remaining components of the argument of $f$ are fixed equal to zero.

\begin{theorem}\label{thm:induction_thm}
Let $\nu\in\Prob_1(\RR)$ and $q\in \Prob(\RR)$ such that $q \ll \lambda$, and set $D:=[0,\diam(\supp(q)))^{n-1}$.
Let $f$ be the corresponding $n$-atomic $q$-Bass map, let $\Ccal\in \Lambda_\nu^{p_1,\dots,p_n}$, and let $\emptyset \neq J \subseteq \{1, \dots, n-1\}$.
Then there exists $\widehat h \in D$ such that
\begin{equation}\label{eq:theoremClaim}
f_i(\widehat h) = \Ccal_i \quad \text{for all } i \in J,
\qquad\text{and}\qquad
\widehat h_i = 0 \quad \text{for all } i \notin J.
\end{equation}
\end{theorem}

\begin{corollary}\label{cor:mainThm-n_atoms-part2}
Let $\nu\in\Prob_1(\RR)$, $q\in \Prob(\RR)$ such that $q \ll \lambda$, and let $f$ be the corresponding $n$-atomic $q$-Bass map. Then
\[
\Lambda^{p_1, \dots, p_n}_\nu \subseteq \Lambda_f.
\]
\end{corollary}

\begin{proof}
Let $\Ccal \in \Lambda^{p_1, \dots, p_n}_\nu$ and set $J:=\{1,\dots,n-1\}$. By Theorem~\ref{thm:induction_thm} applied to the polygonal chain $\Ccal$ and the index set $J$, there exists $\widehat h\in[0,\diam(\supp(q)))^{n-1}$ such that $f(\widehat h)=\Ccal$. In particular, $\widehat h \in (0,\diam(\supp(q)))^{n-1}$. Otherwise, $\Ccal$ would have at least three collinear vertices. Hence $\Ccal\in\Lambda_f$, proving the inclusion.
\end{proof}
\medskip

We conclude this section with the proofs of our main results, that is, existence and uniqueness of the $q$-Bass martingale for $\mu$ finite, stated in the Introduction.

\medskip
\begin{proof}[Proof of Theorem~\ref{thm:q-bass-existence}]
By Lemma~\ref{lemma:n_atoms-cvx-order}, the graph of $U_\mu$ belongs to $\Lambda_\nu^{p_1,\dots,p_n}$. Hence, by Theorem~\ref{thm:mainThm-n_atoms}, there exists $h \in (0,\diam(\supp(q)))^{n-1}$ such that $f_i(h)=U_\mu(p_i^*)=x_i^*$ for all $i=1,\dots,n-1$, that is, $h$ solves \eqref{eq:fixed_point_atoms2}, where $f$ is the $n$-atomic $q$-Bass map with respect to $\nu$. The conclusion follows from Proposition~\ref{prop:dimension_reduction}, Remark~\ref{rem.fpn} and Theorem~\ref{rmk:fixed-point-syst}.
\end{proof}
\medskip

\begin{proof}[Proof of Theorem~\ref{thm:uniqueness-fixed-point-sol}]
Assume that $y,y' \in \RR^n$ are two solutions to \eqref{eq:fixed_point_atoms1}. By Proposition~\ref{prop:dimension_reduction}, their increment vectors $h,h' \in \RR^{n-1}_{\geq 0}$, defined by $h_i:=y_{i+1}-y_i$ and $h_i':=y_{i+1}'-y_i'$ for $i=1,\dots,n-1$, satisfy $f_i(h)=f_i(h')=U_\mu(p_i^*)=x_i^*$ for all $i=1,\dots,n-1$, where $f$ is the $n$-atomic $q$-Bass map with respect to $\nu$. Moreover, by Proposition~\ref{prop:mainThm-n_atoms-part1}, we have $h,h' \in (0,\diam(\supp(q)))^{n-1}$. Since $f$ is the gradient of a strictly concave function on $(0,\diam(\supp(q)))^{n-1}$ by Proposition~\ref{prop:regulatiry_f}, it follows that $h=h'$. Therefore, the Bass measure is unique up to translation.
\end{proof}

\subsection{Proof of Theorem~\ref{thm:induction_thm}}\label{sec:subdle-implication}

The proof of Theorem~\ref{thm:induction_thm} is by induction on $|J|$. We first introduce some auxiliary definitions and preliminary results that will be used repeatedly in the inductive argument.

\begin{remark}[Restriction and merged weights]
\label{rmk:restrictions}
Let $\emptyset \not = I \subsetneq \{1, \dots, n-1\}$ and let $f$ be an $n$-atomic $q$-Bass map. 
Consider the coordinate subspace
\[
\RR^{I}_{\geq 0}\ \cong\ \{h\in \RR^{n-1}_{\geq 0} : h_i = 0 \text{ for all } i \notin I\},
\]
and let $f_{I}:\RR^{I}_{\geq 0}\to\RR^{I}$ be the map obtained by restricting $f$ to this subspace and keeping only the components indexed by $I$, namely
\[
f_{I}(h_I):=\big(f_i(h)\big)_{i\in I},
\qquad\text{where } h_j=
\begin{cases}
(h_I)_j,& j\in I,\\
0,& j\notin I.
\end{cases}
\]
Then $f_{I}$ can be identified with an $m$-atomic $q$-Bass map, where $m:=|I|+1$, with respect to the merged weights obtained by summing the original weights across the blocks determined by $I$.
More precisely, writing $I=\{i_1<\dots<i_{m-1}\}$ and setting $i_0:=0$, $i_m:=n$, define
\[
\bar p_\ell:=\sum_{j=i_{\ell-1}+1}^{i_\ell} p_j,\qquad \ell=1,\dots,m.
\]
With these weights $(\bar p_1,\dots,\bar p_m)$, the map $f_{I}$ has the same form as \eqref{def:f_function} (with $n$ replaced by $m$ and $(p_j)_j$ replaced by $(\bar p_\ell)_\ell$).
\end{remark}

\begin{definition}[Minimal diagonal gap]
\label{def:minimal-diagonal-gap}
Let $\Ccal \in \Lambda_\nu^{p_1,\dots,p_n}$. For every $i\in\{1,\dots,n-1\}$, define
\[
\delta_i:=\mathrm{dist}\Big((p_i^*,\Ccal_i),\,\big[(p_{i-1}^*,\Ccal_{i-1}),(p_{i+1}^*,\Ccal_{i+1})\big]\cap\{x=p_i^*\}\Big).
\]
The quantity
\[
\delta:=\min_{1\le i\le n-1}\delta_i
\]
is called the \emph{minimal diagonal gap} of $\Ccal$. An index $i_\delta\in\{1,\dots,n-1\}$ is called a \emph{minimal-gap index} with respect to $\Ccal$ if
\[
i_\delta\in\arg\min_{1\le j\le n-1}\delta_j.
\]
\end{definition}

\begin{figure}
     \centering
          \begin{subfigure}{0.495\textwidth}
         \centering
         \includegraphics[width=\textwidth]{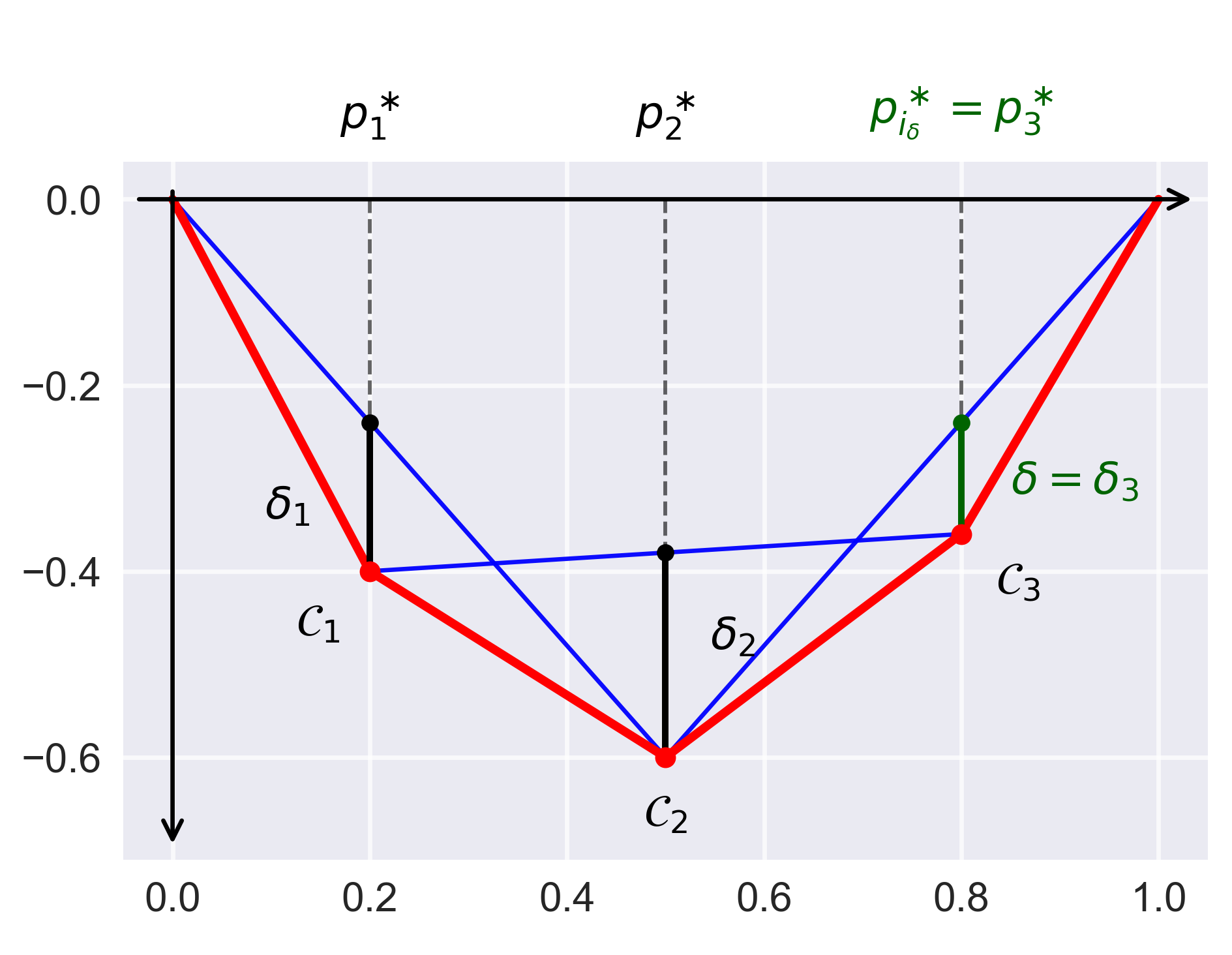}
         \captionsetup{justification=centering}
			\caption{The polygonal chain $\Ccal$, shown in red, has minimal-gap index $3$.}          
			\label{fig:minimal-diagonal-gap}
     \end{subfigure}
     \hfill
     \begin{subfigure}{0.495\textwidth}
         \centering
         \includegraphics[width=\textwidth]{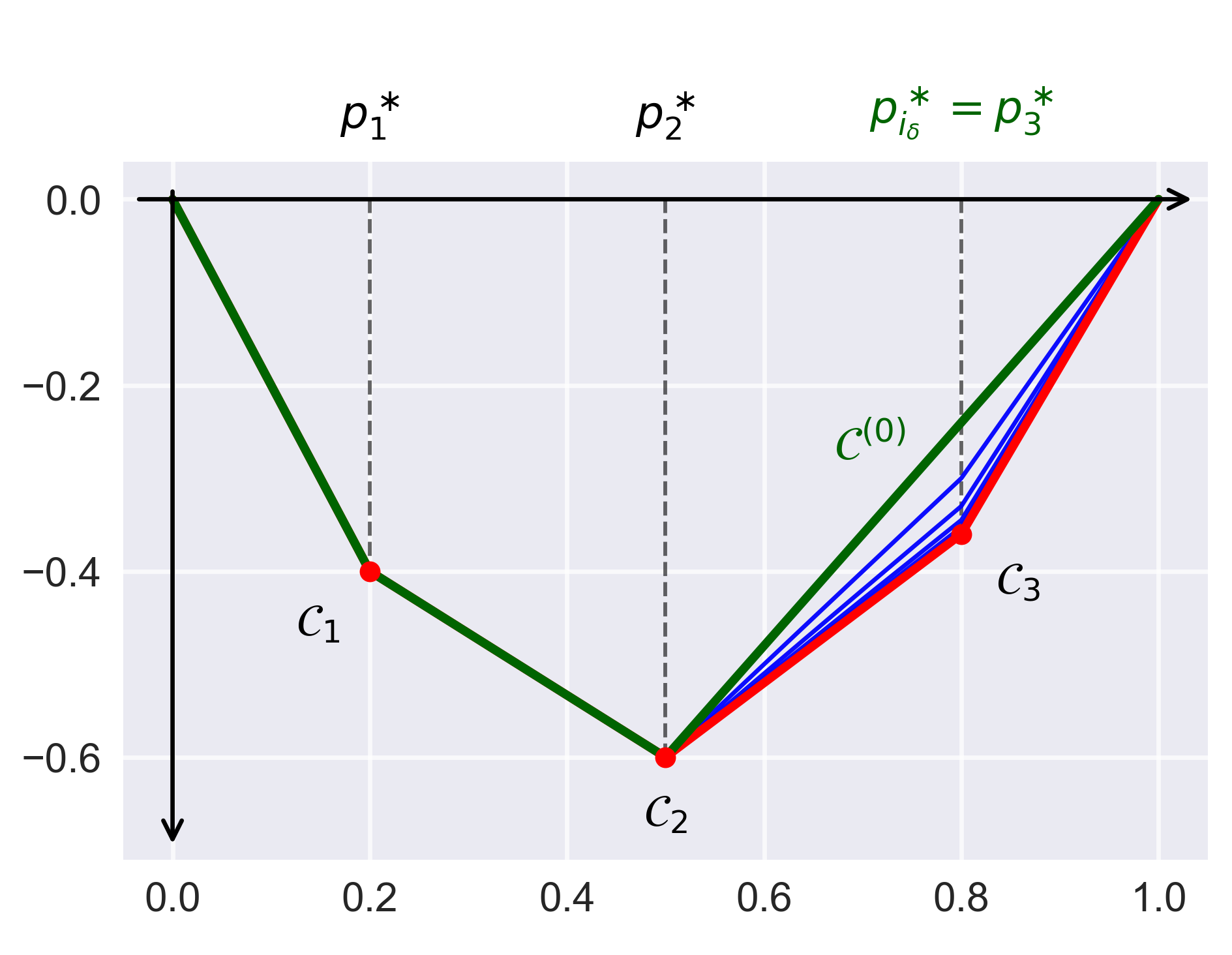}
         \captionsetup{justification=centering}
		\caption{$\Ccal^{(0)}$ is shown in green, while the $1/2$-approximations of $\Ccal$ are shown in blue.}
        \label{fig:L-approximation}
     \end{subfigure}
     \caption{Minimal diagonal gap and $L$-approximations of a convex polygonal chain $\Ccal$.}
\end{figure}

\begin{definition}[$L$-approximation of $\Ccal$ of order $r$ along a minimal-gap index]
Let $\Ccal\in\Lambda_\nu^{p_1,\dots,p_n}$, $\delta$ be its minimal
diagonal gap, and $i_\delta$ be a minimal-gap index. Let $L\in(0,1)$ and
$r\in\NN$. 
Denote by $\widehat y_{i_\delta}$ the ordinate at $p_{i_\delta}^*$ of the
diagonal joining the two neighboring vertices $V_{i_\delta-1}(\Ccal)$ and $V_{i_\delta+1}(\Ccal)$, i.e., $\widehat y_{i_\delta}=\Ccal_{i_\delta} + \delta$.

The \emph{$L$-approximation of $\Ccal$ of order $r$ along $i_\delta$} is the
polygonal chain $\Ccal^{(r)}$ with the same abscissas as $\Ccal$ and ordinates
given by
\[
\Ccal^{(r)}_i=\Ccal_i
\qquad\text{for } i\neq i_\delta,
\]
and
\[
\Ccal^{(r)}_{i_\delta}
=
\widehat y_{i_\delta}-(1-L^r)\delta
=
\Ccal_{i_\delta}+L^r\delta .
\]
\end{definition}

\begin{remark}
\label{rmk:approx}
Let $\Ccal\in \Lambda_\nu^{p_1,\dots,p_n}$, let $L\in(0,1)$, fix a minimal-gap index $i_\delta$ with respect to $\Ccal$, and denote by $\Ccal^{(r)}$ the $L$-approximation of $\Ccal$ of order $r$ along $i_\delta$, for every $r\in\NN$. Then $\Ccal^{(r)}$ is a convex polygonal chain and
\[
\lim_{r\to\infty}\Ccal^{(r)}_i=\Ccal_i,\qquad \text{for all } i\in\{1,\dots,n-1\}.
\]
\end{remark}

\begin{definition}[Error map of order $r$]
Let $\nu\in\Prob_1(\RR)$, $q\in \Prob(\RR)$ such that $q \ll \lambda$, $f$ be the corresponding $n$-atomic $q$-Bass map. Let $\Ccal \in \Lambda_\nu^{p_1,\dots,p_n}$, $i_\delta$ be a minimal-gap index with respect to $\Ccal$, $L\in(0,1)$, and fix $r\in\NN$. Let $\Ccal^{(r)}$ be the $L$-approximation of $\Ccal$ of order $r$ along $i_\delta$.
Define the map $E^{(r)}:\RR^{n-1}_{\ge 0}\to\RR^{n-1}$ by
\[
E^{(r)}_i(h):=f_i(h)-\Ccal^{(r)}_i,\qquad i\in\{1,\dots,n-1\}.
\]
The map $E^{(r)}$ is called the \emph{error map of order $r$} associated with $\Ccal^{(r)}$.
\end{definition}

\paragraph{Idea of the proof of Theorem~\ref{thm:induction_thm}.}
We will first prove Theorem~\ref{thm:induction_thm} under the additional Assumption~\ref{ass:technical-assumptions}. The main tool is Proposition~\ref{prop:non-expans}, which should be interpreted as a non-expansiveness principle for the map $f$. The idea is the following. For simplicity, assume that the set $J$ appearing in the statement of Theorem~\ref{thm:induction_thm} is the whole set $\{1,\dots,n-1\}$. Suppose that we are given a vector $h \in D:=[0,\diam(\supp(q)))^{n-1}$ and that we want to modify it into a new vector $\widetilde h \in D$ in order to improve the distance between $f(h)$ and a target polygonal chain. In the notation of Proposition~\ref{prop:non-expans}, the set $I_{\mathrm{dom}}=\{i:h_i\neq \widetilde h_i\}$ is the set of coordinates of $h$ which are changed, while $I_{\mathrm{cdom}}=\{i:f_i(h)=f_i(\widetilde h)\}$ is the set of components of $f$ which remain fixed during the correction. If $I_{\mathrm{cdom}} \subseteq I_{\mathrm{dom}}$ and $I_{\mathrm{dom}}\setminus I_{\mathrm{cdom}}=\{\iadd\}$, then only the component $f_{\iadd}$ is effectively changed, whereas the components of $f$ indexed by $I_{\mathrm{cdom}}$ are kept fixed. Proposition~\ref{prop:non-expans} says that every other component $f_i$, with $i\neq \iadd$, moves by a strictly smaller amount than $f_{\iadd}$. More precisely, the ratios between these variations are given by the adjugate minors of the averaged Jacobian $\bar J=\int_0^1 Jf(\widetilde h+t(h-\widetilde h))\,dt$, and they are all strictly smaller than $1$.

This is the basic mechanism which allows us to construct sequences improving the error in the $L^\infty$ sense. If one component of $f(h)$ is far from the value of the target convex polygonal chain and we want to pass to a different $\widetilde h$ in order to improve this, then Proposition~\ref{prop:non-expans} ensures that the error created in the other components is smaller than the correction we have just made. In order to use this mechanism uniformly, we work inside a compact cube $[0,H]^{n-1}$ and choose a constant $L\in(0,1)$ which bounds all the relevant ratios of adjugate minors, as in Lemma~\ref{lemma:induction} below. Thus, after the constant $L$ has been fixed, each correction can increase the remaining errors by at most a factor $L$ of the error which has just been corrected.

The role of Lemma~\ref{lemma:induction} is to turn this non-expansiveness principle into an actual construction. Indeed, Proposition~\ref{prop:non-expans} can be applied only after finding two vectors $h,\widetilde h\in D$ such that, among the coordinates where $h$ and $\widetilde h$ differ, the corresponding components of $f(h)$ and $f(\widetilde h)$ agree in all but one coordinate. This is precisely what the lemma provides. Starting from a point $h$ which already matches a given set of vertices, in the sense that $E_i^{(r)}(h)=0$ for $i\in I$, the lemma constructs a new point $\widehat h$ which also matches one additional vertex $\iadd$, while keeping the previously matched components of $f$ fixed. Thus $E_i^{(r)}(\widehat h)=0$ for every $i\in I\cup\{\iadd\}$, and the construction leaves unchanged the coordinates of $h$ outside $I\cup\{\iadd\}$.

The key point is that Lemma~\ref{lemma:induction} is proved by induction over subsets of indices. If, before applying the lemma, the total error is bounded by $\sum_{k=1}^m L^kK\delta$, then after adding the new index the total error is bounded by $\sum_{k=1}^{m+1} L^kK\delta$ (see Figure \ref{fig:steps}). Thus each application of the lemma contributes one additional term to a geometric series. This reflects exactly the mechanism behind Proposition~\ref{prop:non-expans}: correcting one component may introduce errors in the others, but those errors are bounded by a factor controlled by $L$. The induction over subsets is therefore organized so that, after all corrections have been performed, the total variation produced by the procedure is controlled by the sum of this geometric series.

We now explain why the minimal-gap index and the $L$-approximations are introduced. Let $\delta$ be the minimal diagonal gap of $\Ccal$, and let $i_\delta$ be a minimal-gap index. If we leave the vertex $i_\delta$ free, the natural first polygonal chain to consider is $\Ccal^{(0)}$: it agrees with $\Ccal$ at all intermediate vertices except possibly at $i_\delta$, and at $i_\delta$ it is obtained by joining the neighbouring vertices. In this sense, $\Ccal^{(0)}$ is the closest polygonal chain to $\Ccal$ once the vertex $i_\delta$ is released. The distance between $\Ccal^{(0)}$ and $\Ccal$ is precisely measured by the minimal gap $\delta$ at the index $i_\delta$.

However, we do not try to pass directly from $\Ccal^{(0)}$ to $\Ccal$. If we attempted to correct the entire gap $\delta$ at once, the errors produced by the successive applications of Lemma~\ref{lemma:induction} could accumulate to a quantity of order $\sum_{k>0}L^k\delta$, which is not necessarily smaller than the initial error $\delta$. For this reason, we first aim at the intermediate approximation $\Ccal^{(1)}$, rather than at $\Ccal$ itself. At the minimal-gap index, the difference between $\Ccal^{(0)}$ and $\Ccal^{(1)}$ is only $K\delta$, with $K=1-L$. Therefore, the total error generated by the iterative corrections is bounded by $K\sum_{k>0}L^k\delta=L\delta$, which is strictly smaller than $\delta$. This is precisely the purpose of the normalization by $K$: it ensures that the correction procedure remains within the room allowed by the minimal diagonal gap, a key requirement for constructing the $L^\infty$-non-expansive sequence.

Essentially, Lemma~\ref{lemma:induction} is used as a tool to construct points matching the successive approximations of $\Ccal$. More precisely, it first gives a point $h^{(1)}$ such that $f_i(h^{(1)})=\Ccal_i^{(1)}$ for every $i$. Then the same argument is applied to the second approximation, giving a point $h^{(2)}$ such that $f_i(h^{(2)})=\Ccal_i^{(2)}$ for every $i$. Iterating this procedure, for every $k\in\NN$ we obtain a point $h^{(k)}$ satisfying $f_i(h^{(k)})=\Ccal_i^{(k)}$ for every $i$. The sequence $(h^{(k)})_{k\in\NN}$ is bounded, by the choice of the compact set determined by $H$. Hence it admits a limit point $\widehat h$. By Remark \ref{rmk:approx}, we have $\Ccal_i^{(k)}\to\Ccal_i$ for every $i$, and by continuity of $f$ we can pass to the limit in $f_i(h^{(k)})=\Ccal_i^{(k)}$. Therefore $f_i(\widehat h)=\Ccal_i$ for every $i$. This proves Theorem~\ref{thm:induction_thm} under the additional Assumption~\ref{ass:technical-assumptions}.

Finally, the general case is obtained by approximation. We choose sequences $(\nu_k)_{k \in \NN}$ and $(q_k)_{k \in \NN}$ satisfying Assumption~\ref{ass:technical-assumptions} and converging to $\nu$ and $q$, respectively. If $f^{(k)}$ denotes the corresponding $q_k$-Bass map, the result already proved gives, for each $k$, a point $\widetilde h^{(k)}$ such that $f_i^{(k)}(\widetilde h^{(k)})=\Ccal_i$ for every $i$. Passing to a limit point and using the convergence of $f^{(k)}$ to $f$, we obtain the desired point $h^*$ for the original pair $(\nu,q)$.

\begin{remark}
The minimal-gap index need not be unique. Determining a minimal-gap index and constructing the corresponding $L$-approximations along the chosen minimal-gap index are crucial for ensuring that the algorithm works. Nevertheless, the particular choice of the minimal-gap index is irrelevant for the purpose of determining a vector $h \in D$ such that $f(h)$ matches the prescribed components of $\Ccal_i$.
\end{remark}

\stepfig{fig:steps}{Illustration of the algorithm in Theorem ~\ref{thm:induction_thm} and Lemma~\ref{lemma:induction}.}

\R{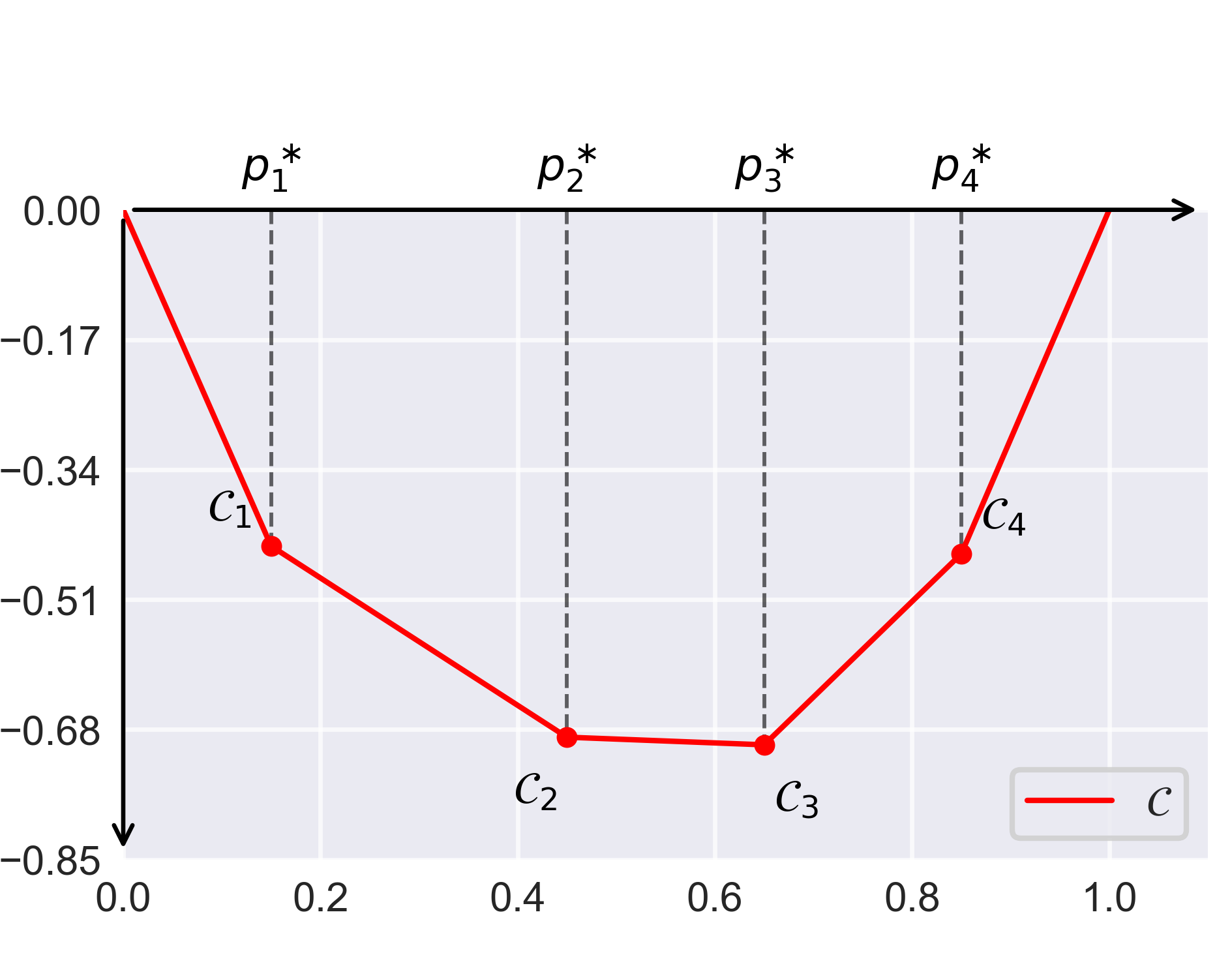}{Fix the constants $L$ and $K=1-L$.}
{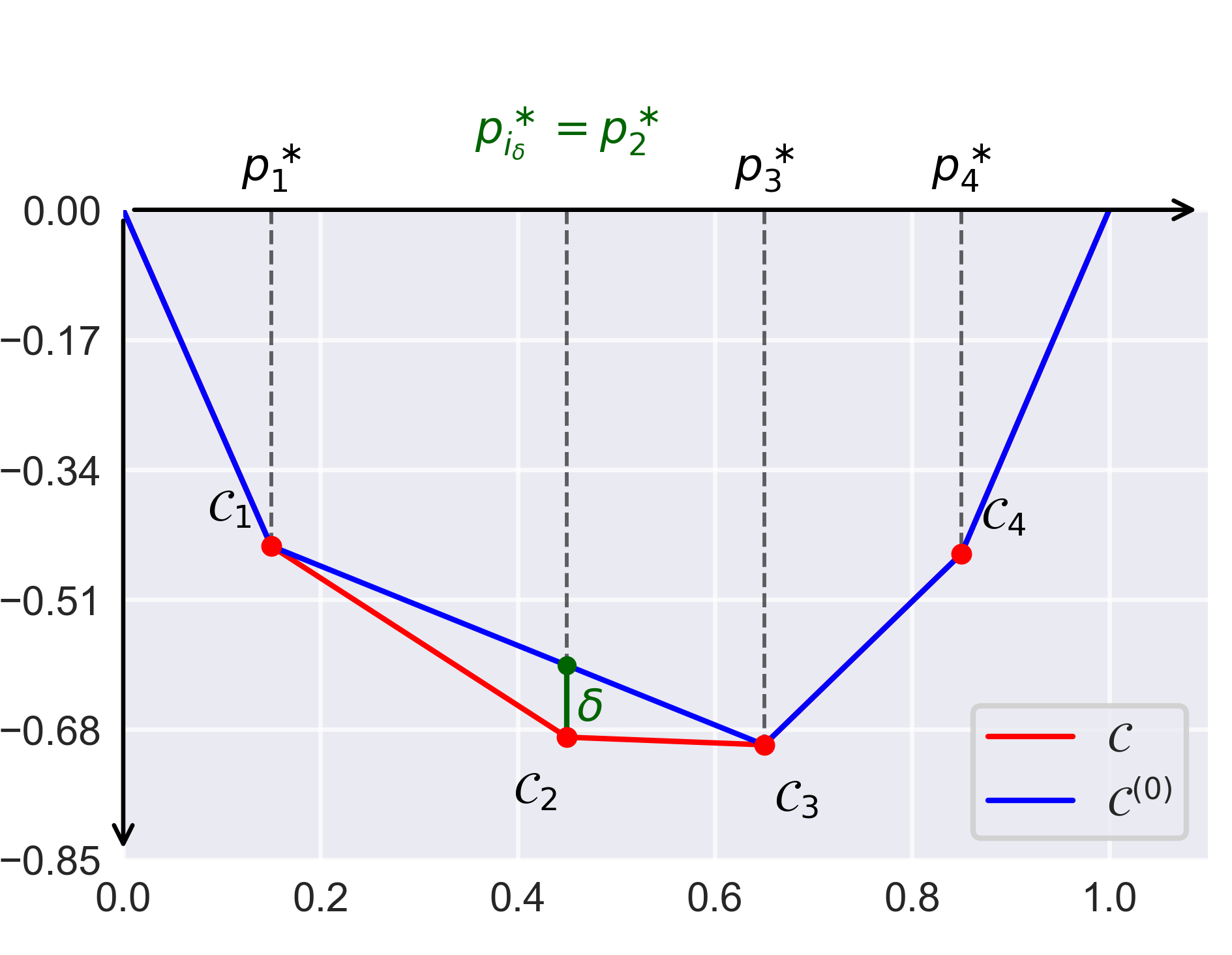}{The minimal-gap index is $i_\delta=2$. By induction, choose $h$ such that $f_i(h)=\Ccal_i^{(0)}$ for all $i$.}
{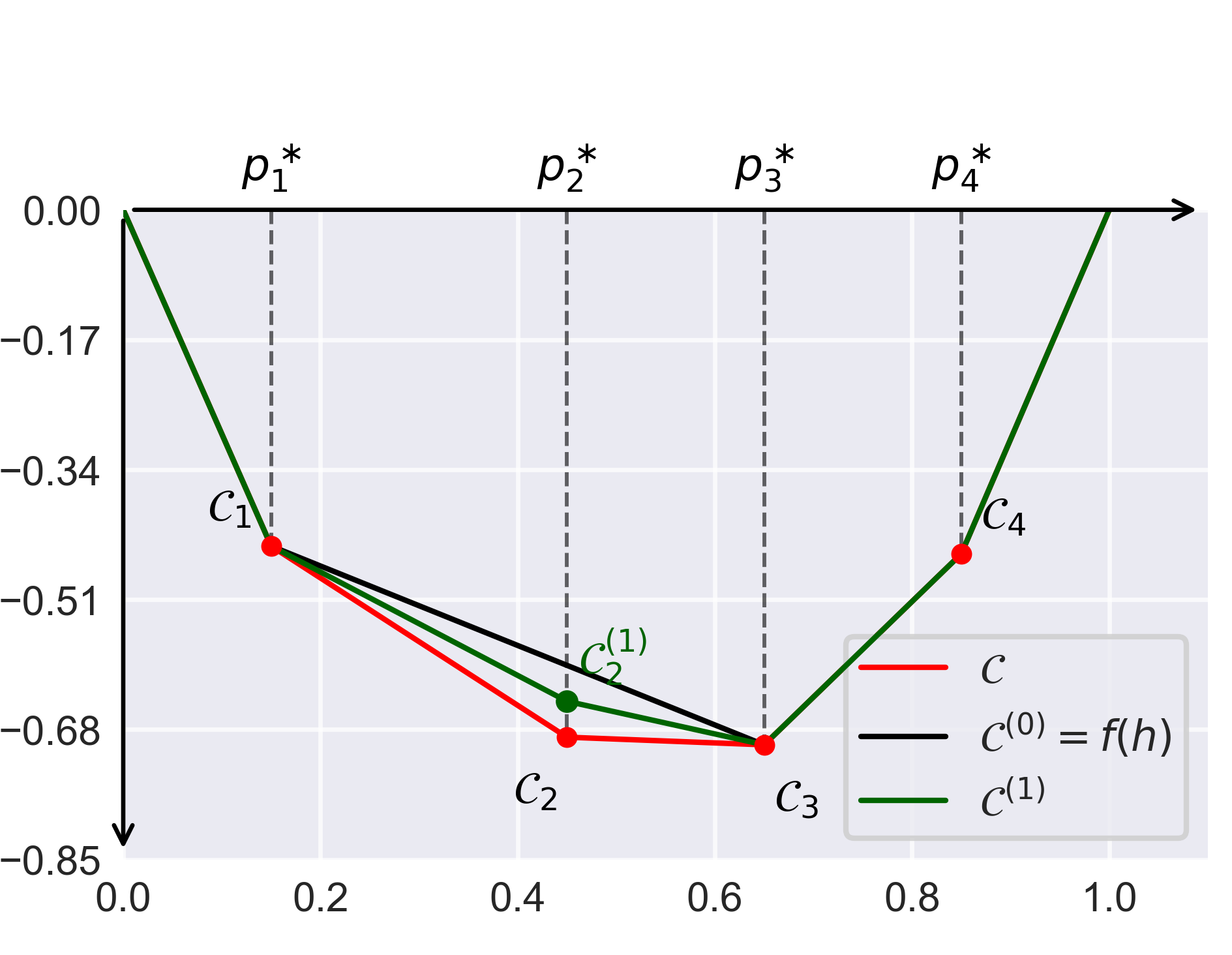}{We now target $\Ccal^{(1)}$, shown in green. Since $f_i(h)=\Ccal_i^{(0)}$, we have $|f_2(h)-\Ccal_2^{(1)}|\leq K\delta$.}

\R{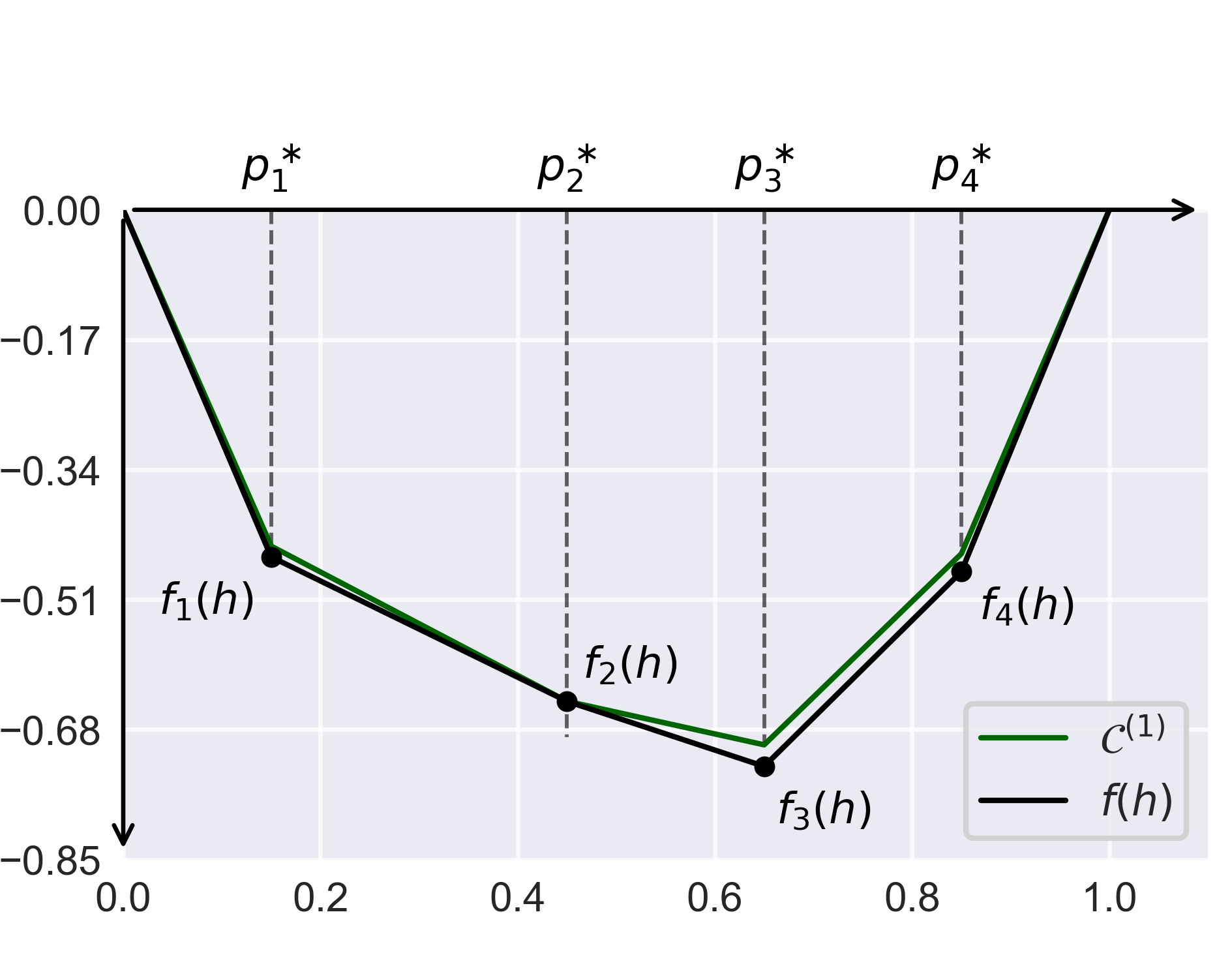}{Increase $h_2$ until $f_2(h)=\Ccal_2^{(1)}$. By Prop. \ref{prop:non-expans}, $|f_i(h)-\Ccal_i^{(1)}|\leq LK\delta$ for every $i$.}
{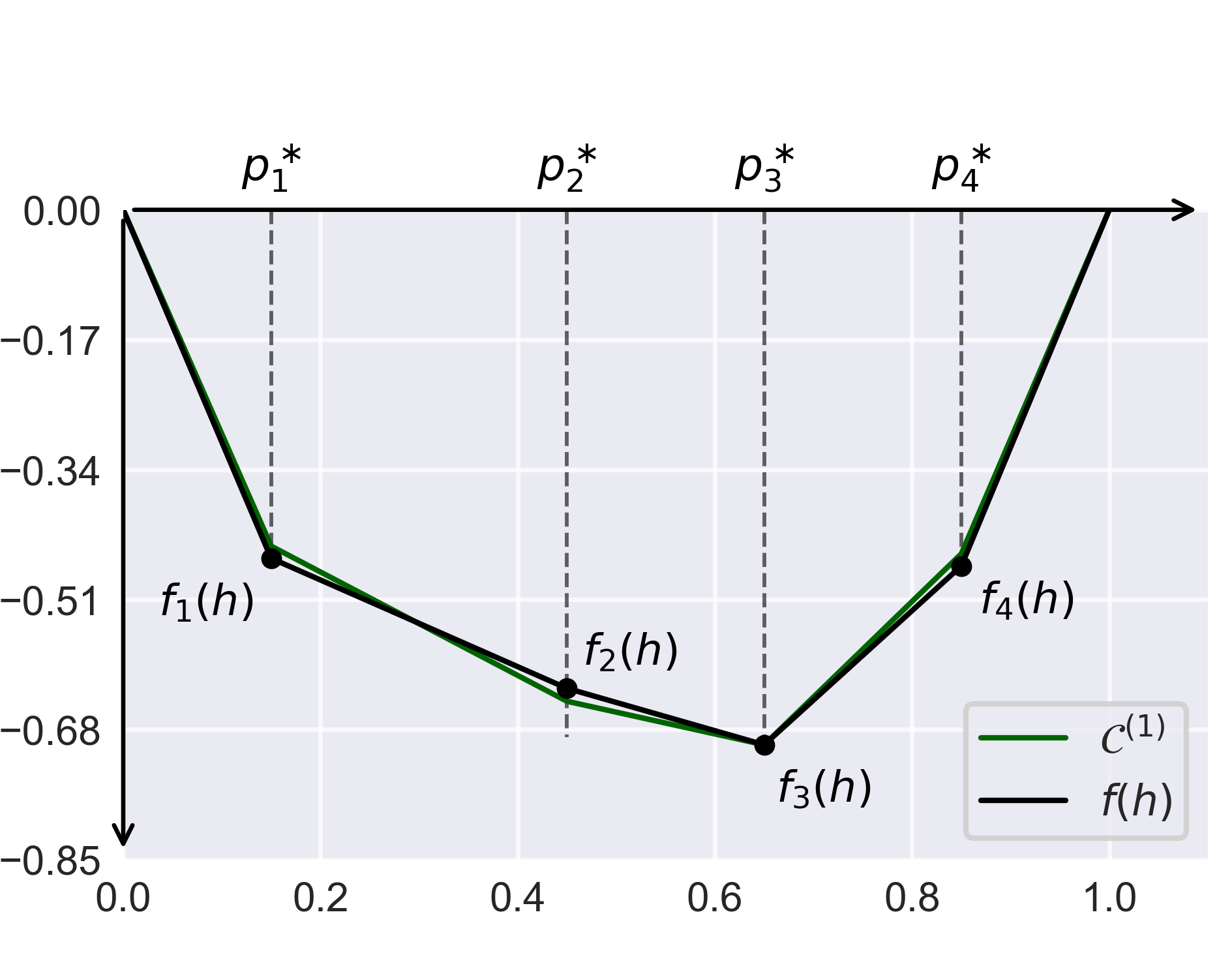}{Decrease $h_3$ until $f_3(h)=\Ccal_3^{(1)}$. Then $|f_2(h)-\Ccal_2^{(1)}|\leq L^2K\delta$, while $|f_i(h)-\Ccal_i^{(1)}|\leq (L+L^2)K\delta$ for the other components.}
{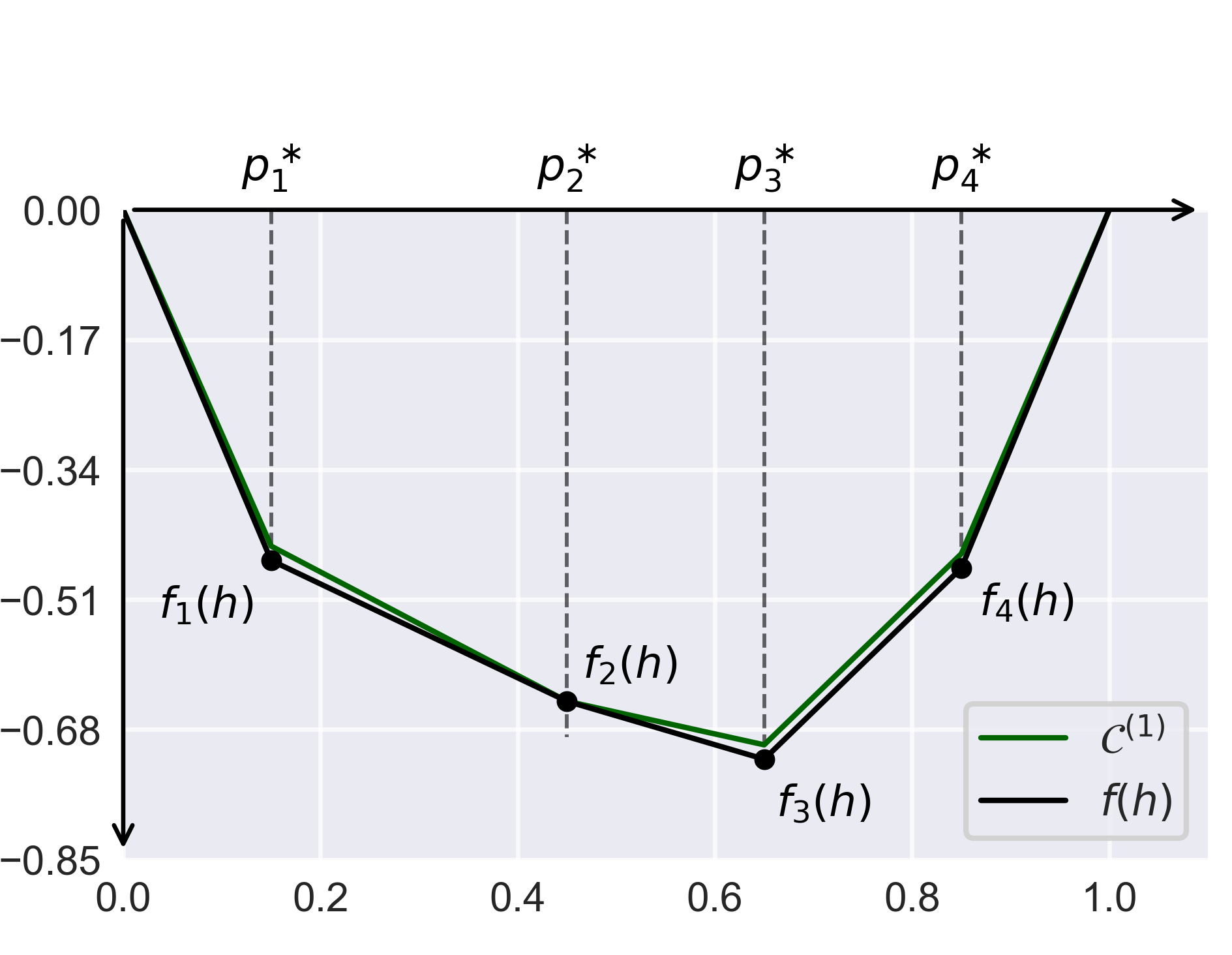}{Increase $h_2$ again until $f_2(h)=\Ccal_2^{(1)}$. Then $|f_3(h)-\Ccal_3^{(1)}|\leq L^3K\delta$, and the remaining errors are bounded by $(L+L^2+L^3)K\delta$.}

\R{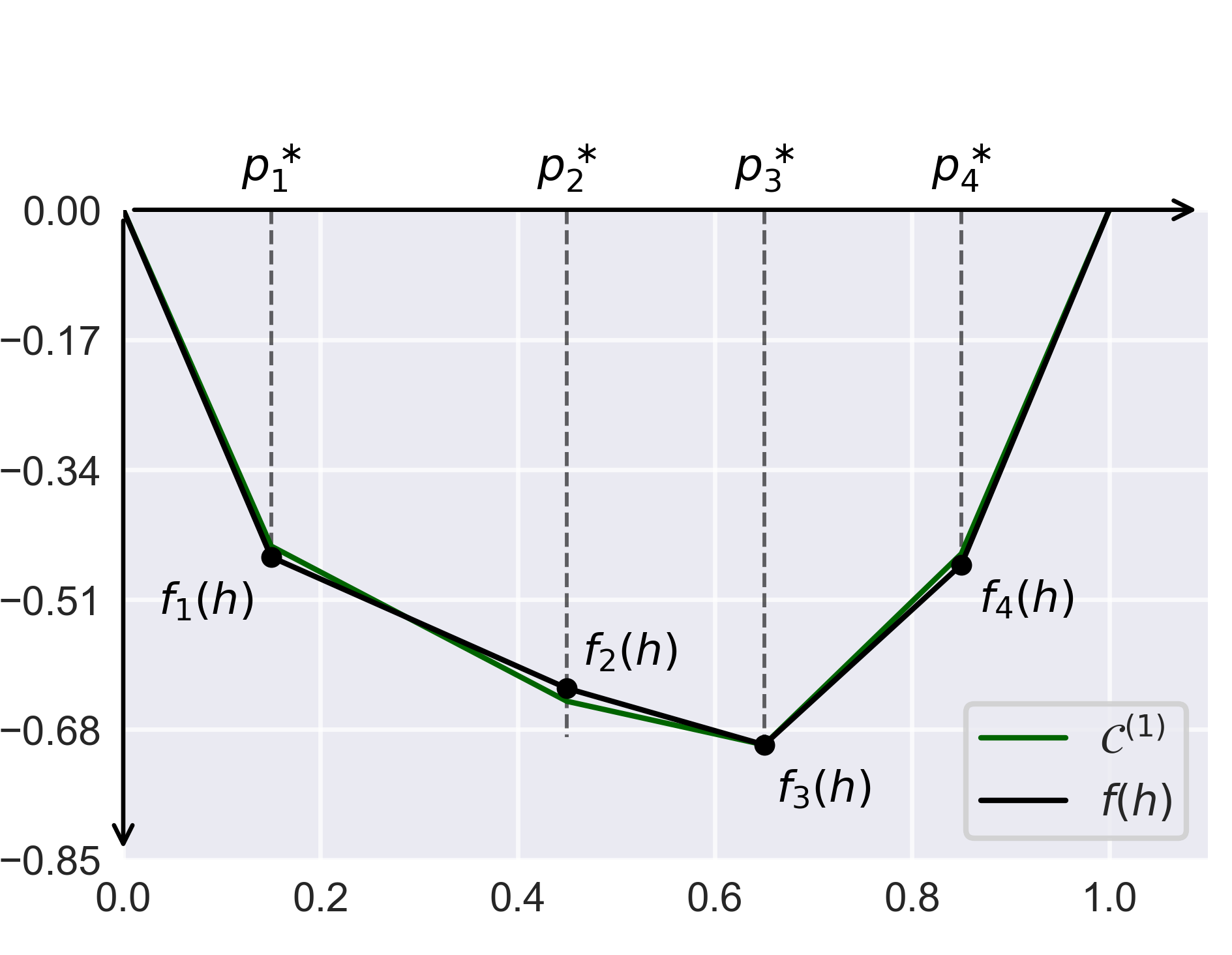}{Decrease $h_3$ again until $f_3(h)=\Ccal_3^{(1)}$. Then $|f_2(h)-\Ccal_2^{(1)}|\leq L^4K\delta$, and the remaining errors are bounded by $(L+L^2+L^3+L^4)K\delta$.}
{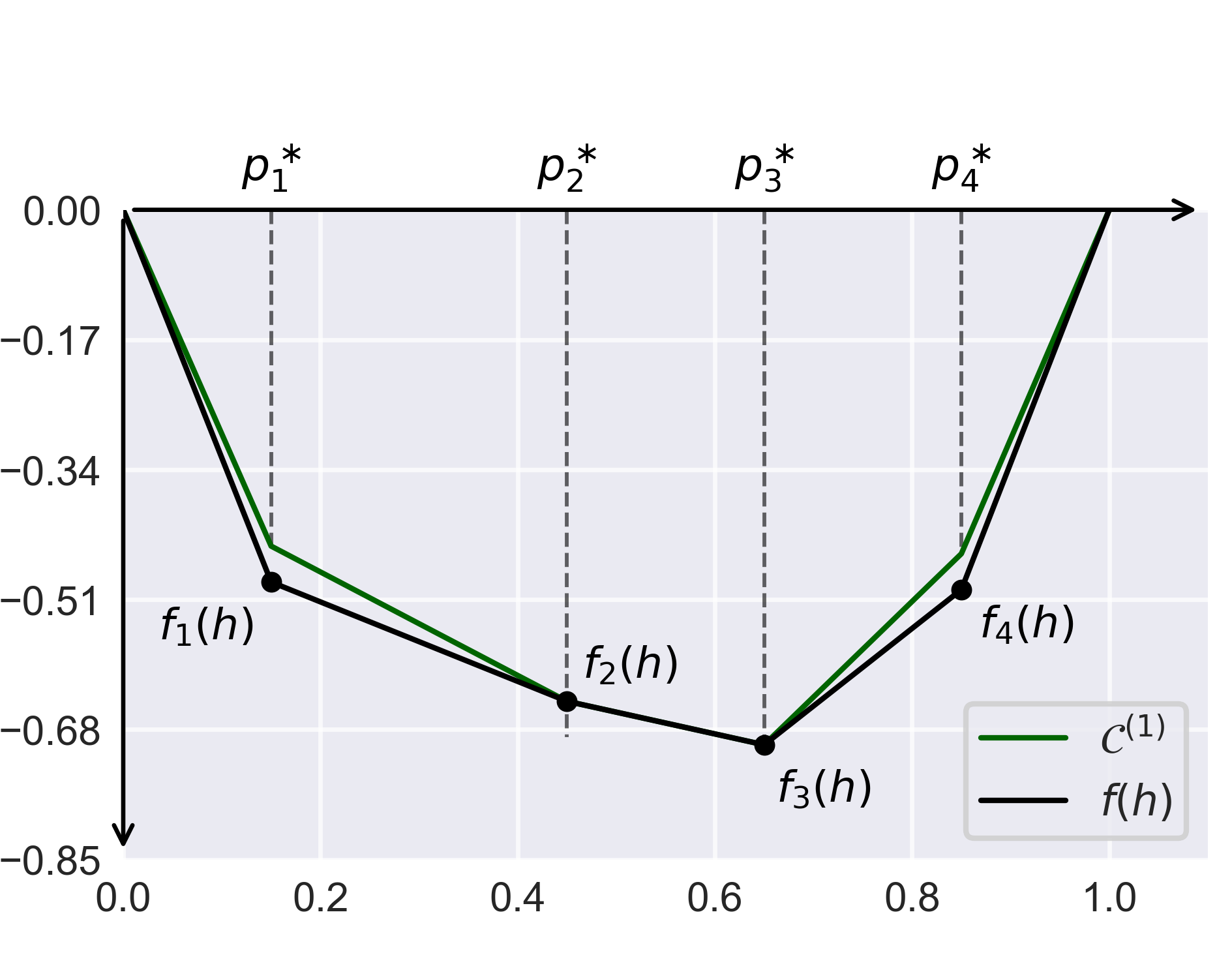}{Continuing this alternating correction and passing to the limit, determine $h_2,h_3$, with $h_1,h_4$ fixed, so that $f_2(h)=\Ccal_2^{(1)}$ and $f_3(h)=\Ccal_3^{(1)}$. By Prop. \ref{prop:non-expans}, with \textbf{(b)} as reference, $|f_i(h)-\Ccal_i^{(1)}|\leq LK\delta$ for all $i$.}
{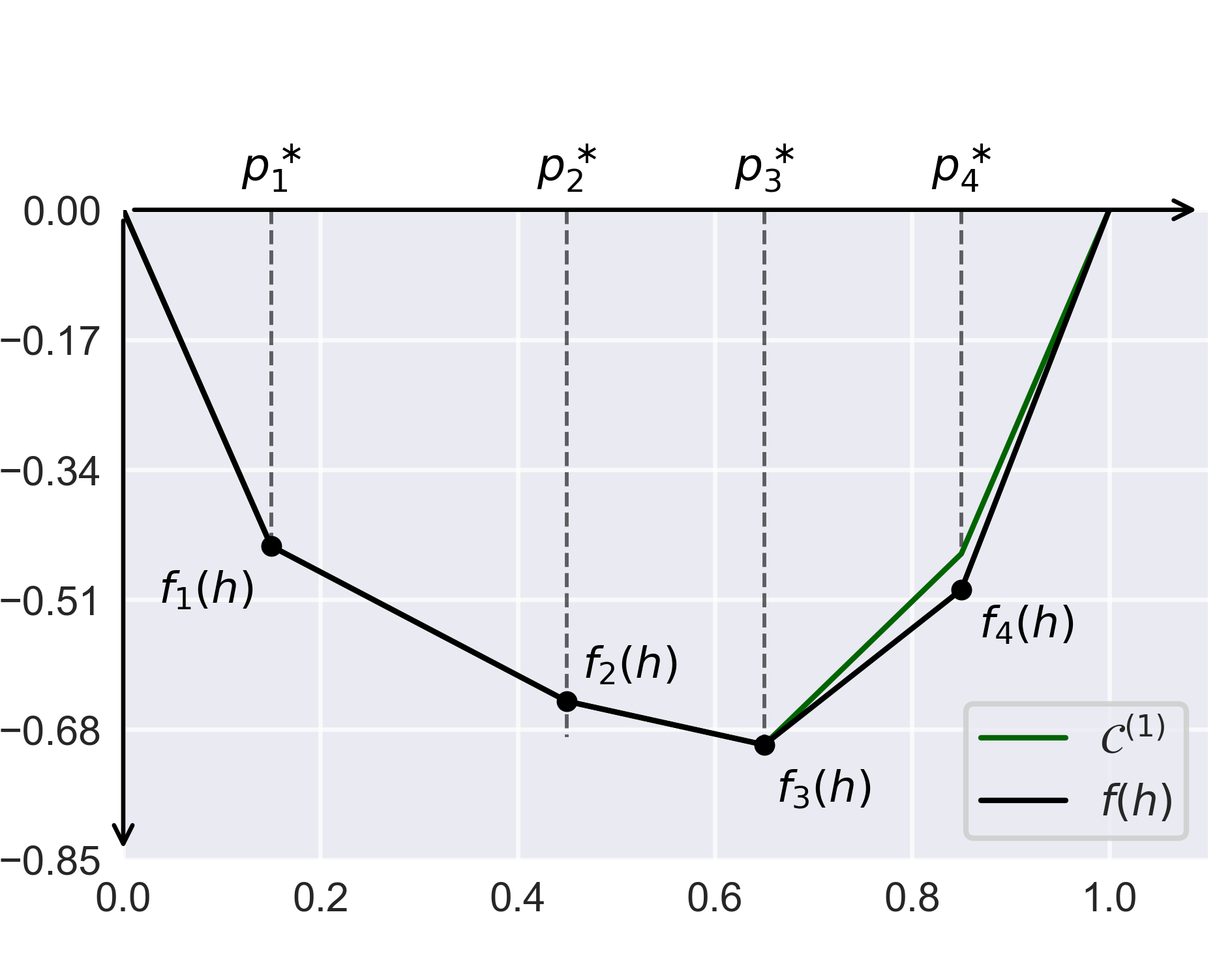}{We now want to match the set $\{1,2,3\}$, as shown above. First alternate $h_1,h_3$ to match components $1$ and $3$; then alternate $h_2,h_3$ to match components $2$ and $3$.}

\R{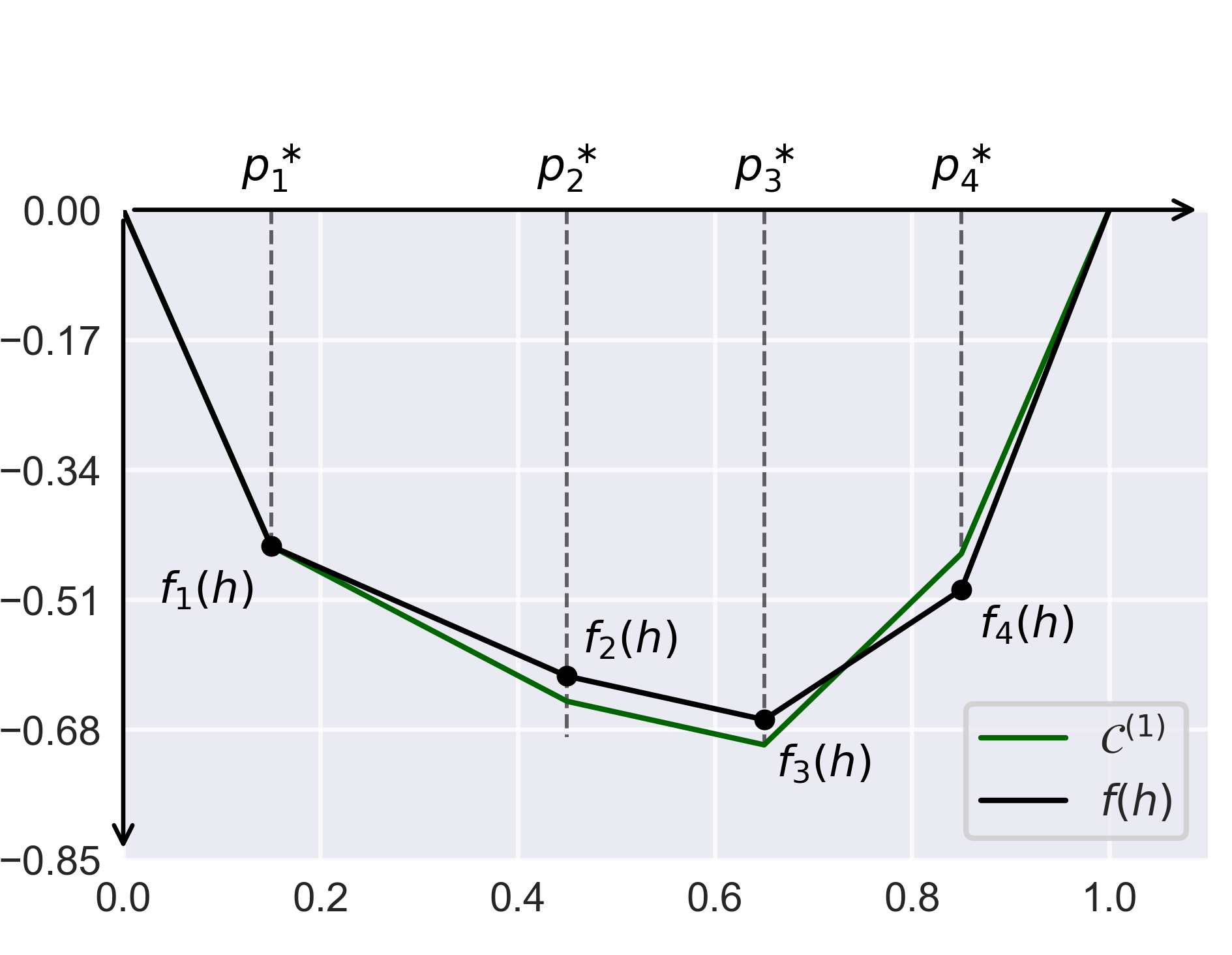}{Decrease $h_1$ until $f_1(h)=\Ccal_1^{(1)}$. Then $|f_i(h)-\Ccal_i^{(1)}|\leq L^2K\delta$ for $i=2,3$, and $|f_4(h)-\Ccal_4^{(1)}|\leq (L+L^2)K\delta$.}
{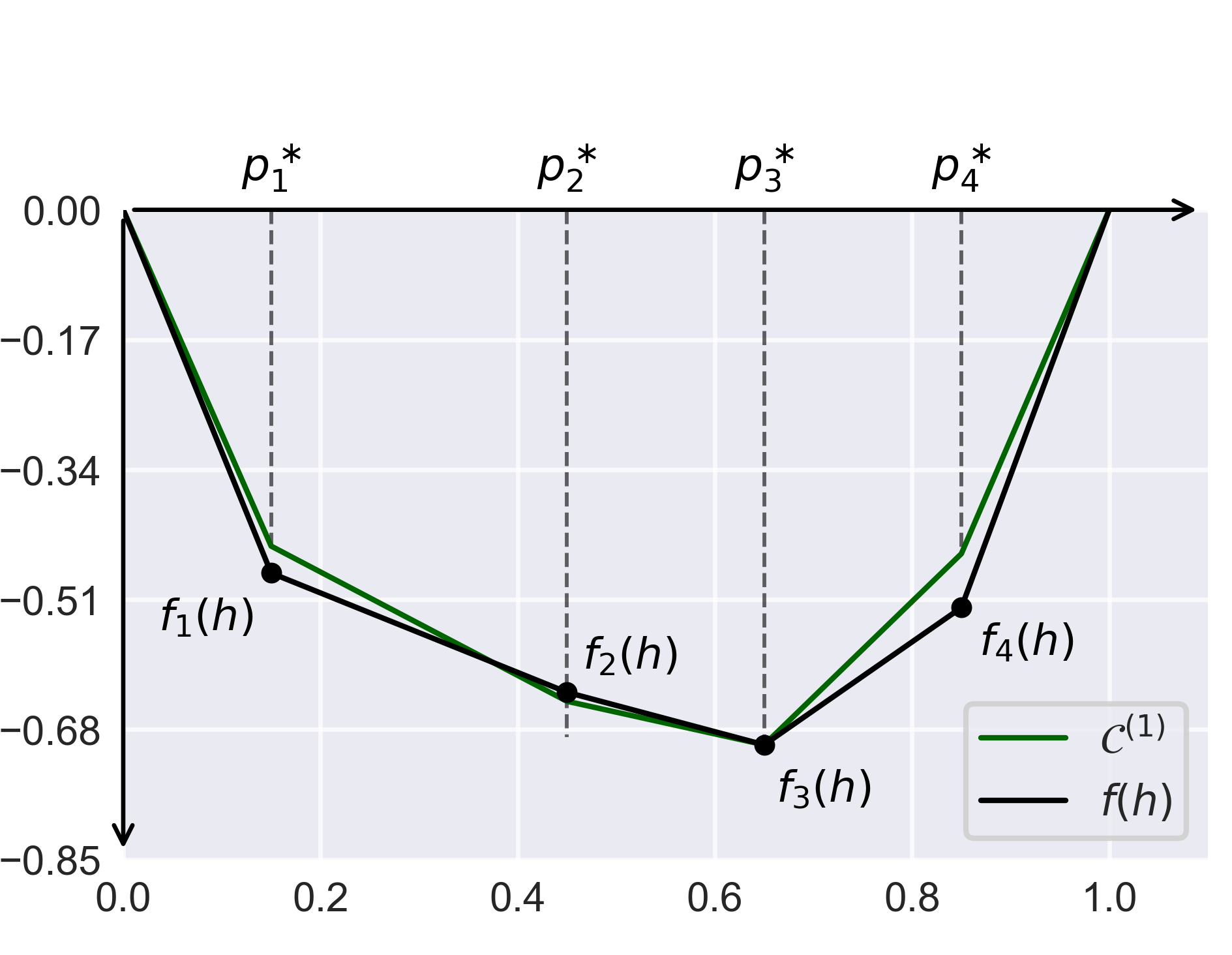}{Increase $h_3$ until $f_3(h)=\Ccal_3^{(1)}$. Then $|f_1(h)-\Ccal_1^{(1)}|\leq L^3K\delta$, $|f_2(h)-\Ccal_2^{(1)}|\leq (L^2+L^3)K\delta$, and $|f_4(h)-\Ccal_4^{(1)}|\leq (L+L^2+L^3)K\delta$.}
{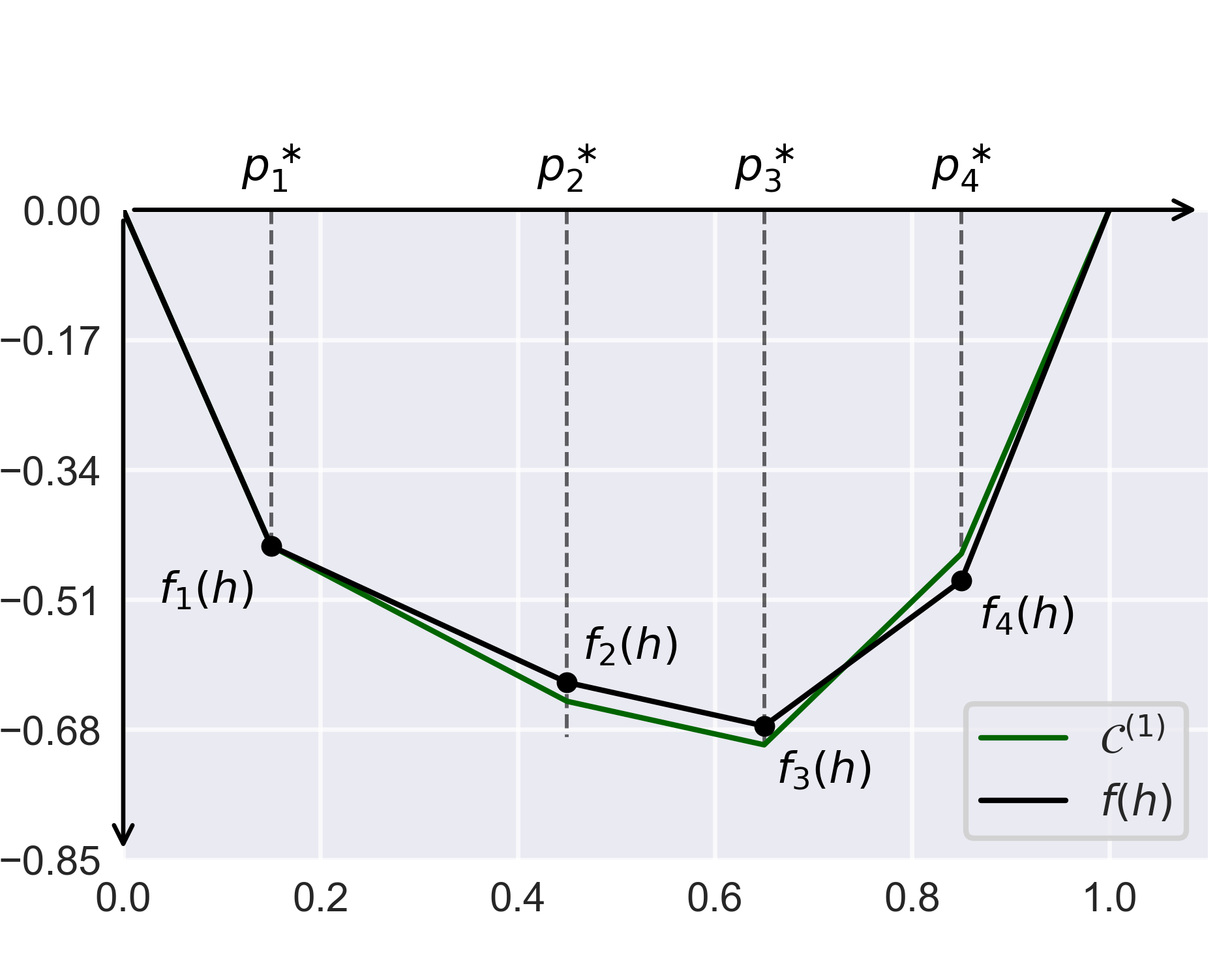}{Decrease $h_1$ again until $f_1(h)=\Ccal_1^{(1)}$. Then $|f_3(h)-\Ccal_3^{(1)}|\leq L^4K\delta$, while the errors in components $2$ and $4$ are bounded by $(L^2+L^3+L^4)K\delta$ and $(L+L^2+L^3+L^4)K\delta$, resp.}

\R{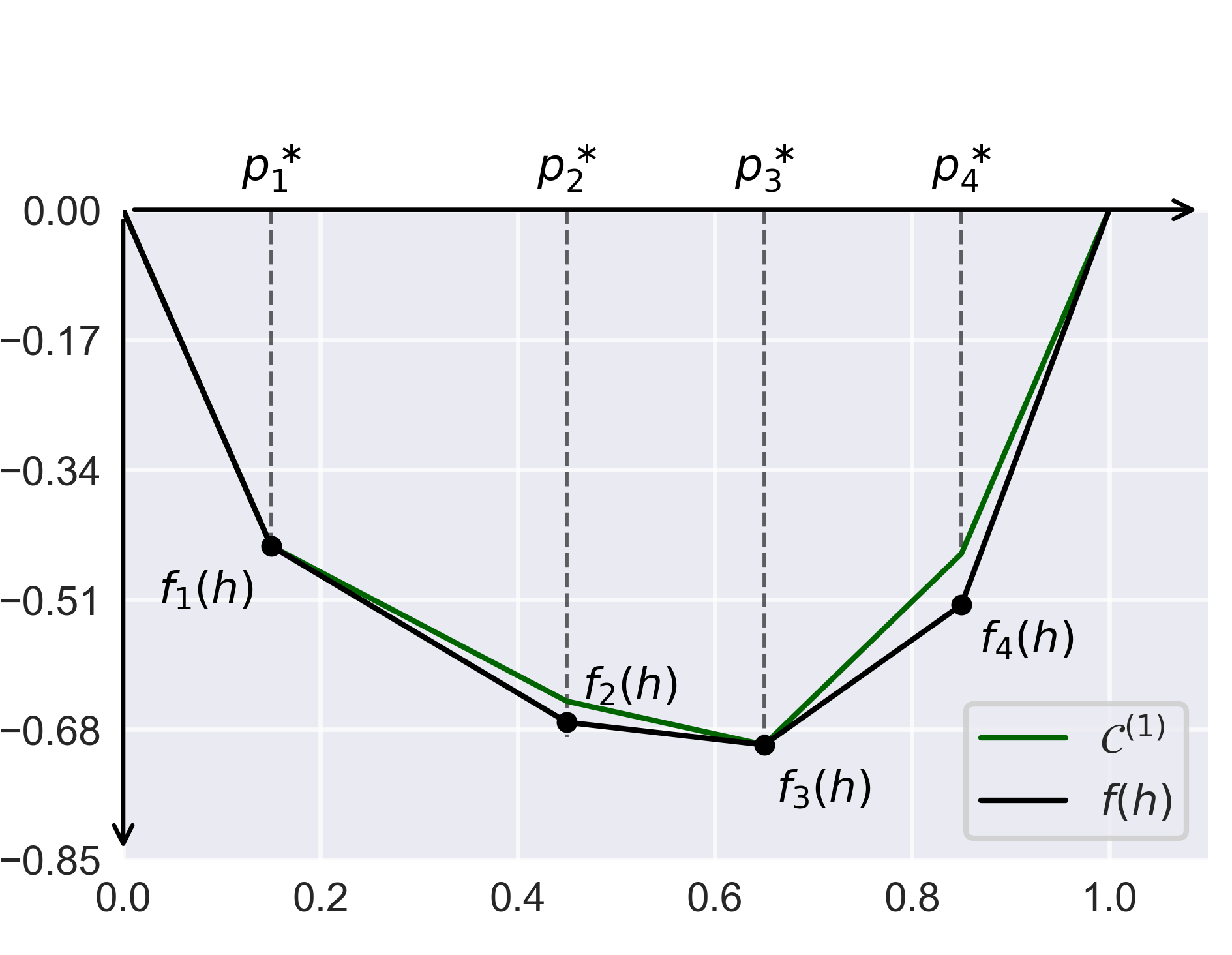}{Iterating the previous two corrections gives $f_i(h)=\Ccal_i^{(1)}$ for $i=1,3$. By Prop. \ref{prop:non-expans}, with \textbf{(h)} as reference, $|f_2(h)-\Ccal_2^{(1)}|\leq L^2K\delta$ and $|f_4(h)-\Ccal_4^{(1)}|\leq (L+L^2)K\delta$.}
{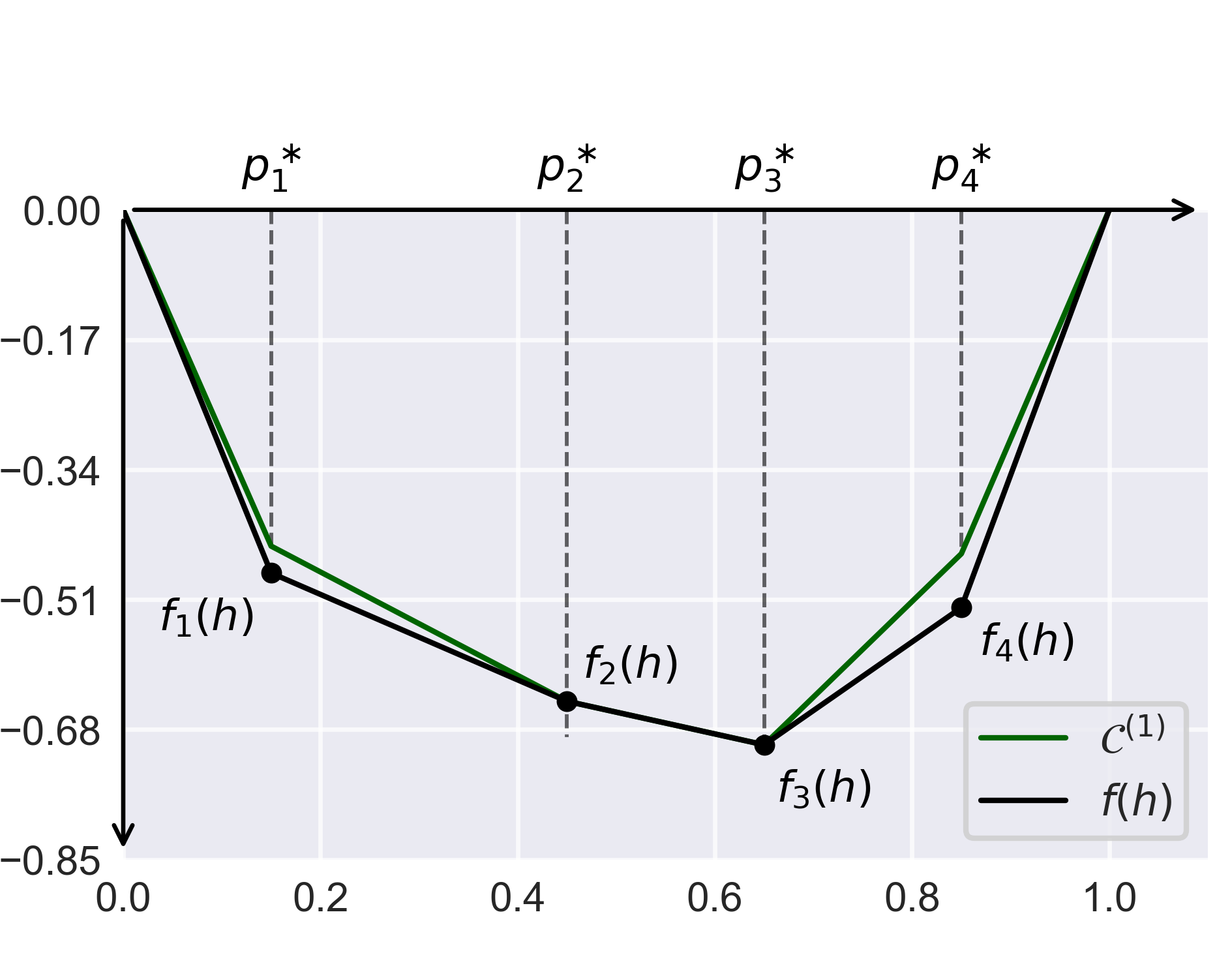}{Alternate $h_2,h_3$ and pass to the limit to obtain $f_i(h)=\Ccal_i^{(1)}$ for $i=2,3$. By Prop. \ref{prop:non-expans}, with \textbf{(m)} as reference, $|f_1(h)-\Ccal_1^{(1)}|\leq L^3K\delta$ and $|f_4(h)-\Ccal_4^{(1)}|\leq (L+L^2+L^3)K\delta$.}
{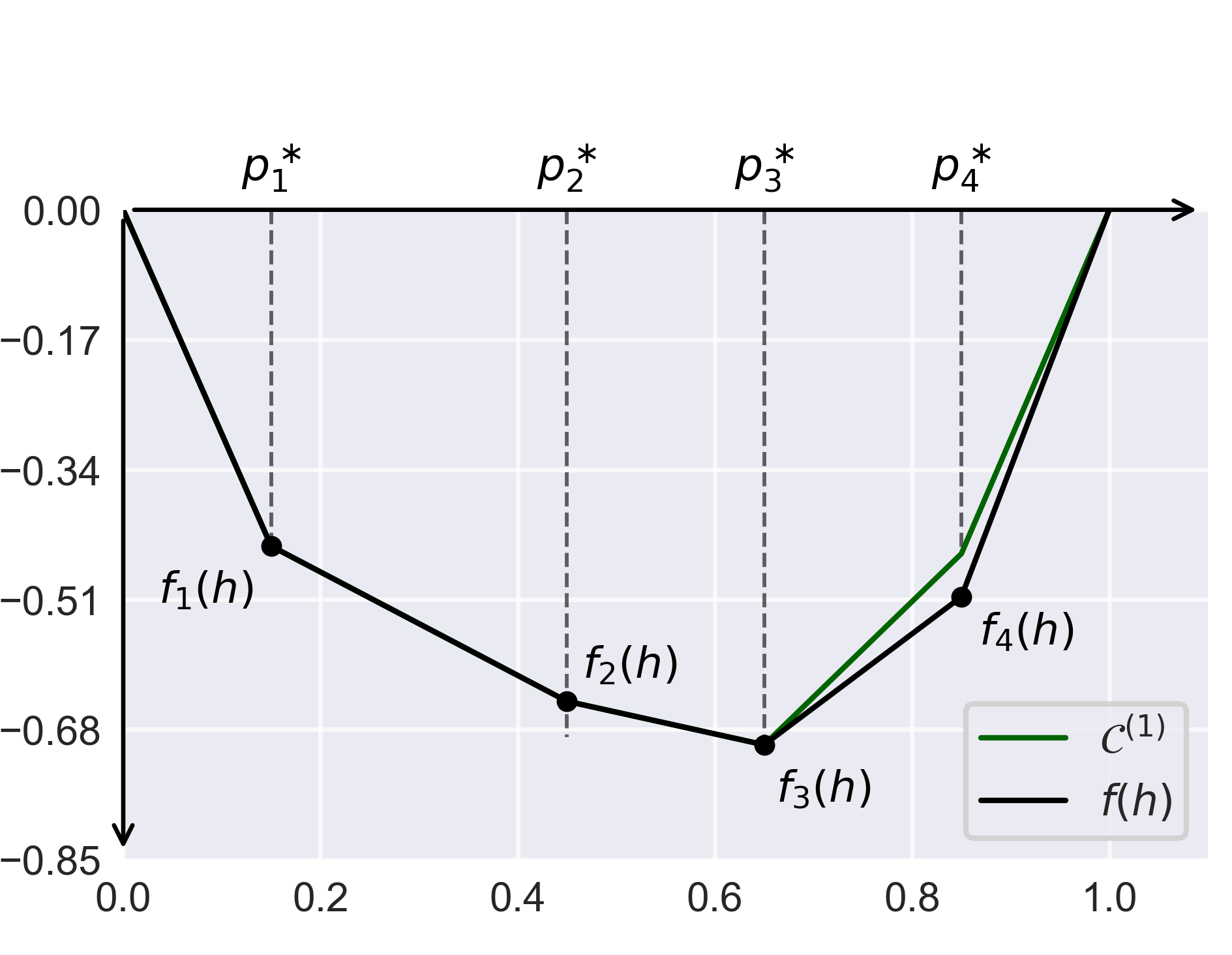}{Repeat the two alternating procedures, first for $h_1,h_3$ and then for $h_2,h_3$. This yields $f_i(h)=\Ccal_i^{(1)}$ for $i=1,2,3$. By Prop. \ref{prop:non-expans}, with \textbf{(c)} as reference, $|f_4(h)-\Ccal_4^{(1)}|\leq LK\delta$.}

\R{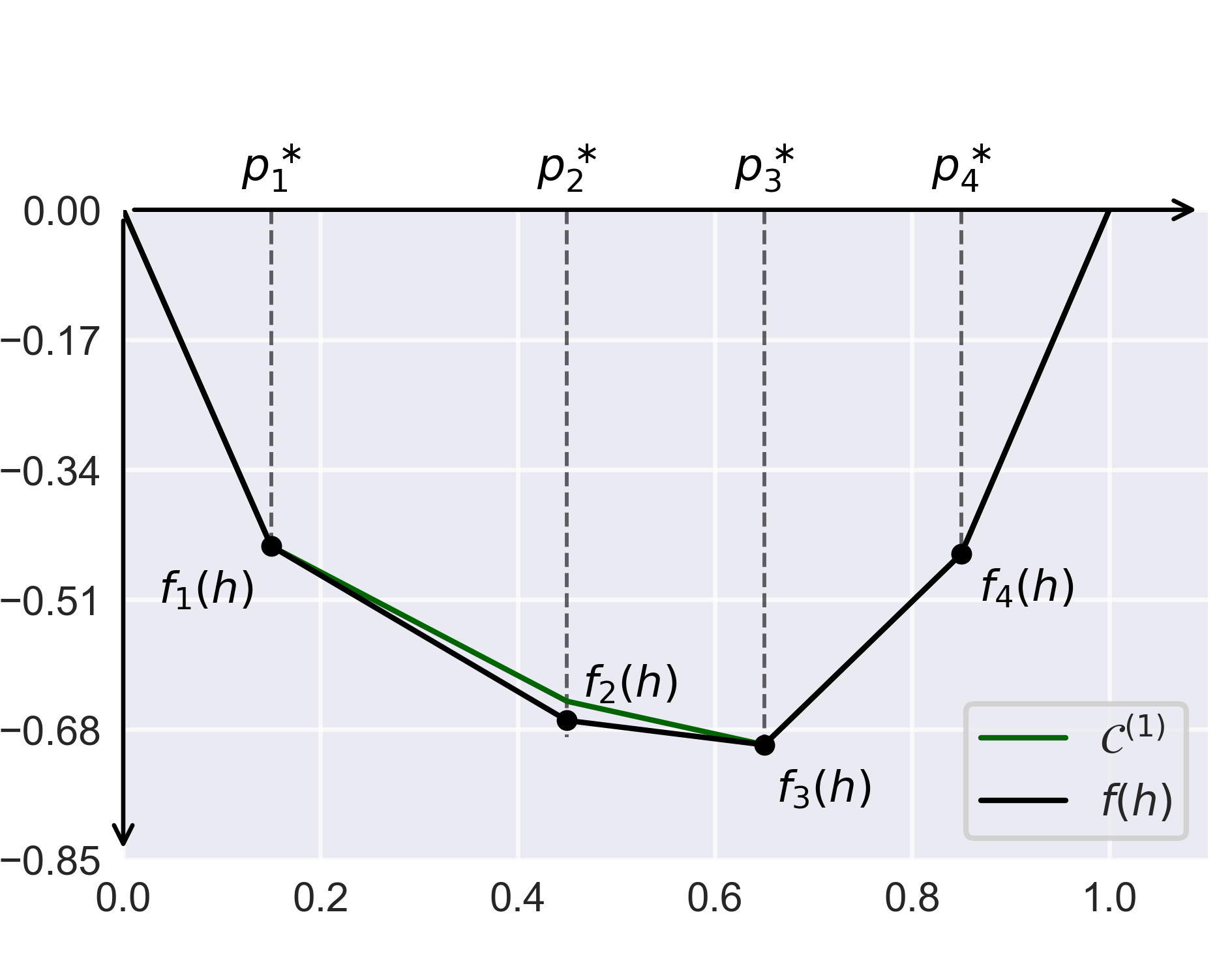}{Apply the same construction to $h_1,h_3,h_4$ to obtain $f_i(h)=\Ccal_i^{(1)}$ for $i=1,3,4$. By Prop. \ref{prop:non-expans}, with \textbf{(o)} as reference, $|f_2(h)-\Ccal_2^{(1)}|\leq L^2K\delta$.}
{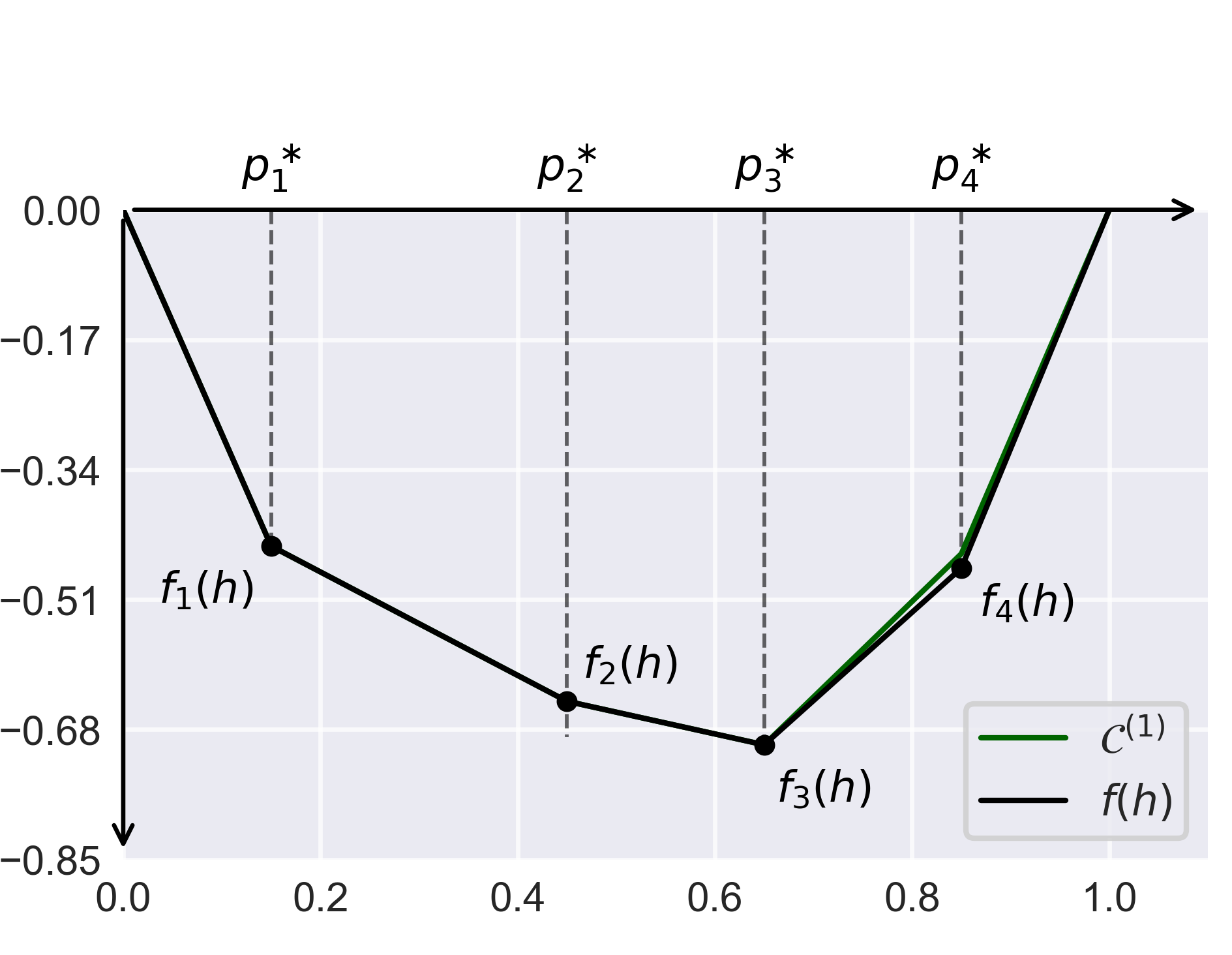}{Apply the construction to $h_1,h_2,h_3$ to obtain $f_i(h)=\Ccal_i^{(1)}$ for $i=1,2,3$. By Prop. \ref{prop:non-expans}, with \textbf{(p)} as reference, $|f_4(h)-\Ccal_4^{(1)}|\leq L^3K\delta$.}
{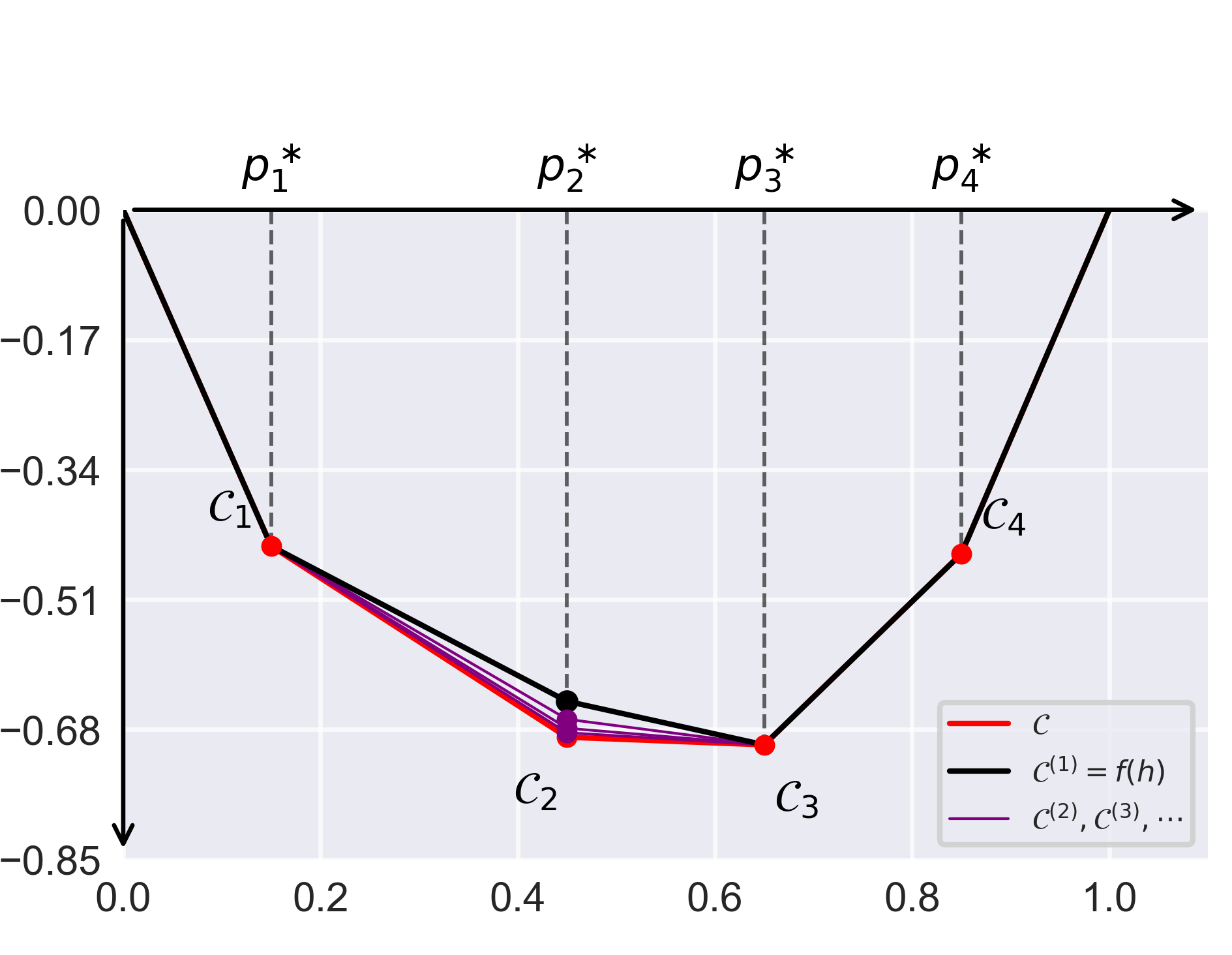}{Repeating this construction, we finally obtain $f_i(h)=\Ccal_i^{(1)}$ for every $i$. The same procedure is then applied to $\Ccal^{(2)},\Ccal^{(3)},\ldots$.}

\closestepfig

\begin{lemma}
\label{lemma:induction}
Let $\nu,q\in\Prob(\RR)$ such that Assumption~\ref{ass:technical-assumptions} holds, let $f$ be the corresponding $n$-atomic $q$-Bass map, and let
$\Ccal\in \Lambda_\nu^{p_1,\dots,p_n}$.
Let $r \in \NN$, $\emptyset \neq I \subsetneq \{1, \dots, n-1\}$, $\iadd \notin I$, and
$h \in [0, \diam(\supp(q)))^{n-1}$.
Define
\[
H = \max \left(
\max_{i \in I \cup \{\iadd\}} \widetilde f_i^{-1}(\Ccal_i),
\|h\|_\infty
\right),
\]
where $\widetilde f:\RR_{\ge 0}\to\RR^{n-1}$ is given by
\begin{equation}
\label{eq:bound-map-q-Bass}
	\widetilde f_i(h) = \int_\RR Q_\nu\!\left(p_i^* F_q(z) + (1-p_i^*)F_q(z-h)\right)\, p_i^* \rho_q(z)\,dz.
\end{equation}
Assume that there exists $L\in(0,1)$ such that
\[
\max_{\substack{
\xi \in [0,H]^{n-1}\\
\Jcal \subseteq \{1,\dots,n-1\}\\
i,j \in \Jcal,\ i\neq j
}}
\left|
\frac{\adj(Jf(\xi)_{\Jcal,\Jcal})_{ij}}{\adj(Jf(\xi)_{\Jcal,\Jcal})_{ii}}
\right|
\le L,
\qquad
K:=1-L.
\]
Let $\delta$ be the minimal diagonal gap of $\Ccal$, and let $i_\delta$ be a minimal-gap index for $\Ccal$.
Assume that $i_\delta \notin I$.
Let $\Ccal^{(r)}$ be the $L$-approximation of $\Ccal$ of order $r$ along $i_\delta$, and let
$E^{(r)}$ be the corresponding error map of order $r$.
Assume that $E_i^{(r)}(h)=0$, for all $i \in I$. If $\iadd=i_\delta$, assume moreover that
\begin{equation}
	\label{eq:tech-implications}
(i_\delta=1 \implies 2 \in I),\qquad
(i_\delta=n-1 \implies n-2 \in I),\qquad
(1<i_\delta<n-1 \implies i_\delta-1,\, i_\delta+1 \in I).
\end{equation}
Suppose that there exists $m \in \NN$ such that
\begin{equation}
\label{eq:bounds-for-inductions}
|E_{\iadd}^{(r)}(h)| \le L^m K\delta
\qquad\text{and}\qquad
\sup_{\substack{j=1,\dots,n-1 \\ j \neq \iadd}} |E_j^{(r)}(h)|
\le \sum_{k=1}^m L^k K\delta
\le L\delta.
\end{equation}
Then there exists $\widehat h \in [0,H]^{n-1}$ such that
\begin{enumerate}[label=(\roman*)]
    \item \label{it:induction_1} $\widehat h_i = h_i$ for all $i \notin I \cup \{\iadd\}$, and
    $E_i^{(r)}(\widehat h)=0$ for all $i \in I \cup \{\iadd\}$;
    \item \label{it:induction_2} the following bound holds:
    \[
    \sup_{j=1,\dots,n-1} |E_j^{(r)}(\widehat h)|
    \le \sum_{k=1}^{m+1} L^k K\delta
    \le L\delta.
    \]
\end{enumerate}
\end{lemma}

\begin{remark}
    The constant $H$ is chosen so that the entire construction takes place inside the compact cube $[0,H]^{n-1}$. This is guaranteed by implication \eqref{eq:bound-coordinate} below, which applies to every vector produced by the lemma. Hence all relevant points remain in $[0,H]^{n-1}$, where the ratios of adjugate minors are uniformly bounded by the constant $L$. This uniform bound is what allows Proposition~\ref{prop:non-expans} to be applied throughout the construction.
\end{remark}

\begin{remark}
The set $I$ in Lemma~\ref{lemma:induction} is the set of indices which are already matched: $E_i^{(r)}(h)=0$ for $i\in I$, or equivalently, $f_i(h)=\Ccal_i^{(r)}$. Lemma~\ref{lemma:induction} adds one new index $\iadd\notin I$ where the error map equals zero, producing $\widehat h$ such that $E_i^{(r)}(\widehat h)=0$ for every $i\in I\cup\{\iadd\}$, while the coordinates of $f$ outside $I\cup\{\iadd\}$ are unchanged.

Condition \eqref{eq:tech-implications} is needed only in the case $\iadd=i_\delta$. In this case, the new index to be matched is precisely the index along which the approximating polygonal chains $\Ccal^{(r)}$ differ from $\Ccal$. Therefore, before correcting the component $i_\delta$, we require the neighbouring components to have already been matched. This point is slightly delicate because, for a fixed approximation $\Ccal^{(r)}$, the index $i_\delta$, which is chosen from the original chain $\Ccal$, need not be a minimal-gap index for $\Ccal^{(r)}$. Thus condition \eqref{eq:tech-implications} is a technical requirement that allows the construction in Lemma~\ref{lemma:induction} to handle the case $\iadd=i_\delta$ safely. If $\iadd\neq i_\delta$, no additional ordering condition is needed.
\end{remark}
\medskip

\begin{proof}[Proof of Lemma \ref{lemma:induction}]
We first present some preliminary arguments that will be used repeatedly in the proof. Fix $h'\in (0,\diam(\supp(q)))^{\,n-1}$. Observe that, for any $i\in I\cup\{\iadd\}$, the implication
\begin{equation}
\label{eq:bound-coordinate}
	f_i(h')\ge \Ccal_i \quad\Longrightarrow\quad h_i'\le H
\end{equation}
holds. Indeed, by monotonicity of $f_i$ with respect to the variables $\{h_j'\}_{j\neq i}$ (see Remark~\ref{rmk:one-coordinate-q-Bass-map}),
\[
\Ccal_i \le f_i(h')\le f_i(h_i'e_i)=\widetilde f_i(h_i'),
\]
where $\widetilde f_i$ is the is the one--dimensional map defined in \eqref{eq:bound-map-q-Bass}, which  is precisely the  $2$-atomic $q$-Bass map associated with the two weights $p_i^*$ and $1-p_i^*$ with respect to $\nu$. In particular, $\widetilde f_i$ is decreasing by Remark~\ref{rmk:one-coordinate-q-Bass-map} and hence admits an inverse $\widetilde f_i^{-1}$. Therefore,
\[
h_i' \le \widetilde f_i^{-1}(\Ccal_i) \le H.
\]

\medskip
\noindent$\bullet$ \emph{A one--coordinate adjustment.}
Assume there exists $\widetilde h\in[0,H]^{n-1}$ such that $\sup_j |E_j^{(r)}(\widetilde h)|<\delta$, and let $i\in\{1,\dots,n-1\}\setminus\{i_\delta\}$.
By Remark~\ref{rmk:one-coordinate-q-Bass-map},
\[
E_i^{(r)}(\widetilde h_1,\dots,\widetilde h_{i-1},0,\widetilde h_{i+1},\dots,\widetilde h_{n-1})\ge 0.
\]
Indeed,
\[
f_j(\widetilde h_1,\dots,\widetilde h_{i-1},0,\widetilde h_{i+1},\dots,\widetilde h_{n-1})
\ge f_j(\widetilde h),
\quad\text{for all } j\in\{1,\dots,n-1\}.
\]
Moreover, by Proposition~\ref{prop:mainThm-n_atoms-part1}, the point
\[
\bigl(p_i^*,\,f_i(\widetilde h_1,\dots,\widetilde h_{i-1},0,\widetilde h_{i+1},\dots,\widetilde h_{n-1})\bigr)
\]
lies on the segment joining the points $\bigl(p_j^*,\,f_j(\widetilde h_1,\dots,\widetilde h_{i-1},0,\widetilde h_{i+1},\dots,\widetilde h_{n-1})\bigr)$ with $j=i-1,i+1$. Hence, Definition \ref{def:minimal-diagonal-gap} of the minimal gap yields (see Figure \ref{fig:proof-1})
\begin{equation}\label{eq:one-coorinate-1}
    f_i(\widetilde h_1,\dots,\widetilde h_{i-1},0,\widetilde h_{i+1},\dots,\widetilde h_{n-1}) \ge \Ccal_i.
\end{equation}
\vspace{-5mm}
\begin{figure}[H]
     \centering
          \begin{subfigure}{0.65\textwidth}
         \centering
         \includegraphics[width=\textwidth]{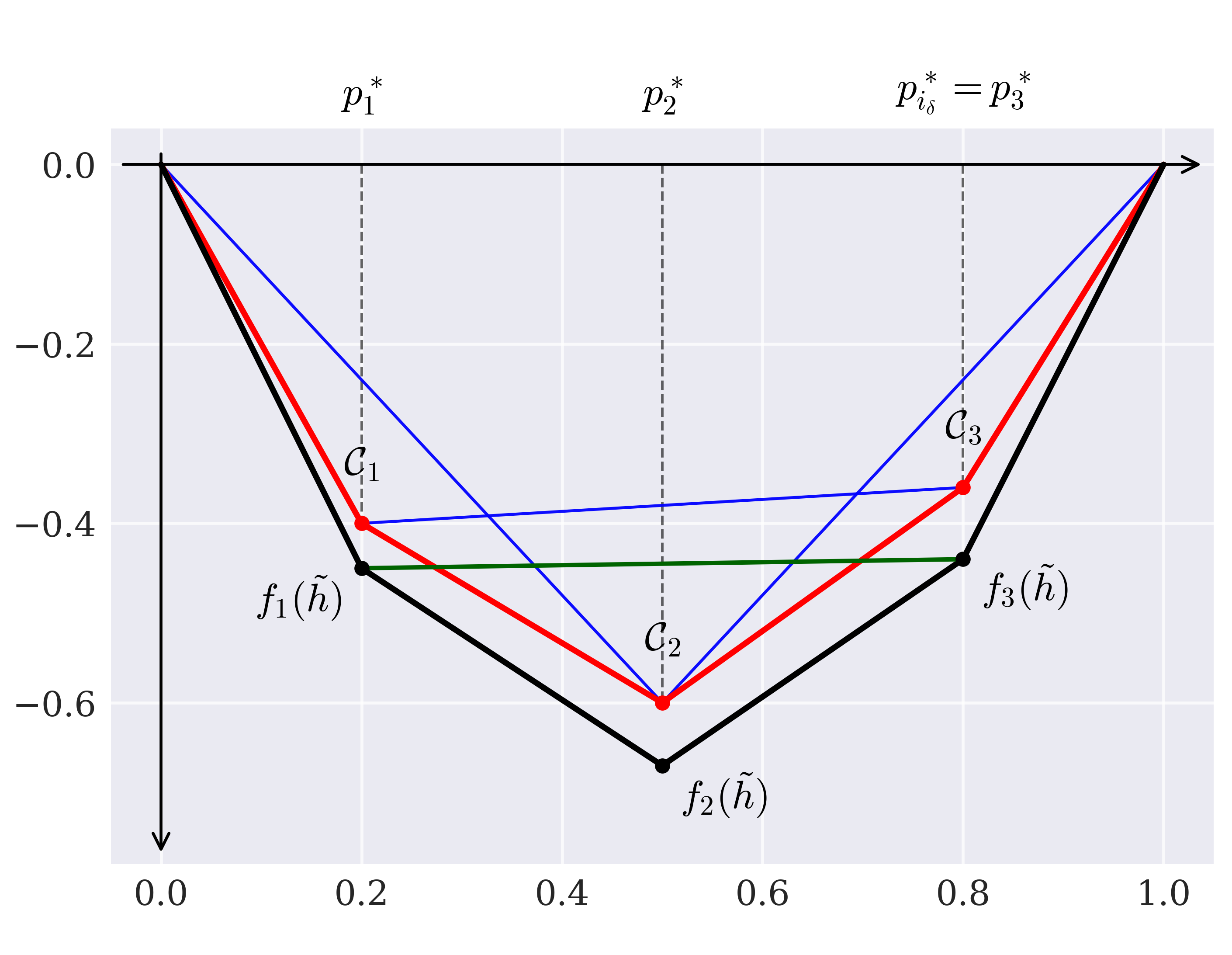}
     \end{subfigure}
	\caption{Example with $i=2$. Definition \ref{def:minimal-diagonal-gap} gives $\delta_2 < \delta$. Therefore, by Remark~\ref{rmk:one-coordinate-q-Bass-map} and Proposition~\ref{prop:mainThm-n_atoms-part1},
$f_2(\widetilde h_1,0,\widetilde h_3) \ge \Ccal_2$.}
     		\label{fig:proof-1}
\end{figure}
On the other hand, Proposition~\ref{prop:mainThm-n_atoms-part1} and Definition \ref{def:lambda_nu} imply
\[
\lim_{t\to\infty}E_i^{(r)}(\widetilde h_1,\dots,\widetilde h_{i-1},t,\widetilde h_{i+1},\dots,\widetilde h_{n-1})
=U_\nu(p_i^*)-\Ccal_i^{(r)}
\le U_\nu(p_i^*)-\Ccal_i
\le 0.
\]
Therefore, by continuity, there exists $t_i\geq 0$ such that replacing the $i$-th component of $\widetilde h$ by $t_i$ yields a vector satisfying $E_i^{(r)}=0$ and the bound $t_i\le H$ follows from \eqref{eq:bound-coordinate}. 

We shall use the above argument repeatedly. In the case $i=i_\delta$, the same conclusion holds under the additional assumption that $E_j^{(r)}(\widetilde h)=0$ for $j=i_\delta \pm 1$. This extra assumption is needed precisely at the step where the minimal diagonal gap is used. Indeed, for $i\neq i_\delta$, the argument above gives \eqref{eq:one-coorinate-1}, as illustrated in Figure~\ref{fig:proof-1}. If $i=i_\delta$, this implication is no longer automatic, because the vertex $\Ccal_{i_\delta}$ has been replaced by $\Ccal_{i_\delta}^{(r)}$. Indeed, since $\Ccal_{i_\delta}$ has been replaced by $\Ccal_{i_\delta}^{(r)}$, the diagonal gap corresponding to the index $i_\delta$ for the perturbed chain $\Ccal^{(r)}$ need not be equal to the minimal diagonal gap $\delta$ of $\Ccal$, and may be smaller.  By requiring the adjacent errors to vanish, the neighboring vertices are already fixed at their prescribed values for $\Ccal^{(r)}$, which allows us to conclude \eqref{eq:one-coorinate-1} also in this case.

\medskip

We argue by induction on $s:=|I|$. The smallest possible value of $s$ depends on the relation between $\iadd$ and $i_\delta$. If $\iadd \neq i_\delta$, then the minimum possible cardinality of $I$ is $1$. The same holds when $\iadd=i_\delta=1$ or $\iadd=i_\delta=n-1$, since \eqref{eq:tech-implications} implies that $2\in I$ in the first case and $n-2\in I$ in the second. Thus, the base case is $s=1$ whenever $\iadd \neq i_\delta$ or $\iadd=i_\delta\in\{1,n-1\}$, with the unique element of $I$ determined by \eqref{eq:tech-implications}. If instead $1<\iadd=i_\delta<n-1$, then \eqref{eq:tech-implications} implies that $\{i_\delta-1,i_\delta+1\}\subseteq I$. Consequently, the smallest possible value of $s$ is $2$, and this case has to be treated separately.

\medskip
\noindent$\bullet$ \emph{Base case $s=1$ $\mathrm (\iadd \not = i_\delta$ or $\iadd=i_\delta\in\{1,n-1\}\mathrm)$.}
Let $I=\{\widetilde i\}$. Starting from $h^{(0)}:=h$, we first apply the above one-coordinate adjustment with $\widetilde h=h^{(0)}$ and $i=\iadd$. In this way, we obtain a vector $h^{(1)}$, defined from $h^{(0)}$ by changing only the $\iadd$-th component so that $E_{\iadd}^{(r)}(h^{(1)})=0$. By Proposition~\ref{prop:non-expans}, applied with $I_{\mathrm{dom}}=\{\iadd\}$ and $I_{\mathrm{cdom}}=\emptyset$,
\[
    |f_i(h^{(1)})-f_i(h^{(0)})| \le L|f_{\iadd}(h^{(1)})-f_{\iadd}(h^{(0)})| = L |E_{\iadd}^{(r)}(h^{(0)})|, \quad \text{for all } i\neq\iadd.
\]
Let $m\in\NN$ be such that \eqref{eq:bounds-for-inductions} holds. Since $|E_{\iadd}^{(r)}(h^{(0)})|\le L^mK\delta$, the previous estimate gives $|E_{\widetilde i}^{(r)}(h^{(1)})|=|f_{\widetilde i}(h^{(1)})-\Ccal_{\widetilde i}^{(r)}|=|f_{\widetilde i}(h^{(1)})-f_{\widetilde i}(h^{(0)})|\le L^{m+1}K\delta$, where we used $E_{\widetilde i}^{(r)}(h^{(0)})=0$. Moreover, for every $j\neq\iadd$,
\[
    |E_j^{(r)}(h^{(1)})|
    = |f_j(h^{(1)})-\Ccal_j^{(r)}|
    \le |f_j(h^{(0)})-\Ccal_j^{(r)}|
    + |f_j(h^{(1)})-f_j(h^{(0)})|
    \le |E_j^{(r)}(h^{(0)})| + L |E_{\iadd}^{(r)}(h^{(0)})|.
\]
Hence, by \eqref{eq:bounds-for-inductions},
\[
    \sup_{\substack{j=1,\dots,n-1\\ j\neq\iadd}}
    |E_j^{(r)}(h^{(1)})|
    \le \sum_{k=1}^m L^kK\delta+L^{m+1}K\delta
    = \sum_{k=1}^{m+1}L^kK\delta
    <\delta.
\]
Together with $E_{\iadd}^{(r)}(h^{(1)})=0$, this shows that $h^{(1)}$ satisfies \eqref{eq:bounds-for-inductions} with $m+1$ in place of $m$.

We can therefore apply the one-coordinate adjustment once more, now with $\widetilde h=h^{(1)}$ and $i=\widetilde i$. Since $i_\delta\notin I$ by assumption, we have $\widetilde i\neq i_\delta$. Hence the bound provided by \eqref{eq:bounds-for-inductions} is sufficient to perform the one-coordinate adjustment. In this way, we obtain a vector $h^{(2)}$, defined from $h^{(1)}$ by changing only the $\widetilde i$-th component so that $E_{\widetilde i}^{(r)}(h^{(2)})=0$. Applying Proposition~\ref{prop:non-expans} again gives
\[
    |E_{\iadd}^{(r)}(h^{(2)})|\le L^{m+2}K\delta,
    \qquad
    \sup_j |E_j^{(r)}(h^{(2)})|
    \le \sum_{k=1}^{m+2}L^kK\delta
    <\delta.
\]

Iterating (alternating the correction of the $\iadd$-th and $\widetilde i$-th components), we construct a sequence $(h^{(k)})_{k\in\NN}$
contained in $[0, H]^{n-1}$ such that
\[
\sup_{i = \widetilde i, \iadd} |E_{i}^{(r)}(h^{(k)})| \le L^{m+k}K\delta.
\]
In particular, since $L<1$,
\[
E_i^{(r)}(h^{(k)})\to 0 \quad \text{for } i\in \{\iadd, \widetilde i\}.
\]
Because $[0,H]^{n-1}$ is compact, the sequence $(h^{(k)})_{k\in\NN}$ admits a subsequence converging to some $\widehat h\in[0,H]^{n-1}$.
By continuity of $E^{(r)}$, we have $E_i^{(r)}(\widehat h)=0$ for $i\in \{\iadd, \widetilde i\}$ and, by construction, $\widehat h_i=h_i$ for all $i\notin \{\iadd, \widetilde i\}$. Therefore, $(i)$ of Lemma~\ref{lemma:induction} holds.

Finally, applying Proposition~\ref{prop:non-expans} to $\widehat h$ and $h$,  with $I_{\mathrm{dom}}=\{\iadd, \widetilde i \}$ and $I_{\mathrm{cdom}}=\{\widetilde i \}$, gives
\[
\sup_{j=1,\dots,n-1} |E_j^{(r)}(\widehat h)|\le \sum_{k=1}^{m+1} L^k K\delta \le L\delta,
\]
which is $(ii)$ of Lemma~\ref{lemma:induction}.

\medskip
\noindent$\bullet$ \emph{Base case $s=2$ $\mathrm(1<\iadd=i_\delta<n-1\mathrm)$.}
Assume $\iadd=i_\delta$, $1<i_\delta<n-1$, and let $m \in \NN$ be such that \eqref{eq:bounds-for-inductions} holds.
Then, by assumption, $\{i_\delta-1,i_\delta+1\}\subseteq I$, and the smallest case is $I=\{i_\delta-1,i_\delta+1\}$.
In this situation, we need to ensure that the one--coordinate adjustment for the $i_\delta$-th component remains available along the iteration.
We proceed as follows.

Start from $h^{(0)}:=h$ and adjust the $i_\delta$-th component to obtain $h^{(1)}$ with $E_{i_\delta}^{(r)}(h^{(1)})=0$.
By Proposition~\ref{prop:non-expans}, the deviations at $i_\delta\pm 1$ satisfy
\[
|E_{i_\delta-1}^{(r)}(h^{(1)})|\le L^{m+1}K\delta,
\qquad
|E_{i_\delta+1}^{(r)}(h^{(1)})|\le L^{m+1}K\delta,
\qquad
\sup_j |E_j^{(r)}(h^{(1)})| \leq \sum_{k=1}^{m+1} L^k K\delta <\delta.
\]
Next adjust the $(i_\delta-1)$-th component to obtain $h^{(2)}$ with $E_{i_\delta-1}^{(r)}(h^{(2)})=0$.
If $E_{i_\delta}^{(r)}(h)\ge 0$, then by monotonicity of $f$ the adjustment of the $(i_\delta-1)$-th component preserves
$E_{i_\delta}^{(r)}(h^{(2)})\ge 0$, hence the $i_\delta$-th component can be adjusted again.
If instead $E_{i_\delta}^{(r)}(h)\le 0$, then $h^{(1)}_{i_\delta}-h_{i_\delta}\le 0$ and $h^{(2)}_{i_\delta-1}-h_{i_\delta-1}\ge 0$.
Since $f_{i_\delta-1}(h^{(2)})=f_{i_\delta-1}(h)$, Taylor's Theorem yields
\[
0=f_{i_\delta-1}(h^{(2)})-f_{i_\delta-1}(h)
= \frac{\partial f_{i_\delta-1}(\xi)}{\partial h_{i_\delta-1}} |h^{(2)}_{i_\delta-1}-h_{i_\delta-1}|
- \frac{\partial f_{i_\delta-1}(\xi)}{\partial h_{i_\delta}} |h^{(1)}_{i_\delta}-h_{i_\delta}|,
\]
for some $\xi = t h + (1-t)h^{(2)}$, $t\in [0,1]$.
By Proposition~\ref{prop:regulatiry_f}, $\left|\frac{\partial f_{i_\delta-1}(\xi)}{\partial h_{i_\delta-1}}\right| \geq
\left|\frac{\partial f_{i_\delta-1}(\xi)}{\partial h_{i_\delta}}\right|$, hence
\[
|h^{(2)}_{i_\delta-1}-h_{i_\delta-1}| \leq |h^{(1)}_{i_\delta}-h_{i_\delta}|.
\]
Applying Taylor's Theorem to $f_{i_\delta+1}$ and using again Proposition~\ref{prop:regulatiry_f}, we conclude that
\[
f_{i_\delta+1}(h^{(2)})-f_{i_\delta+1}(h)
= \frac{\partial f_{i_\delta+1}(\xi)}{\partial h_{i_\delta-1}} |h^{(2)}_{i_\delta-1}-h_{i_\delta-1}|
- \frac{\partial f_{i_\delta+1}(\xi)}{\partial h_{i_\delta}} |h^{(1)}_{i_\delta}-h_{i_\delta}| \ge 0,
\]
thus $E_{i_\delta+1}^{(r)}(h^{(2)})\ge 0$. In particular, the $i_\delta$-th component can be adjusted again (see Figure \ref{fig:proof-2}). This sign property is preserved at each subsequent step of the same construction. 
Therefore, we can repeat the same iterative procedure used in the base case $s=1$. More precisely, we construct a sequence $(h^{(k)})_{k\in\NN}\subseteq[0,H]^{n-1}$ by alternating the one-coordinate adjustment between the coordinates $i_\delta-1$ and $i_\delta$. As in the base case $s=1$, compactness allows us to extract a convergent subsequence, whose limit we denote by $\widetilde h^{(1)}\in[0,H]^{n-1}$. By continuity,
\[
E_j^{(r)}(\widetilde h^{(1)})=0, \qquad j=i_\delta-1,i_\delta.
\]
Moreover, applying Proposition~\ref{prop:non-expans} to $h$ and $\widetilde h^{(1)}$, with $I_{\mathrm{dom}}=\{i_\delta-1,i_\delta\}$ and  $I_{\mathrm{cdom}}=\{i_\delta-1\}$, we obtain
\[
|E_{i_\delta+1}^{(r)}(\widetilde h^{(1)})|\le L^{m+1}K\delta,
\qquad
\sup_{j=1,\dots,n-1}|E_j^{(r)}(\widetilde h^{(1)})|
\le \sum_{k=1}^{m+1}L^kK\delta
\le L\delta.
\]
We can now repeat the same argument, this time alternating the adjustment between the coordinates $i_\delta+1$ and $i_\delta-1$. Let $\widetilde h^{(2)}\in[0,H]^{n-1}$ be a limit point of the sequence obtained in this way. Then
\[
E_j^{(r)}(\widetilde h^{(2)})=0,
\qquad j=i_\delta-1,i_\delta+1.
\]
Applying Proposition~\ref{prop:non-expans} to $\widetilde h^{(1)}$ and $\widetilde h^{(2)}$, with $I_{\mathrm{dom}}=\{i_\delta-1,i_\delta+1\}$ and $I_{\mathrm{cdom}}=\{i_\delta-1\}$, gives
\[
|E_{i_\delta}^{(r)}(\widetilde h^{(2)})|\le L^{m+2}K\delta,
\qquad
\sup_{j=1,\dots,n-1}|E_j^{(r)}(\widetilde h^{(2)})|
\le \sum_{k=1}^{m+2}L^kK\delta
\le L\delta.
\]

\begin{figure}[H]
     \centering
          \begin{subfigure}{0.65\textwidth}
         \centering
         \includegraphics[width=\textwidth]{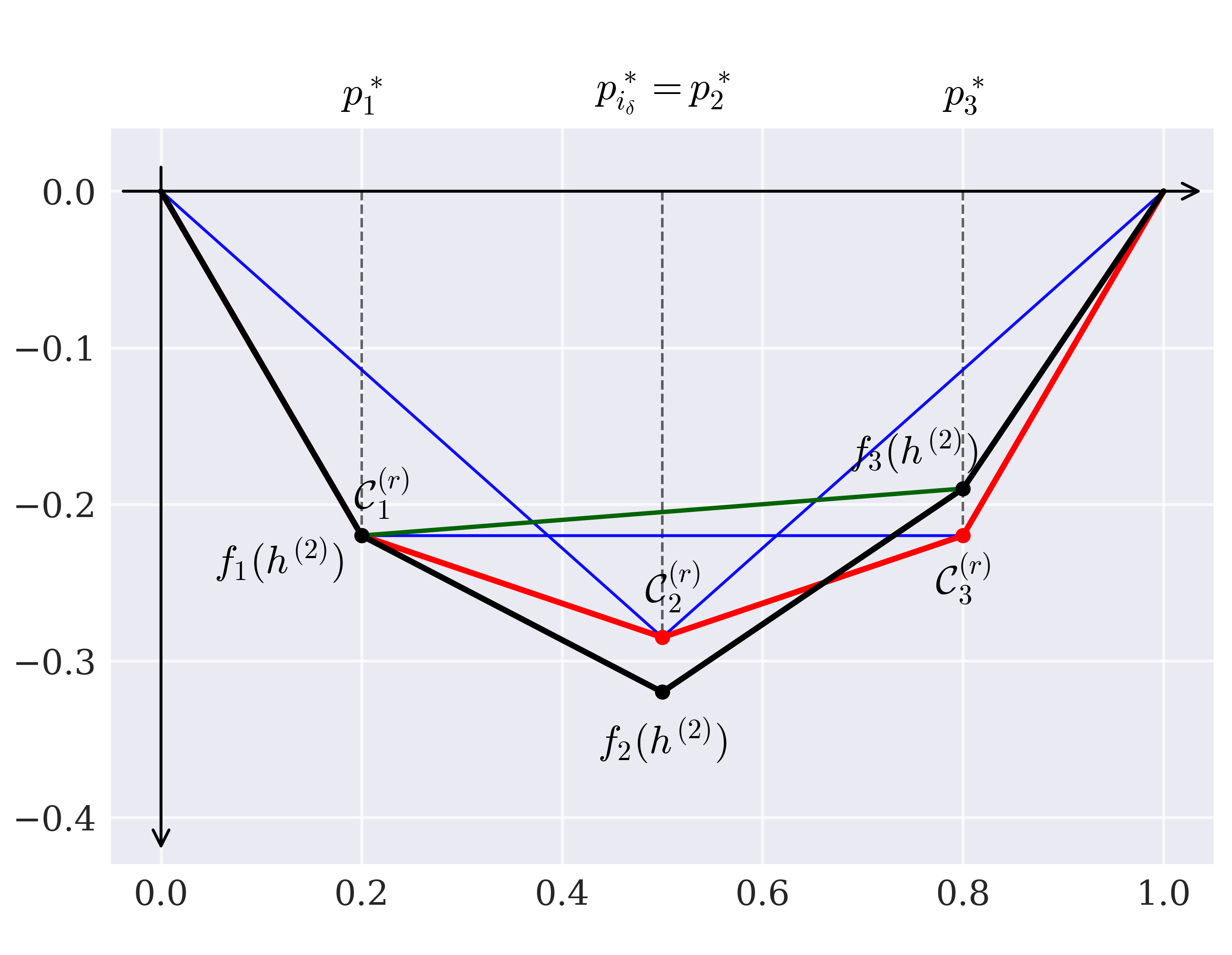}
     \end{subfigure}
	\caption{In this example, $\Ccal^{(r)}$ is shown in red and $i_\delta=2$. Since $E^{(r)}_{i_\delta+1}(h^{(2)}) \ge 0$, we also have $E^{(r)}_{i_\delta}(h') \ge 0$, where $h'$ is the vector obtained from $h^{(2)}$ by setting its $i_\delta$-th component equal to $0$. Hence, the $i_\delta$-th component can be adjusted again using the standard argument.}
	\label{fig:proof-2}
\end{figure}

Running the procedure alternately on the pairs $(i_\delta,i_\delta-1)$ and $(i_\delta-1,i_\delta+1)$, in the order 
\[
(i_\delta,i_\delta-1),\ (i_\delta-1,i_\delta+1),\ (i_\delta,i_\delta-1),\ (i_\delta-1,i_\delta+1),\ \dots,
\]
yields a sequence $(\widetilde h^{(k)})_{k\in\NN}\subseteq [0,H]^{n-1}$ such that $E_i^{(r)}(\widetilde h^{(k)})\to 0$ for all $i\in I\cup\{\iadd\}$. By compactness of $[0,H]^{n-1}$, there exists a convergent subsequence with limit $\widehat h\in[0,H]^{n-1}$, and by continuity
\[
E_{i_\delta}^{(r)}(\widehat h)=E_{i_\delta-1}^{(r)}(\widehat h)=E_{i_\delta+1}^{(r)}(\widehat h)=0.
\]
Finally, the estimate in \ref{it:induction_2} follows by applying Proposition~\ref{prop:non-expans} to $h$ and $\widehat h$, with $I_{\mathrm{dom}}=\{i_\delta-1,i_\delta,i_\delta+1\}$ and $I_{\mathrm{cdom}}=\{i_\delta-1,i_\delta+1\}$.

\medskip
\noindent$\bullet$ \emph{Induction step.}
Fix $s\in\{1,\dots,n-2\}$ and assume the statement holds for every $I'\subseteq\{1,\dots,n-1\}$ with $|I'|=s$ (with the same hypotheses on $i_\delta$ and, when relevant, on the neighbors of $\iadd=i_\delta$).
Let $I\subseteq\{1,\dots,n-1\}$  such that $|I|=s+1$, $i_\delta \notin I$ and fix $m \in \NN$ such that \eqref{eq:bounds-for-inductions} holds.
If $\iadd=i_\delta$ and $1<i_\delta<n-1$, choose $\widetilde i\in I\setminus\{i_\delta-1,i_\delta+1\}$; otherwise choose any $\widetilde i\in I$.
Set $\widetilde I:=I\setminus\{\widetilde i\}$, so $|\widetilde I|=s$ and $\widetilde I$ still satisfies the assumptions of the lemma.

Apply the induction hypothesis to the pair $(\widetilde I,\iadd)$ and the vector $h$.
We obtain $h^{(1)}\in [0,H]^{n-1}$ such that
\[
h^{(1)}_i=h_i\ \text{ for } i\notin \widetilde I\cup\{\iadd\},
\qquad
E_i^{(r)}(h^{(1)})=0\ \text{ for } i\in \widetilde I\cup\{\iadd\},
\qquad
\sup_j |E_j^{(r)}(h^{(1)})|\le \sum_{k=1}^{m+1} L^kK\delta.
\]
By Proposition~\ref{prop:non-expans} applied to $h$ and $h^{(1)}$ with $I_{\mathrm{dom}}=\widetilde I\cup\{\iadd\}$ and $I_{\mathrm{cdom}}=\widetilde I$, we also have $|E_{\;\widetilde i}^{(r)}(h^{(1)})|\le L^{m+1}K\delta$.
Now apply the induction hypothesis again, this time to the pair $(\widetilde I,\widetilde i)$ and the vector $h^{(1)}$ (note $\widetilde i\notin \widetilde I$ and $i_\delta\notin \widetilde I$),
to obtain $h^{(2)}\in [0,H]^{n-1}$ such that
\[
h^{(2)}_i=h_i\ \text{ for } i\notin I,
\qquad
E_i^{(r)}(h^{(2)})=0\ \text{ for } i\in I,
\qquad
\sup_j |E_j^{(r)}(h^{(2)})|\le \sum_{k=1}^{m+2} L^kK\delta.
\]
Again, Proposition~\ref{prop:non-expans}  applied to $h^{(1)}$ and $h^{(2)}$ with $I_{\mathrm{dom}}=I$ and $I_{\mathrm{cdom}}=\widetilde I$, gives $|E_{\iadd}^{(r)}(h^{(2)})|\le L^{m+2}K\delta$.
Iterating these two steps alternately produces a sequence $(h^{(k)})_{k\in\NN}\subseteq [0,H]^{n-1}$
for which $h^{(k)}_i=h_i$ for $i\notin I\cup\{\iadd\}$ and
\[
\sup_{i \in I\cup\{\iadd\}}|E_{i}^{(r)}(h^{(k)})|\le L^{m+k}K\delta.
\]
Since $L<1$, it follows that $E_i^{(r)}(h^{(k)})\to 0$ for every $i\in I\cup\{\iadd\}$.
By compactness of $[0,H]^{n-1}$, the sequence $(h^{(k)})_{k\in\NN}$ admits a convergent subsequence with limit $\widehat h\in[0,H]^{n-1}$.
By continuity, $E_i^{(r)}(\widehat h)=0$ for all $i\in I\cup\{\iadd\}$, and \ref{it:induction_1} holds by construction.
Finally, applying Proposition~\ref{prop:non-expans} to $\widehat h$ and $h$ gives \ref{it:induction_2}, completing the induction.
\end{proof}
\medskip

\begin{proof}[Proof of  Theorem~\ref{thm:induction_thm}]
We first prove the statement under the additional Assumption~\ref{ass:technical-assumptions}, arguing by induction on $s:=|J|$.  The induction step repeatedly uses Lemma~\ref{lemma:induction}. We recall that this lemma relies on Assumption~\ref{ass:technical-assumptions}, under which the required regularity properties of $f$ were established. These properties enter in two ways. First, they are used to prove Proposition~\ref{prop:non-expans}, which is a key ingredient in the proof of Lemma~\ref{lemma:induction}. Second, they ensure that the constant $L\in(0,1)$ can be chosen with the required properties.
In this first part we therefore construct $\widehat h\in D$ satisfying \eqref{eq:theoremClaim} within that setting. Once this is done, we remove Assumption~\ref{ass:technical-assumptions} and we choose suitable sequences $(\nu_k,q_k)$ satisfying Assumption~\ref{ass:technical-assumptions} and converging to $(\nu,q)$, obtain $\widehat h^{(k)}$ satisfying \eqref{eq:theoremClaim} from the first part, and then pass to the limit to produce a vector $\widehat h$ solving \eqref{eq:theoremClaim} for the original problem.

\medskip
\noindent \emph{I. Proof by induction under Assumption~\ref{ass:technical-assumptions}.}\\
\noindent$\bullet$ \emph{Base case $s=1$.}
Let $J=\{\widetilde i\}$. Consider the one--parameter family
\[
h(t):=(0,\dots,0,t,0,\dots,0)\in\RR^{n-1}_{\ge 0}, \qquad t\ge 0,
\]
where $t$ sits in the $\widetilde i$-th coordinate. By continuity and monotonicity of $f_{\;\widetilde i}$ in the $\widetilde i$-th coordinate, the map $t\mapsto f_{\;\widetilde i}(h(t))$ is continuous and monotone. Moreover,
\[
f_{\;\widetilde i}(h(0))=f_{\;\widetilde i}(0)=0,
\qquad
\lim_{t\to\infty} f_{\;\widetilde i}(h(t)) = U_\nu(p_{\;\widetilde i}^*),
\]
by Proposition~\ref{prop:mainThm-n_atoms-part1}. Therefore, by the intermediate value theorem, there exists $t_{\;\widetilde i}\in(0,\diam(\supp(q)))$ such that
$f_{\;\widetilde i}(h(t_{\;\widetilde i}))=\Ccal_{\;\widetilde i}$. Setting $\widehat h:=h(t_{\;\widetilde i})$ yields \eqref{eq:theoremClaim}.

\medskip
\noindent$\bullet$ \emph{Induction step.}
Assume that the statement holds for every nonempty subset $J'\subseteq\{1,\dots,n-1\}$ with $|J'|=s$, and let $J\subseteq\{1,\dots,n-1\}$ satisfy $|J|=s+1$.
By Remark \ref{rmk:restrictions}, restricting $f$ to the coordinate subspace
\[
\{h\in\RR^{n-1}_{\ge 0}:h_i=0\text{ for }i\notin J\}\cong \RR^J_{\ge 0},
\]
and keeping only the components indexed by $J$, yields a map $f_J:\RR^J_{\ge 0}\to\RR^J$. Denote by $\Ccal_J$ the convex polygonal chain with vertices $(0,0), \{(p_i^*,\Ccal_i), i\in J\}, (1,0)$. With this identification, it is enough to find $\widehat h_J\in\RR^J_{\ge 0}$ such that
\[
f_J(\widehat h_J)=\Ccal_J,
\]
and then extend $\widehat h_J$ by $0$ outside $J$.

Let $\delta$ be the minimal diagonal gap of $\Ccal_J$, and let $i_\delta$ be a minimal-gap index with respect to $\Ccal_J$. Set $J':=J\setminus\{i_\delta\}$. Then $|J'|=s$. By the induction hypothesis applied to $J'$, there exists $h_J^{(0)}\in [0,\diam(\supp(q)))^J$ such that
\[
(f_J)_i(h_J^{(0)})=(\Ccal_J)_i,\qquad i\in J',\qquad\text{and}\qquad (h_J^{(0)})_{i_\delta}=0.
\]
Let $\Ccal_J^{(0)}$ be the polygonal obtained from $\Ccal_J$ by removing the vertex indexed by $i_\delta$, as in the definition of the $L$-approximation. Since the only removed coordinate is $i_\delta$, we have $(\Ccal_J^{(0)})_i=(\Ccal_J)_i$, for $i\in J'$.
Moreover, by the collinearity relation \eqref{eq:colinearity} and the fact that $(h_J^{(0)})_{i_\delta}=0$, we also have $(f_J)_{i_\delta}(h_J^{(0)})=(\Ccal_J^{(0)})_{i_\delta}$. Hence
\[
(f_J)_i(h_J^{(0)})=(\Ccal_J^{(0)})_i,\qquad i\in J.
\]

Set
\[
H:=\max\left(\max_{i\in J}\widetilde f_i^{-1}(\Ccal_i),\|h_J^{(0)}\|_\infty\right),
\]
where $\widetilde f$ is the map defined in \eqref{eq:bound-map-q-Bass}. By Lemma~\ref{adjLemma}, the quantity
\[
L:=\max_{\substack{\xi \in [0,H]^{n-1}\\ \Jcal \subseteq \{1,\dots,n-1\}\\ i,j \in \Jcal,\ i\neq j}}
\left|\frac{\adj(Jf(\xi)_{\Jcal,\Jcal})_{ij}}{\adj(Jf(\xi)_{\Jcal,\Jcal})_{ii}}\right|
\]
belongs to $(0,1)$. We set $K:=1-L$.

We now verify that the hypotheses of Lemma~\ref{lemma:induction} are satisfied in the restricted problem. We apply the lemma to the map $f_J$, with $I=J'$ and $\iadd=i_\delta$. Thus $I$ is nonempty and proper in $J$, $\iadd\notin I$, and, by construction, $i_\delta\notin I$.
The choice of $H$ gives the required a priori bound for the initial point $h_J^{(0)}$. Moreover, since every point produced by Lemma~\ref{lemma:induction} belongs to $[0,H]^J$, the same value of $H$ can be used at each step. The adjugate-ratio estimate required in Lemma~\ref{lemma:induction} is precisely ensured by the definition of $L$.

It remains to check the compatibility condition \eqref{eq:tech-implications} when the added index is the minimal-gap index. In the restricted polygonal chain $\Ccal_J$, after relabelling the indices of $J$ increasingly, the neighbours of $i_\delta$ are exactly the adjacent vertices to $(p_{i_\delta}^*,(\Ccal_J)_{i_\delta})$. Since $J'=J\setminus\{i_\delta\}$, all these neighbouring indices belong to $J'$. Hence the analogue of \eqref{eq:tech-implications} holds in the restricted coordinate system.

For every $r\in\NN$, let $\Ccal_J^{(r)}$ denote the $L$-approximation of $\Ccal_J$ of order $r$ along $i_\delta$. We prove by induction on $r$ that there exists $h_J^{(r)}\in[0,H]^J$ such that
\[
(f_J)_i(h_J^{(r)})=(\Ccal_J^{(r)})_i,\qquad i\in J.
\]
The case $r=0$ has just been established. Assume now that $h_J^{(r)}\in[0,H]^J$ has been constructed and satisfies $(f_J)_i(h_J^{(r)})=(\Ccal_J^{(r)})_i$, for all $i\in J$. We want to apply Lemma~\ref{lemma:induction} to the error map associated to $\Ccal_J^{(r+1)}$, with $I=J'$, $\iadd=i_\delta$ and $m=r$, thus we are going to show that all necessary conditions are satisfied. Since $\Ccal_J^{(r+1)}$ differs from $\Ccal_J^{(r)}$ only in the $i_\delta$-th component, we have
\[
|E_{i_\delta}^{(r+1)}(h_J^{(r)})|=|(\Ccal_J^{(r)})_{i_\delta}-(\Ccal_J^{(r+1)})_{i_\delta}|=L^rK\delta, \qquad \text{and}\qquad E_i^{(r+1)}(h_J^{(r)})=0,\quad i\in J'.
\]
Therefore, $E_{i_\delta}^{(r+1)}(h_J^{(r)})\le L^rK\delta$ and, since all other errors vanish,
\[
\sup_{\substack{i\in J\\ i\neq i_\delta}} |E_i^{(r+1)}(h_J^{(r)})|=0\le \sum_{k=1}^{r}L^kK\delta\le L\delta.
\]
Thus all the hypotheses of Lemma~\ref{lemma:induction} are satisfied. Hence there exists $h_J^{(r+1)}\in[0,H]^J$ such that
\[
(f_J)_i(h_J^{(r+1)})=(\Ccal_J^{(r+1)})_i,\qquad i\in J.
\]
This completes the induction on $r$.

By Remark \ref{rmk:approx}, we have
\[
\lim_{r\to\infty}(\Ccal_J^{(r)})_i=(\Ccal_J)_i,\qquad i\in J.
\]
Consequently,
\[
\lim_{r\to\infty}(f_J)_i(h_J^{(r)})=(\Ccal_J)_i,\qquad i\in J.
\]
Since $[0,H]^J$ is compact, there exists a subsequence $h_J^{(r_\ell)}$ converging to some $\widehat h_J\in[0,H]^J$. By continuity of $f_J$,
\[
(f_J)_i(\widehat h_J)=\lim_{\ell\to\infty}(f_J)_i(h_J^{(r_\ell)})=(\Ccal_J)_i,\qquad i\in J.
\]
Finally, define $\widehat h\in\RR^{n-1}_{\ge 0}$ by
\[
\widehat h_i=(\widehat h_J)_i \quad\text{for } i\in J,\qquad \widehat h_i=0 \quad\text{for } i\notin J.
\]
Then \eqref{eq:theoremClaim} holds.

\medskip

\noindent \emph{II. Proof by approximation argument in the general case.} We now drop the extra Assumption~\ref{ass:technical-assumptions}.
We denote by $\mu \in \Prob_1(\RR)$ the distribution whose integrated quantile function has graph equal to the polygonal chain $\Ccal$. In particular, the pair $(\mu,\nu)$ is irreducible.
By Propositions \ref{prop:approx_nu_1} and \ref{prop:approx_q_1}, there exist two sequences $(\nu_k)_{k \in \NN}$ and $(q_k)_{k \in \NN}$ in $\Prob_1(\RR)$ such that, for every $k \in \NN$,
\begin{enumerate}[label=(\roman*)]
	\item $\nu_k$ and $q_k$ are absolutely continuous,
	\item $\supp(\nu_k)$ and $\supp(q_k)$ are bounded intervals,
	\item $\rho_{\nu_k}$ is bounded away from $0$ on $\supp(\nu_k)$ and $\rho_{q_k} \in L^\infty(\RR)$,
	\item $(\mu,\nu_k)$ is irreducible,
	\item $Q_{\nu_k} \to Q_\nu$ pointwise on $(0,1)$,
	\item $\rho_{q_k} \to \rho_q$ in $L^1(\RR)$,
	\item there exists $Q \in L^1(0,1)$ such that
	\[
	|Q_{\nu_k}(p)| \leq Q(p), \qquad \text{for all } p \in (0,1).
	\]
\end{enumerate}
These properties imply that $\nu_k$ and $q_k$ satisfy Assumption~\ref{ass:technical-assumptions} for every $k \in \NN$. In particular, $\nu_k, q_k \in \Prob_\infty(\RR)$ because they are compactly supported. Moreover, $Q_{\nu_k}' \in L^\infty(0,1)$, since $\rho_{\nu_k}$ is bounded away from $0$ on its support.

Now, let
\[
f^{\nu_k,q_k}:\RR_{\geq 0}^{n-1} \to \RR^{n-1}
\]
be the $q_k$-Bass map with final marginal $\nu_k$, namely
\[
f^{\nu_k,q_k}_i(h_1,\dots,h_{n-1})
=
\int_\RR
Q_{\nu_k}\!\left(\sum_{j=1}^{n} p_j F_{q_k}\!\left(z-\sum_{\ell=1}^{j-1} h_\ell\right)\right)
\sum_{j=1}^{i} p_j \rho_{q_k}\!\left(z-\sum_{\ell=1}^{j-1} h_\ell\right)\, dz,
\]
for every $i \in \{1,\dots,n-1\}$.
Since $(\mu,\nu_k)$ is irreducible, for every $k \in \NN$ we can apply Theorem~\ref{thm:induction_thm} already proved under Assumption~\ref{ass:technical-assumptions}. Therefore, there exists $h^{(k)} \in \RR_{\geq 0}^{n-1}$ such that
\[
f^{\nu_k,q_k}_i(h^{(k)})=\Ccal_i
\quad \text{for all } i \in J, \qquad \text{and} \qquad h^{(k)}_i = 0 \quad \text{for all } i \not\in J.
\]

Next, we prove that the sequence $(h^{(k)})_{k\in\NN}$ is bounded. Suppose, by contradiction, that it is unbounded. Then, by Proposition~\ref{prop:q-Bass-map-stability} there exist an index $\widehat i\in\{1,\dots,n-1\}$ and a subsequence $(h^{(k_n)})_{n\in\NN}$ such that
\[
f^{\nu_{k_n},q_{k_n}}_{\widehat i}(h^{(k_n)}) \rightarrow U_\nu(p_{\widehat i}^*)<\Ccal_{\widehat i},
\]
which contradicts
\[
f^{\nu_{k_n},q_{k_n}}_{\widehat i}(h^{(k_n)})=\Ccal_{\widehat i}
\qquad\text{for every } n\in\NN.
\]
Hence $(h^{(k)})_{k\in\NN}$ is bounded. Therefore, we can invoke Proposition~\ref{prop:q-Bass-map-stability} one more time to extract a converging subsequence $\widehat h\in\RR_{\geq 0}^{n-1}$ such that $h^{(k_n)}\to \widehat h \in \RR_{\geq 0}^{n-1}$ and
\[
	f_i(\widehat h) = \lim_{n\to\infty} f_i^{\nu_{k_n},q_{k_n}}(h^{(k_n)}) = \Ccal_i, \quad \text{for all } i \in J, \qquad \text{and} \qquad \widehat h_i = 0 \quad \text{for all } i \not\in J,
\]
which concludes the proof.
\end{proof}

\appendix
\section{Postponed proofs and auxiliary results}\label{app.post}
This section contains some postponed proofs, as well as auxiliary results.
We start by providing the proof of two results stated in the Introduction: the well-posedness of the map $q\star T$ in the $q$-Bass martingale, and the new characterizations of convex order and irreducibility.

\begin{proof}[Proof of Proposition~\ref{prop:q-star-T-well-posed}]
Since the function $x \mapsto |Q_\nu(F_{q \ast \alpha}(x))|$ 
is measurable and non-negative, we define $I:\RR \rightarrow [0,\infty]$
\[
	I(x) = \int_\RR |Q_\nu(F_{q \ast \alpha}(x+z))| \rho_q(z)\, dz.
\]		
Now, note that 
\begin{align*}
	\int_\RR I(x) \alpha(dx) & = \int_0^1 I(Q_\alpha(u)) du = \int_0^1 \int_\RR |Q_\nu(F_{\alpha \ast q}(Q_\alpha(u)+z))| \rho_q(z)\, dz \, du\\
	& = \int_0^1 \int_\RR \left|Q_\nu\left(\int_0^1 F_q(z+Q_\alpha(u)-Q_\alpha(v))\, dv\right)\right| \rho_q(z)\, dz \, du \\
	& =  \int_\RR \left|Q_\nu\left(\int_0^1 F_q(z-Q_\alpha(v))\, dv\right)\right| \int_0^1 \rho_q(z-Q_\alpha(u))\, du \, dz = \int_0^1 |Q_\nu(u)| du <\infty,
\end{align*}
since $\nu \in \Prob_1(\RR)$. Then $I(x)<\infty$ $\alpha$-a.e.
\end{proof}
\medskip

\begin{proof}[Proof of Proposition~\ref{prop:convex_order_char}]
\emph{Proof of (i).}\    
    Let $p\in(0,1)$. We have
	\[
		|U_\eta(p)|
		=
		\left|\int_0^p Q_\eta(u)\,du\right|
		\leq
		\int_0^p |Q_\eta(u)|\,du
		\leq
		\int_0^1 |Q_\eta(u)|\,du
		<\infty.
	\]
	Therefore $U_\eta$ is well-defined in $[0,1]$. Since $Q_\eta$ is
	non-decreasing, $U_\eta$ is convex. Moreover,
	\[
		U_\eta(0)=0,
		\qquad
		U_\eta(1)=\int_0^1 Q_\eta(u)\,du=\mean(\eta).
	\]

\noindent\emph{Proof of (ii).}\ By \cite[Theorem 3.A.5]{ShSh07}, for
	$\eta,\eta'\in\Prob_1(\RR)$ one has $\eta\preceq_c\eta'$ if and only if
	\[
		\mean(\eta)=\mean(\eta')
	\]
	and
	\[
		\int_0^p Q_\eta(u)\,du
		\geq
		\int_0^p Q_{\eta'}(u)\,du,
		\qquad p\in(0,1).
	\]
	By (i), $\mean(\eta)=U_\eta(1)$ and
	$\mean(\eta')=U_{\eta'}(1)$. Hence the previous condition is equivalent
	to
	\[
		U_\eta(1)=U_{\eta'}(1),
		\qquad
		U_\eta(p)\geq U_{\eta'}(p)
		\quad\text{for all }p\in(0,1).
	\]

\noindent\emph{Proof of (iii).}\ We use the call-function characterization recalled in Remark \ref{rmk:convex-order}, so that, if $\eta\preceq_c\eta'$, then $D(x):=C_{\eta'}(x)-C_\eta(x)\geq 0$ for all $x \in \RR$. We also use the following elementary duality. For
	$\rho\in\Prob_1(\RR)$ and $q\in(0,1)$, set $\overline U_\rho(q):=\int_{1-q}^1 Q_\rho(u)\,du$. Then
	\[
		\overline U_\rho(q)
		=
		\inf_{x\in\RR}\{C_\rho(x)+qx\},
	\]
	and the minimizers are precisely the $(1-q)$-quantiles of $\rho$.
	Therefore,
	\[
		\overline U_\rho(q)=\mean(\rho)-U_\rho(1-q).
	\]

	Assume first that $(\eta,\eta')$ is irreducible. By definition,
	$\eta\preceq_c\eta'$, and
	\[
		I:=\{x\in\RR:C_\eta(x)<C_{\eta'}(x)\}
		=
		\{x\in\RR:D(x)>0\}
	\]
	is an interval such that $\eta(I)=1$. By (ii), $U_\eta(1)=U_{\eta'}(1)$, and	
	$U_\eta(p)\geq U_{\eta'}(p)$
	for $p\in(0,1)$.
	Now, let $p\in(0,1)$ and set $q:=1-p$. Suppose by contradiction that $U_\eta(p)=U_{\eta'}(p)$. Since the means are equal, this is equivalent to $\overline U_\eta(q)=\overline U_{\eta'}(q)$. 
	
	Let $x$ be a minimizer of $z\mapsto C_{\eta'}(z)+qz$. Then, using $C_\eta\leq C_{\eta'}$,
	\[
		\overline U_\eta(q)
		\leq
		C_\eta(x)+qx
		\leq
		C_{\eta'}(x)+qx
		=
		\overline U_{\eta'}(q).
	\]
	All inequalities are therefore equalities. Hence $C_\eta(x)=C_{\eta'}(x)$, and $x$ is also a minimizer of $z\mapsto C_\eta(z)+qz$. Thus $x$ is a $p$-quantile of $\eta$. Since $p\in(0,1)$ and $\eta(I)=1$, every such quantile belongs to $I$. Therefore $C_\eta(x)<C_{\eta'}(x)$, which contradicts $C_\eta(x)=C_{\eta'}(x)$. Hence $U_\eta(p)>U_{\eta'}(p)$, for all $p \in (0,1)$.

	Conversely, assume that
	\[
		U_\eta(1)=U_{\eta'}(1), \qquad U_\eta(p)>U_{\eta'}(p) \quad\text{for all }p\in(0,1).
	\]
	By (ii), $\eta\preceq_c\eta'$, and hence $D\geq0$. Let $I:=\{x\in\RR:D(x)>0\}$. We prove that $\eta(I)=1$ and that $I$ is an interval. Let $x$ be a $p$-quantile of $\eta$ for some $p\in(0,1)$, and set $q:=1-p$. If $D(x)=0$, then by the dual formula,
	\[
		\overline U_\eta(q) = C_\eta(x)+qx = C_{\eta'}(x)+qx \geq \overline U_{\eta'}(q).
	\]
	On the other hand, since $C_\eta\leq C_{\eta'}$, $\overline U_\eta(q)\leq \overline U_{\eta'}(q)$. Therefore $\overline U_\eta(q)=\overline U_{\eta'}(q)$ and, since the means are equal, $U_\eta(p)=U_{\eta'}(p)$ contradicting the assumed strict inequality. Hence $D(x)>0$ for every non-trivial quantile $x$ of $\eta$. It follows in particular that $\eta(I)=1$.

	It remains to observe that $I$ is an interval. Let
	\[
		J:=\{x\in\RR:\exists p\in(0,1) \text{ such that }x\text{ is a }p\text{-quantile of }\eta\}.
	\]
	Then $J$ is an interval, $\eta(J)=1$, and the previous argument gives $J\subseteq I$. Moreover, $D$ is locally absolutely continuous and, at $\lambda$-a.e.
	point,
	\[
		D'(x)=F_{\eta'}(x)-F_\eta(x).
	\]
	On the left of $J$ we have $F_\eta=0$, hence $D'\geq0$ a.e. Since $D\geq0$, the set $\{D>0\}$ on the left of $J$ can only be an interval attached to $J$. Similarly, on the right of $J$ we have $F_\eta=1$, hence $D'\leq0$ a.e. and $\{D>0\}$ on the right of $J$ can only be an interval attached to $J$. Since $J\subseteq I=\{D>0\}$, no further connected component of $I$ can occur away from $J$. Hence $I$ is an interval. Moreover, $\eta(I)=1$ because $\eta(J)=1$ and $J\subseteq I$. Thus $(\eta,\eta')$ is irreducible.
\end{proof}
To show regularity of the potential function, we will use the following result.
\begin{lemma}[Differentiation under the integral sign]
\label{prop:interchange}
	Let $X\subseteq \RR$ be an open set, and $\Omega$ be a measure space. Suppose that a measurable function $f : X \times \Omega \rightarrow \RR$ satisfies the following conditions:
	\begin{enumerate}[label=(\roman*)]
		\item for any $x \in X$, $\int_\Omega |f(x, \omega)| d\omega<\infty$;
		\item for any $\omega \in \Omega$, $f(x, \omega)$ is an absolutely continuous function of $x$;
		\item one of the following holds:
			\begin{enumerate}
				\item $\int_a^b \int_\Omega \left | \frac{\partial }{\partial x}f(x,\omega) \right| d\omega dx < \infty$ for all $a, b \in X$, $a<b$, 
				\item $\frac{\partial }{\partial x}f(x,\omega) \geq 0$ for all $(x, \omega) \in X \times \Omega$.
			\end{enumerate}
	\end{enumerate}
	Then the map $x \mapsto \int_\Omega f(x, \omega) dw$ is  absolutely continuous and, for almost every $x \in X$, its derivative exists and is given by
	\begin{equation}
		\label{interchange}
		\frac{d}{dx}\int f(x, \omega) d\omega = \int_\Omega \frac{\partial}{\partial x} f(x, \omega) d\omega.
	\end{equation}
\end{lemma}

\begin{proof}
	For any $\omega \in \Omega$, we have
	\[
		f(b, \omega) - f(a, \omega) = \int_a^b \frac{\partial}{\partial x }f(x, \omega) dx.
	\]
Therefore, condition (iii)(a) allows us to apply Fubini's theorem, while condition (iii)(b) allows us to apply Tonelli's theorem. In either case, we conclude that
\[
	\int_{\Omega}f(b, \omega)d\omega - \int_{\Omega}f(a, \omega) d\omega = \int_{\Omega}\int_a^b \frac{\partial}{\partial x }f(x, \omega) dx d\omega = \int_a^b\int_{\Omega}\frac{\partial}{\partial x }f(x, \omega) d\omega dx.
	\]
	Therefore, the function $x \mapsto \int_\Omega f(x, \omega)$ is absolutely continuous and \eqref{interchange} holds a.e.
\end{proof}
\medskip

\begin{proof}[Proof of Proposition
\ref{prop:regulatiry_f}]

Fix $h\in D$. Showing that $V(h)$ is well-defined amounts to proving that $z\mapsto v(h,z)=U_\nu(g(h,z))$ is integrable. Set
\[
s_j:=\sum_{k=1}^{j-1}h_k, \qquad j=1,\ldots,n, \qquad\text{and}\qquad \overline h:=\sum_{k=1}^{n-1}h_k.
\]
Since $h\in D\subseteq\RR_{\geq 0}^{n-1}$, one has $0\leq s_j\leq\overline h$ for all $j=1,\dots,n$. Since $\mean(\nu)=0$, one has $\int_0^1 Q_\nu(u)\,du=0$ and hence
\[
U_\nu(p)=\int_0^p Q_\nu(u)\,du = -\int_p^1 Q_\nu(u)\,du,
\qquad p\in[0,1].
\]
Let $a>1$ be as in \ref{ass:A3}. By Hölder's inequality, for every $p\in[0,1]$,
\begin{align*}
|U_\nu(p)|
&=\left|\int_0^p Q_\nu(u)\,du\right|
\le \|Q_\nu\|_{L^a(0,1)}\,\|\mathbf 1_{[0,p]}\|_{L^{\frac{a}{a-1}}(0,1)}
= \|Q_\nu\|_{L^a(0,1)}\,p^{\frac{a-1}{a}},\\
|U_\nu(p)|
&=\left|\int_p^1 Q_\nu(u)\,du\right|
\le \|Q_\nu\|_{L^a(0,1)}\,\|\mathbf 1_{[p,1]}\|_{L^{\frac{a}{a-1}}(0,1)}
= \|Q_\nu\|_{L^a(0,1)}\,(1-p)^{\frac{a-1}{a}}.
\end{align*}
Consequently,
\begin{equation}
|U_\nu(p)| \leq \|Q_\nu\|_{L^a(0,1)} \min\bigl\{p^{\frac{a}{a-1}},(1-p)^{\frac{a}{a-1}}\bigr\}, \qquad p\in[0,1]. \label{eq:tails-estim-integrated}
\end{equation}

Since $F_q$ is non-decreasing and $0\leq s_j\leq\overline h$, we have $g(h,z) = \sum_{j=1}^n p_jF_q(z-s_j) \leq F_q(z)$ and $g(h,z) \geq F_q(z-\overline h)$. Thus, it follows from \eqref{eq:tails-estim-integrated} that
\[
|U_\nu(g(h,z))| \leq \|Q_\nu\|_{L^a(0,1)}
\begin{cases}
F_q(z)^{\frac{a}{a-1}}, & z\leq -1,\\ 1, & -1<z<\overline h+1,\\ \bigl(1-F_q(z-\overline h)\bigr)^{\frac{a}{a-1}}, & z\geq\overline h+1.
\end{cases}
\]
Therefore, by the change of variables $w=z-\overline h$ in the
last integral,
\[
\int_{\RR}|U_\nu(g(h,z))|\,dz \leq \|Q_\nu\|_{L^a(0,1)} \Bigg( \overline h+2 +\int_{-\infty}^{-1}F_q(z)^{\frac{a}{a-1}}\,dz +\int_1^\infty\bigl(1-F_q(w)\bigr)^{\frac{a}{a-1}}\,dw \Bigg).
\]

Let $b$ be as in \ref{ass:A3}, and let $X\sim q$. For $z>1$, Markov's inequality applied to the non-negative random variable $|X|^b$ gives
\[
1-F_q(z)=\PP(X>z)\le \PP(|X|>z)\le \frac{\EE[|X|^b]}{z^b}
=\frac{\|Q_q\|_{L^b(0,1)}^b}{z^b}.
\]
For $z<-1$, set $t:=|z|>1$ and note that $\{X\le z\}\subseteq \{|X|\ge t\}$. Hence, again by Markov's inequality,
\[
F_q(z)=\PP(X\le z)\le \PP(|X|\ge t)\le \frac{\EE[|X|^b]}{t^b}
=\frac{\|Q_q\|_{L^b(0,1)}^b}{|z|^b}.
\]
Therefore,
\[
\int_{1}^{\infty} (1-F_q(z))^{\frac{a-1}{a}}\,dz
+ \int_{-\infty}^{-1} (F_q(z))^{\frac{a-1}{a}}\,dz
\le 2\|Q_q\|_{L^b(0,1)}^{\frac{b(a-1)}{a}}
\int_{1}^{\infty} z^{-\frac{b(a-1)}{a}}\,dz<\infty,
\]
since $b>\frac{a}{a-1}$. This proves that $V$ is well-defined on $D$.

Fix $1\le r\le n-1$ and set $H(h):=\sum_{k=1}^{n-1} h_k$. By the change of variables $z\mapsto z+H(h)$,
\[
V(h)=\int_\RR U_\nu(g(h,z+H(h)))\,dz.
\]
For each fixed $z\in\RR$ and fixed $(h_i)_{i\neq r}$, the map $h_r\mapsto U_\nu(g(h,z+H(h)))$ is absolutely
continuous and is differentiable for a.e. $h_r$. Using the chain rule and $\partial_{h_r}H(h)=1$,
\begin{align*}
\frac{\partial}{\partial h_r}U_\nu(g(h,z+H(h)))
&= Q_\nu(g(h,z+H(h)))\left(\frac{\partial}{\partial h_r}g(h,z+H(h))+\frac{\partial}{\partial z}g(h,z+H(h))\right).
\end{align*}
Since
\[
g(h,w)=\sum_{j=1}^n p_j F_q\left(w-\sum_{k=1}^{j-1}h_k\right),
\]
one has
\[
\frac{\partial}{\partial z}g(h,w)=\sum_{j=1}^n p_j\,\rho_q\left(w-\sum_{k=1}^{j-1}h_k\right),
\qquad
\frac{\partial}{\partial h_r}g(h,w)=-\sum_{j=r+1}^n p_j\,\rho_q\left(w-\sum_{k=1}^{j-1}h_k\right),
\]
and therefore
\[
\frac{\partial}{\partial h_r}g(h,w)+\frac{\partial}{\partial z}g(h,w)
=\sum_{j=1}^r p_j\,\rho_q\left(w-\sum_{k=1}^{j-1}h_k\right).
\]
Thus
\[
\frac{\partial}{\partial h_r}U_\nu(g(h,z+H(h)))
=
Q_\nu(g(h,z+H(h)))\sum_{j=1}^r p_j\,\rho_q\left(z+H(h)-\sum_{k=1}^{j-1}h_k\right).
\]

Let $-\infty<a<b<\infty$. Then
\begin{align*}
\int_a^b \int_\RR \Bigg|\frac{\partial}{\partial h_r}U_\nu(g(h,z+H&(h)))\Bigg|\,dz\,dh_r\\
&\le \int_a^b \int_\RR \left|Q_\nu(g(h,z+H(h)))\right|
\sum_{j=1}^n p_j\,\rho_q\left(z+H(h)-\sum_{k=1}^{j-1}h_k\right)\,dz\,dh_r \\
&= \int_a^b \int_\RR \left|Q_\nu(g(h,w))\right|\,\frac{\partial}{\partial z}g(h,w)\,dw\,dh_r \\
&= \int_a^b \int_0^1 |Q_\nu(u)|\,du\,dh_r
\le (b-a)\|Q_\nu\|_{L^1(0,1)}<\infty,
\end{align*}
where we used the change of variables $w=z+H(h)$ and then $u=g(h,w)$, which is admissible since, for each $h\in D$, the map $w\mapsto g(h,w)$ is differentiable.
Therefore, Lemma~\ref{prop:interchange} applies and yields that $V$ is absolutely continuous in $h_r$ and
\begin{align*}
\frac{\partial V(h)}{\partial h_r}
&=\int_\RR \frac{\partial}{\partial h_r}U_\nu(g(h,z+H(h)))\,dz
= \int_\RR Q_\nu(g(h,w))\sum_{j=1}^r p_j\,\rho_q\left(w-\sum_{k=1}^{j-1}h_k\right)\,dw
= f_r(h).
\end{align*}

Now fix $1\le r\le s\le n-1$. For each fixed $z\in\RR$ and fixed $(h_i)_{i\neq s}$, the map
\[
h_s \longmapsto Q_\nu(g(h,z))\sum_{j=1}^r p_j\,\rho_q\left(z-\sum_{k=1}^{j-1}h_k\right)
\]
is absolutely continuous. Moreover, since $j\le r\le s$, the sum $\sum_{j=1}^r p_j\,\rho_q\left(z-\sum_{k=1}^{j-1}h_k\right)$ does not depend on $h_s$, and differentiating yields
\begin{align*}
\frac{\partial}{\partial h_s}\Bigg[ Q_\nu(g(h,z))\sum_{j=1}^r p_j\,&\rho_q \left(z-\sum_{k=1}^{j-1}h_k\right)\Bigg]
= Q_\nu'(g(h,z))\,\frac{\partial g(h,z)}{\partial h_s}\sum_{j=1}^r p_j\,\rho_q\left(z-\sum_{k=1}^{j-1}h_k\right) \\
&= -Q_\nu'(g(h,z))\sum_{j=1}^r\sum_{\ell=s+1}^{n} p_j p_\ell\,
\rho_q\left(z-\sum_{k=1}^{j-1}h_k\right)\rho_q\left(z-\sum_{k=1}^{\ell-1}h_k\right),
\end{align*}
since $\partial_{h_s}g(h,z)=-\sum_{\ell=s+1}^n p_\ell\,\rho_q\left(z-\sum_{k=1}^{\ell-1}h_k\right)$.

To apply Lemma~\ref{prop:interchange} again, it suffices to show that for $-\infty<a<b<\infty$,
\[
\int_a^b \int_\RR |Q_\nu'(g(h,z))|\left[\frac{\partial}{\partial z}g(h,z)\right]^2 dz\,dh_s<\infty,
\]
because
\[
\sum_{j=1}^r\sum_{\ell=s+1}^{n} p_j p_\ell\,
\rho_q\left(z-\sum_{k=1}^{j-1}h_k\right)\rho_q\left(z-\sum_{k=1}^{\ell-1}h_k\right)
\le \left(\sum_{j=1}^n p_j\,\rho_q\left(z-\sum_{k=1}^{j-1}h_k\right)\right)^2
=\left[\frac{\partial}{\partial z}g(h,z)\right]^2.
\]
Using the change of variables $u=g(h,z)$ (so $du=\partial_z g(h,z)\,dz$), we obtain
\[
\int_\RR |Q_\nu'(g(h,z))|\left[\frac{\partial}{\partial z}g(h,z)\right]^2 dz
=
\int_0^1 |Q_\nu'(u)|\,\frac{\partial}{\partial z}g\bigl(h,g^{-1}(h,u)\bigr)\,du.
\]
Applying Hölder's inequality with exponents $\varsigma$ and $\frac{\varsigma}{\varsigma-1}$ yields
\begin{align*}
\int_0^1 |Q_\nu'(u)|\,\frac{\partial}{\partial z}g\bigl(h,g^{-1}(h,u)\bigr)\,du
&\le \|Q_\nu'\|_{L^\varsigma(0,1)}
\left(\int_0^1 \left|\frac{\partial}{\partial z}g\bigl(h,g^{-1}(h,u)\bigr)\right|^{\frac{\varsigma}{\varsigma-1}}du\right)^{\frac{\varsigma-1}{\varsigma}} \\
&= \|Q_\nu'\|_{L^\varsigma(0,1)}
\left(\int_\RR \left|\frac{\partial}{\partial z}g(h,z)\right|^{\frac{2\varsigma-1}{\varsigma-1}}dz\right)^{\frac{\varsigma-1}{\varsigma}}.
\end{align*}
Since $\varsigma>1$, the map $x\mapsto x^{\frac{2\varsigma-1}{\varsigma}}$ is convex and, therefore, we infer
\[
\left\|\frac{\partial}{\partial z}g(h,\cdot)\right\|_{L^{\frac{2\varsigma-1}{\varsigma-1}}(\RR)}^{\frac{2\varsigma-1}{\varsigma}}
\le \sum_{j=1}^n p_j
\left\|\rho_q\left(\cdot-\sum_{k=1}^{j-1}h_k\right)\right\|_{L^{\frac{2\varsigma-1}{\varsigma-1}}(\RR)}^{\frac{2\varsigma-1}{\varsigma}}
=
\|\rho_q\|_{L^{\frac{2\varsigma-1}{\varsigma-1}}(\RR)}^{\frac{2\varsigma-1}{\varsigma}}.
\]
Therefore,
\[
\int_a^b \int_\RR |Q_\nu'(g(h,z))|\left[\frac{\partial}{\partial z}g(h,z)\right]^2 dz\,dh_s
\le (b-a)\|Q_\nu'\|_{L^\varsigma(0,1)}\,
\|\rho_q\|_{L^{\frac{2\varsigma-1}{\varsigma-1}}(\RR)}^{\frac{2\varsigma-1}{\varsigma}}
<\infty,
\]
and Lemma~\ref{prop:interchange} yields that the derivative exists for a.e. $h_s$ and is given by
\[
\frac{\partial^2V(h)}{\partial h_s\partial h_r} = -\sum_{j=1}^r\sum_{\ell=s+1}^n\psi_{j\ell}(h).
\]
where, for $1\le j<\ell\le n$,
\[
\psi_{j\ell}(h)
:= p_j p_\ell \int_\RR Q_\nu'(g(h,z))\,
\rho_q\left(z-\sum_{k=1}^{j-1}h_k\right)\rho_q\left(z-\sum_{k=1}^{\ell-1}h_k\right)\,dz \ge 0.
\]

By Lemma~\ref{lem:psi-continuous}, the mixed second derivatives $\partial_{s}\partial_{r}V$ are continuous on $D$, and \cite[Theorem 9.4]{Rudin1976} implies
\[
\frac{\partial^2 V(h)}{\partial h_r \partial h_s}=\frac{\partial^2 V(h)}{\partial h_s \partial h_r},
\qquad 1\le r,s\le n-1,\ \ h\in D,
\]
so the Hessian matrix of $V$ is symmetric and $V\in C^2(\interior(D))$.

Finally, it follows from the representation of the mixed derivatives that
\[
D^2 V(h) = -\sum_{1 \leq j< \ell\leq n} \psi_{j\ell}(h)\, v^{j\ell}(v^{j\ell})^T,
\]
where $v^{j\ell} \in \RR^{n-1}$ has components
\[
v^{j\ell}_r = \begin{cases}
1, & \text{if } j \leq r< \ell,\\
0, & \text{otherwise}.
\end{cases}
\]
Hence $D^2V(h)$ is negative semidefinite. Moreover, by Assumptions \ref{ass:A1}--\ref{ass:A2}, for each $h\in D$ and $m=1,\dots,n-1$, one has $Q_\nu' > 0$, $h_m \in [0, \diam(\supp(q)))$ and
\begin{align*}
\psi_{m(m+1)}(h)
&= p_m p_{m+1}\int_\RR Q_\nu'(g(h,z))\,
\rho_q\left(z-\sum_{k=1}^{m-1}h_k\right)\rho_q\left(z-\sum_{k=1}^{m}h_k\right)\,dz \\
&= p_m p_{m+1}\int_\RR Q_\nu'\left(g\left(h,z+\sum_{k=1}^{m}h_k\right)\right)\rho_q(z+h_m)\rho_q(z)\,dz>0,
\end{align*}
and therefore $D^2V(h)$ is negative definite. This proves that $V$ is strictly concave. Since $D^2V(h)$ is negative definite for all $h\in \interior(D)$, the inverse function theorem yields that $\nabla V$ is a local $C^1$-diffeomorphism. Additionally, strict concavity implies that $\nabla V$ is injective on $D$, and hence $f|_{\interior(D)}$ is a global diffeomorphism onto $f(\interior(D))$.
\end{proof}

\begin{lemma}\label{adjLemma}
Let $d\in\NN$ and let $H\in \RR^{d\times d}$ be a symmetric positive definite matrix of the form
\[
H=\sum_{1\le j<\ell\le d+1} a_{j\ell}\, v^{j\ell}(v^{j\ell})^T,
\]
where $a_{j\ell}\ge0$ for all $1\le j<\ell\le d+1$, $a_{j(j+1)}>0$ for all $1\le j< d+1$, and $v^{j\ell}\in \RR^{d}$ is the vector whose components satisfy
\[
v^{j\ell}_r=\begin{cases}
1, & \text{if } j \leq r< \ell,\\
0, & \text{otherwise}.
\end{cases}
\]
Let $m=\binom{d+1}{2}$, let $\tau : \{1,\dots,m\} \rightarrow \{(j,\ell): 1\leq j < \ell \leq d+1\}$ be a bijection, and let $W \in \RR^{d\times m}$ be the matrix whose $s$-th column is the vector $v^{\tau(s)}, s=1,\ldots,m$.
Then, for every $i,j\in \{1,\dots,d\}$, \[
\adj(H)_{ij} = (-1)^{i+j} \sum_{\substack{S \subseteq \{1,\dots,m\}\\ |S|=d-1 }} \prod_{s \in S}a_{\tau(s)}\det(W_{\{1,\dots,d\}\setminus\{i\}, S}) \det(W_{\{1,\dots,d\}\setminus\{j\}, S}).
\]
In particular,
\[
\frac{|\adj(H)_{ij}|}{\adj(H)_{jj}}<1 \quad \text{for } i \not = j.
\]
\end{lemma}

\begin{proof}
By construction,
\[
H=\sum_{s=1}^m a_{\tau(s)}\, v^{\tau(s)}(v^{\tau(s)})^T
= W \diag(a_{\tau(1)},a_{\tau(2)}, \dots, a_{\tau(m)}) W^T.
\]
Using $\adj(H)_{ij}=(-1)^{i+j}\det\!\bigl(H_{\{1,\dots,d\}\setminus\{i\},\{1,\dots,d\}\setminus\{j\}}\bigr)$, we obtain
\[
\adj(H)_{ij}
= (-1)^{i+j}\det \!\left(W_{\{1,\dots,d\}\setminus \{i\}, \{1,\dots,m\}}\diag(a_{\tau(1)},a_{\tau(2)}, \dots, a_{\tau(m)})W_{\{1,\dots,d\}\setminus \{j\}, \{1,\dots,m\}}^T\right).
\]
Applying the Cauchy--Binet formula gives
\begin{align*}
	\adj(H)_{ij}
	 & = (-1)^{i+j} \sum_{\substack{S \subseteq \{1,\dots,m\}\\ |S|=d-1 }} \det \!\left(W_{\{1,\dots,d\}\setminus \{i\}, S}\,\,\diag((a_{\tau(s)})_{s\in S}) \,\, W_{\{1,\dots,d\}\setminus \{j\}, S}^T\right)\\
	& = (-1)^{i+j} \sum_{\substack{S \subseteq \{1,\dots,m\}\\ |S|=d-1 }} \,\,\prod_{s \in S}a_{\tau(s)}\det(W_{\{1,\dots,d\}\setminus\{i\}, S}) \det(W_{\{1,\dots,d\}\setminus\{j\}, S}).
\end{align*}

Every column of $W$ is the indicator of a discrete interval, hence has the form \[(0,\dots,0,1,\dots,1,0,\dots,0)^T\] with consecutive ones. Matrices with the consecutive-ones property are totally unimodular. In particular, for any $I\subseteq \{1,\dots,d\}$ and any $S\subseteq \{1,\dots,m\}$ with $|I|=|S|$,
\[
\det(W_{I,S})\in\{0,\pm 1\}.
\]
Fix $i\neq j$ in $\{1,\dots,d\}$. Using the expansion above, we obtain
\begin{align*}
|\adj(H)_{ij}|
&\le \sum_{\substack{S \subseteq \{1,\dots,m\}\\ |S|=d-1 }} \prod_{s \in S}a_{\tau(s)}
\left|\det(W_{\{1,\dots,d\}\setminus\{i\}, S}) \det(W_{\{1,\dots,d\}\setminus\{j\}, S})\right| \\
&\le \sum_{\substack{S \subseteq \{1,\dots,m\}\\ |S|=d-1 }} \prod_{s \in S}a_{\tau(s)}\det(W_{\{1,\dots,d\}\setminus\{j\}, S})^2
= \adj(H)_{jj},
\end{align*}
where the last identity is the same Cauchy--Binet expansion with $i=j$.

To see that the inequality is strict, consider the subset
\[
S_j:=\{\,s\in\{1,\dots,m\}:\tau(s)=(r,r+1)\ \text{for some }r\in\{1,\dots,d\}\setminus\{j\}\,\}.
\]
Then $|S_j|=d-1$ and the columns $\{v^{r(r+1)}:r\in\{1,\dots,d\}\setminus\{j\}\}$ form a permutation of the standard basis of $\RR^{d-1}$, so
\[
\det(W_{\{1,\dots,d\}\setminus\{j\},S_j})^2=1.
\]
and,  by the assumption $a_{r,r+1}>0$, one has
$a_{\tau(s)}>0$ for every $s\in S_j$.

Since $i\neq j$, the set $S_j$ contains the unique index $s_i$ with $\tau(s_i)=(i,i+1)$, and the corresponding column of $W_{\{1,\dots,d\}\setminus\{i\},S_j}$ is the zero vector, because $v^{i(i+1)}=e_i$ and the $i$-th row is removed. Hence, one has
\[
\det(W_{\{1,\dots,d\}\setminus\{i\},S_j})=0.
\]
Thus $S_j$ contributes the strictly positive term $\prod_{s\in S_j}a_{\tau(s)}$ to $\adj(H)_{jj}$, while it contributes $0$ to $|\adj(H)_{ij}|$. Consequently $|\adj(H)_{ij}|<\adj(H)_{jj}$, which proves the claim.
\end{proof}
\medskip

\begin{proof}[Proof of Proposition \ref{prop:q-Bass-map-stability}]
Let $(h^{(k)})_{k \in \NN} \subseteq \RR^{n-1}_{\geq 0}$ be an arbitrary sequence, independent of $(\nu_k)_{k \in \NN}$ and $(q_k)_{k \in \NN}$, where $(\nu_k)_{k \in \NN} \subseteq \Prob_1(\RR)$ and $(q_k)_{k \in \NN} \subseteq \Prob(\RR)$ satisfy the assumptions of Proposition~\ref{prop:q-Bass-map-stability}.
By the dominated convergence theorem, $(Q_{\nu_k})_{k \in \NN}$ converges to $Q_\nu$ in $L^1(0,1)$. Moreover,
\[
    \sup_{x\in\RR}|F_{q_k}(x)-F_q(x)| \leq \|\rho_{q_k}-\rho_q\|_{L^1(\RR)} \longrightarrow 0 .
\]
By a change of variables, for every $\psi\in L^1(0,1)$,
\begin{equation} \label{eq:change-of-variables-continuouty}
    \int_{\RR} \psi\left(\sum_{r=1}^{m}p_r F_{q_k}\left( z-\sum_{\ell=1}^{r-1}h_\ell^{(k)} \right) \right) \sum_{r=1}^{m}p_r \rho_{q_k}\left( z-\sum_{\ell=1}^{r-1}h_\ell^{(k)} \right)\,dz = \int_0^{p_m^*}\psi(u)\,du \leq \|\psi\|_{L^1(0,1)}.
\end{equation}
The same bound holds with $q_k,h^{(k)}$ replaced by $q,h$.  
\medskip

\noindent\emph{Step 1.} Assume first that $(h^{(k)})_{k\in\NN}$ is bounded. Passing to a subsequence, we may suppose that $h^{(k)}\rightarrow h$ for some $h\in\RR_{\geq0}^{n-1}$.
We prove that $f_i^{(k)}(h^{(k)}) \rightarrow f_i(h)$, for every $i\in\{1,\dots,n-1\}$. Fix $i\in\{1,\dots,n-1\}$. Let $\varphi\in C_b([0,1])$. For every $r\in\{1,\dots,n\}$,
\begin{align*}
        \sup_{z\in\RR} \Bigg| F_{q_k}\left(z-\sum_{\ell=1}^{r-1}h_\ell^{(k)} \right) & - F_q\left( z-\sum_{\ell=1}^{r-1}h_\ell \right) \Bigg|
        \\
        & \leq \sup_{x\in\RR}|F_{q_k}(x)-F_q(x)| +\sup_{z\in\RR} \left| F_q\left( z-\sum_{\ell=1}^{r-1}h_\ell^{(k)} \right) - F_q\left( z-\sum_{\ell=1}^{r-1}h_\ell \right) \right|.
\end{align*}
As $k\to\infty$, the first term tends to $0$, and so does the second one, since $F_q$ is uniformly continuous.
Moreover, we have
\begin{align*}
    \Bigg\|\sum_{r=1}^{i}p_r & \rho_{q_k}\left( \cdot-\sum_{\ell=1}^{r-1}h_\ell^{(k)} \right)  - \sum_{r=1}^{i}p_r \rho_q\left(
    \cdot-\sum_{\ell=1}^{r-1}h_\ell \right) \Bigg\|_{L^1(\RR)} \\
    & \leq \sum_{r=1}^{i}p_r \left\| \rho_{q_k}\left( \cdot-\sum_{\ell=1}^{r-1}h_\ell^{(k)} \right) - \rho_q\left( \cdot-\sum_{\ell=1}^{r-1}h_\ell^{(k)} \right) \right\|_{L^1(\RR)} \\
    &\qquad + \sum_{r=1}^{i}p_r \left\| \rho_q\left( \cdot-\sum_{\ell=1}^{r-1}h_\ell^{(k)} \right) - \rho_q\left( \cdot-\sum_{\ell=1}^{r-1}h_\ell \right) \right\|_{L^1(\RR)} \longrightarrow 0 .
\end{align*}
Indeed, the first sum tends to $0$ since $\rho_{q_k}\to\rho_q$ in $L^1(\RR)$, and the second one tends to $0$ by continuity of translations in $L^1(\RR)$.
Since $\varphi$ is bounded and uniformly continuous on $[0,1]$, we obtain
\begin{equation} \label{eq:conv-continuoty-Cb}
    \begin{aligned}
    \int_{\RR} \varphi &\left( \sum_{r=1}^{n}p_r F_{q_k}\left( z-\sum_{\ell=1}^{r-1}h_\ell^{(k)}\right)\right) \sum_{r=1}^{i}p_r \rho_{q_k}\left( z-\sum_{\ell=1}^{r-1}h_\ell^{(k)} \right)\,dz \\
    &\qquad\qquad\qquad\longrightarrow \int_{\RR} \varphi\left(\sum_{r=1}^{n}p_r F_q\left( z-\sum_{\ell=1}^{r-1}h_\ell \right) \right) \sum_{r=1}^{i}p_r \rho_q\left( z-\sum_{\ell=1}^{r-1}h_\ell \right)\,dz .
\end{aligned}
\end{equation}

We now pass from bounded continuous functions $\varphi$ to $Q_\nu$. Let $\varepsilon>0$, and choose $\varphi\in C_b([0,1])$ such that $\|Q_\nu-\varphi\|_{L^1(0,1)}<\varepsilon$. By \eqref{eq:change-of-variables-continuouty}, we get
\begin{align*}
    \int_{\RR} \Bigg| Q_\nu\left( \sum_{r=1}^{n}p_r F_{q_k}\left( z-\sum_{\ell=1}^{r-1}h_\ell^{(k)} \right) \right) - & \varphi\left( \sum_{r=1}^{n}p_r F_{q_k}\left( z-\sum_{\ell=1}^{r-1}h_\ell^{(k)} \right) \right) \Bigg| \\
    &\qquad \qquad\cdot \sum_{r=1}^{i}p_r \rho_{q_k}\left( z-\sum_{\ell=1}^{r-1}h_\ell^{(k)} \right)\,dz \leq \varepsilon .
\end{align*}
The same estimate holds with $q_k,h^{(k)}$ replaced by $q,h$. Combining these two estimates with \eqref{eq:conv-continuoty-Cb}, and then letting $\varepsilon\downarrow0$, gives
\begin{equation} \label{eq:conv-continuoty-Q-nu}
    \begin{aligned}
    \int_{\RR} Q_\nu &\left( \sum_{r=1}^{n}p_r F_{q_k}\left( z-\sum_{\ell=1}^{r-1}h_\ell^{(k)}\right)\right) \sum_{r=1}^{i}p_r \rho_{q_k}\left( z-\sum_{\ell=1}^{r-1}h_\ell^{(k)} \right)\,dz \\
    &\qquad\qquad\qquad\longrightarrow \int_{\RR} Q_\nu\left(\sum_{r=1}^{n}p_r F_q\left( z-\sum_{\ell=1}^{r-1}h_\ell \right) \right) \sum_{r=1}^{i}p_r \rho_q\left( z-\sum_{\ell=1}^{r-1}h_\ell \right)\,dz .
\end{aligned}
\end{equation}
Finally, again by \eqref{eq:change-of-variables-continuouty},
\[
\begin{aligned}
    &\left| \int_{\RR} (Q_{\nu_k}-Q_\nu)\left( \sum_{r=1}^{n}p_r F_{q_k}\left(z-\sum_{\ell=1}^{r-1}h_\ell^{(k)} \right) \right) \sum_{r=1}^{i}p_r \rho_{q_k}\left(z-\sum_{\ell=1}^{r-1}h_\ell^{(k)} \right)\,dz \right|  \leq \|Q_{\nu_k}-Q_\nu\|_{L^1(0,1)} \rightarrow 0 .
\end{aligned}
\]
Together with \eqref{eq:conv-continuoty-Q-nu}, this proves $f_i^{(k)}(h^{(k)})\longrightarrow f_i(h)$.
Since $i\in\{1,\dots,n-1\}$ was arbitrary,  case $(i)$ in the statement of Proposition~\ref{prop:q-Bass-map-stability} holds.
\medskip

\noindent\emph{Step 2.}  Assume now that $(h^{(k)})_{k\in\NN}$ is unbounded. Since all components are nonnegative, there exist a subsequence, still denoted by $(h^{(k)})_{k\in\NN}$, and some index $i\in\{1,\dots,n-1\}$, such that $h_i^{(k)}\rightarrow +\infty$. 
In this case, we prove that $f_i^{(k)}(h^{(k)})\rightarrow U_\nu(p_i^*)$. By \eqref{eq:change-of-variables-continuouty} with $m=i$ and $\psi=Q_{\nu_k}$,
\begin{equation}
\label{eq:continuous-partial-result-2}
\begin{aligned}
    &\int_{\RR} Q_{\nu_k}\left(\sum_{r=1}^{i}p_r F_{q_k}\left( z-\sum_{\ell=1}^{r-1}h_\ell^{(k)} \right) \right) \sum_{r=1}^{i}p_r \rho_{q_k}\left( z-\sum_{\ell=1}^{r-1}h_\ell^{(k)} \right)\,dz \\
    &\qquad = \int_0^{p_i^*}Q_{\nu_k}(u)\,du \longrightarrow \int_0^{p_i^*}Q_\nu(u)\,du = U_\nu(p_i^*).
\end{aligned}
\end{equation}
It remains to replace the sum up to $i$ inside $Q_{\nu_k}$ by the sum up to $n$. Fix $r>i$ and $j\leq i$. By the change of variable $x=z-\sum_{\ell=1}^{j-1}h_\ell^{(k)}$, we have
\begin{equation}\label{eq.fqk}
    \int_{\RR} F_{q_k}\left( z-\sum_{\ell=1}^{r-1}h_\ell^{(k)} \right) \rho_{q_k}\left( z-\sum_{\ell=1}^{j-1}h_\ell^{(k)} \right)\,dz = \int_{\RR} F_{q_k}\left(x-\sum_{\ell=j}^{r-1}h_\ell^{(k)} \right) \rho_{q_k}(x)\,dx .
\end{equation}
Since $j\leq i<r$, one has $\sum_{\ell=j}^{r-1}h_\ell^{(k)} \geq h_i^{(k)} \longrightarrow+\infty$. Therefore,
\begin{align}
\begin{split}\label{eq.fqkq}
    &\int_{\RR} F_{q_k}\left(x-\sum_{\ell=j}^{r-1}h_\ell^{(k)} \right)\rho_{q_k}(x)\,dx \\
    &\qquad\quad \leq \sup_{y\in\RR}|F_{q_k}(y)-F_q(y)|+\|\rho_{q_k}-\rho_q\|_{L^1(\RR)}
    +\int_{\RR} F_q\left(x-\sum_{\ell=j}^{r-1}h_\ell^{(k)} \right)\rho_q(x)\,dx .
    \end{split}
\end{align}
As $k\to\infty$, the first two terms tend to $0$. The last term also tends to $0$, by the dominated convergence theorem, since $F_q\left(x-\sum_{\ell=j}^{r-1}h_\ell^{(k)} \right)$ tends to $0$ for every $x \in \RR$ and it is bounded by $1$, while $\rho_q\in L^1(\RR)$. By \eqref{eq.fqk} and \eqref{eq.fqkq}, this yields
\begin{align*}
    \int_{\RR} F_{q_k}\left( z-\sum_{\ell=1}^{r-1}h_\ell^{(k)} \right)\rho_{q_k}\left( z-\sum_{\ell=1}^{j-1}h_\ell^{(k)} \right)\,dz \longrightarrow 0 .
\end{align*}
Since there are only finitely many pairs $j\leq i<r$, it follows that
\begin{equation}
\label{eq:integral-sum-CDFs}
\begin{aligned}
    &\int_{\RR} \left[ \sum_{r=i+1}^{n}p_r F_{q_k}\left( z-\sum_{\ell=1}^{r-1}h_\ell^{(k)} \right) \right] \left[ \sum_{j=1}^{i}p_j \rho_{q_k}\left( z-\sum_{\ell=1}^{j-1}h_\ell^{(k)} \right) \right]\,dz \longrightarrow 0 .
\end{aligned}
\end{equation}
Let $\varphi\in C_b([0,1])$. By uniform continuity of $\varphi$, \eqref{eq:change-of-variables-continuouty}, and \eqref{eq:integral-sum-CDFs}, we get
\begin{equation}
\label{eq:result-test-functions}
\begin{aligned}
    &\int_{\RR} \Bigg| \varphi\left( \sum_{r=1}^{n}p_r F_{q_k}\left( z-\sum_{\ell=1}^{r-1}h_\ell^{(k)} \right) \right) - \varphi\left( \sum_{r=1}^{i}p_r F_{q_k}\left( z-\sum_{\ell=1}^{r-1}h_\ell^{(k)} \right) \right) \Bigg| \\
    &\qquad\qquad\cdot \sum_{j=1}^{i}p_j \rho_{q_k}\left(z-\sum_{\ell=1}^{j-1}h_\ell^{(k)} \right)\,dz \longrightarrow 0 .
\end{aligned}
\end{equation}
Indeed, for any $\eta>0$, choose $\delta>0$ such that $|\varphi(x)-\varphi(y)|<\eta$ whenever $|x-y|<\delta$. The difference between the two arguments of $\varphi$ is $\sum_{r=i+1}^{n}p_r F_{q_k}\left( z-\sum_{\ell=1}^{r-1}h_\ell^{(k)} \right)$. On the set where this quantity is smaller than $\delta$, the integrand is bounded by $\eta$ times the density factor. On the complementary set, we use the bound $2\|\varphi\|_\infty$ and the inequality
\begin{align*}
    \mathbf 1_{\left\{\sum_{r=i+1}^{n}p_r F_{q_k}\left( z-\sum_{\ell=1}^{r-1}h_\ell^{(k)} \right)\geq \delta\right\}}
    \leq
    \frac{1}{\delta}
    \sum_{r=i+1}^{n}p_r F_{q_k}\left( z-\sum_{\ell=1}^{r-1}h_\ell^{(k)} \right).
\end{align*}
Therefore, the integral in \eqref{eq:result-test-functions} is bounded by
\begin{align*}
    &\eta p_i^* + \frac{2\|\varphi\|_\infty}{\delta}\int_{\RR} \left[ \sum_{r=i+1}^{n}p_r F_{q_k}\left( z-\sum_{\ell=1}^{r-1}h_\ell^{(k)} \right) \right] \left[ \sum_{j=1}^{i}p_j \rho_{q_k}\left( z-\sum_{\ell=1}^{j-1}h_\ell^{(k)} \right) \right]\,dz,
\end{align*}
and the second term tends to $0$ by \eqref{eq:integral-sum-CDFs}. Since $\eta>0$ is arbitrary, \eqref{eq:result-test-functions} follows. 

As done in Step 1, we now pass from bounded continuous functions $\varphi$ to $Q_{\nu_k}$. Let $\varepsilon>0$, and choose $\varphi\in C_b([0,1])$ such that $\|Q_\nu-\varphi\|_{L^1(0,1)}<\varepsilon$.
We decompose the difference by adding and subtracting the same two terms with $\varphi$ in place of $Q_{\nu_k}$, and then with $Q_\nu$ in place of $Q_{\nu_k}$. By \eqref{eq:change-of-variables-continuouty},
\begin{align*}
    &\int_{\RR} \left| (Q_{\nu_k}-Q_\nu)\left( \sum_{r=1}^{s}p_r F_{q_k}\left( z-\sum_{\ell=1}^{r-1}h_\ell^{(k)} \right) \right) \right| \sum_{j=1}^{i}p_j \rho_{q_k}\left( z-\sum_{\ell=1}^{j-1}h_\ell^{(k)} \right)\,dz \leq \|Q_{\nu_k}-Q_\nu\|_{L^1(0,1)},
\end{align*}
for $s=i, n$.
Similarly, replacing $Q_{\nu_k}-Q_\nu$ by $Q_\nu-\varphi$, the corresponding two terms are bounded by $\varepsilon$. Therefore, using \eqref{eq:result-test-functions}, we obtain
\begin{align*}
    &\limsup_{k\to\infty} \left| \int_{\RR} Q_{\nu_k}\left( \sum_{r=1}^{n}p_r F_{q_k}\left( z-\sum_{\ell=1}^{r-1}h_\ell^{(k)}\right) \right) \sum_{j=1}^{i}p_j \rho_{q_k}\left(z-\sum_{\ell=1}^{j-1}h_\ell^{(k)} \right)\,dz \right. \\
    &\qquad\left. - \int_{\RR} Q_{\nu_k}\left( \sum_{r=1}^{i}p_r F_{q_k}\left( z-\sum_{\ell=1}^{r-1}h_\ell^{(k)} \right) \right) \sum_{j=1}^{i}p_j \rho_{q_k}\left( z-\sum_{\ell=1}^{j-1}h_\ell^{(k)} \right)\,dz \right| \leq 2\varepsilon .
\end{align*}
Since $\varepsilon>0$ was arbitrary,
\begin{equation}
\label{eq:continuous-final-result-2}
\begin{aligned}
    &\int_{\RR} Q_{\nu_k}\left( \sum_{r=1}^{n}p_r F_{q_k}\left( z-\sum_{\ell=1}^{r-1}h_\ell^{(k)} \right) \right) \sum_{j=1}^{i}p_j \rho_{q_k}\left( z-\sum_{\ell=1}^{j-1}h_\ell^{(k)} \right)\,dz \\
    &\qquad - \int_{\RR} Q_{\nu_k}\left( \sum_{r=1}^{i}p_r F_{q_k}\left( z-\sum_{\ell=1}^{r-1}h_\ell^{(k)} \right) \right) \sum_{j=1}^{i}p_j \rho_{q_k}\left( z-\sum_{\ell=1}^{j-1}h_\ell^{(k)} \right)\,dz \longrightarrow 0 .
\end{aligned}
\end{equation}
Combining \eqref{eq:continuous-partial-result-2} and \eqref{eq:continuous-final-result-2}, we conclude that $f_i^{(k)}(h^{(k)})\longrightarrow U_\nu(p_i^*)$, so that  case $(ii)$ in the statement of Proposition~\ref{prop:q-Bass-map-stability} holds.
\end{proof}

\section{Approximation results}
In this section, we prove approximation results for the final marginal and the reference measure.
\begin{proposition}[Approximation of the final marginal]\label{prop:approx_nu_1}
Let $\mu,\nu\in \Prob_1(\RR)$ such that $\mu$ is $n$-atomic with representation \eqref{n_atoms_repr}, and $(\mu,\nu)$ is irreducible.
Then there exists a sequence $(\nu_k)_{k \in \NN}\subseteq \Prob_1(\RR)$ such that, for all $k \in \NN$,
\begin{enumerate}[label=(\roman*)]
\item\label{prop:approx_nu_1:ac} $\nu_k \ll \lambda$ and $\supp(\nu_k)$ is a bounded interval,
\item\label{prop:approx_nu_1:lb} $\rho_{\nu_k}$ is bounded away from $0$ on $\supp(\nu_k)$,
\item\label{prop:approx_nu_1:cx}  $(\mu,\nu_k)$ is irreducible,
\item\label{prop:approx_nu_1:pt} $Q_{\nu_k}\to Q_\nu$ pointwise a.e. on $(0,1)$ as $k\to\infty$,
\item\label{prop:approx_nu_1:dom} there exists $Q\in L^1(0,1)$ such that $|Q_{\nu_k}(p)|\le Q(p)$, for all $p\in(0,1)$.
\end{enumerate}
\end{proposition}

\begin{proof}
Without loss of generality, assume that $\mean(\mu)=\mean(\nu)=0$. If $\nu=\delta_0$, then the convex-order relation implies that $\mu=\delta_0$. In this case, we may choose $\nu_k$ to be a symmetrically truncated centered Gaussian distribution with variance $1/k$. Hence, in the remainder of the proof, we may assume that $\nu\neq\delta_0$.
Since $\mu$ is $n$-atomic, $Q_\mu$ is constant on $(p_{i-1}^*,p_i^*]$ and $U_\mu(p):=\int_0^p Q_\mu$ is affine on each $[p_{i-1}^*,p_i^*]$.
By irreducibility of $(\mu,\nu)$ and \eqref{n_atoms_repr},
\[
U_\mu(p_i^*)>U_\nu(p_i^*)\qquad\text{for all }i=1,\dots,n-1,
\]
hence
\begin{equation*}
\vartheta:=\min_{1\le i\le n-1}\bigl(U_\mu(p_i^*)-U_\nu(p_i^*)\bigr)>0.
\end{equation*}
If $\nu=\delta_0$, then $\mu=\delta_0$ and the choice $\nu_k:=\text{Unif}_{[-1/k,\,1/k]}$ satisfies \ref{prop:approx_nu_1:ac}--\ref{prop:approx_nu_1:dom}, so assume $\nu\neq\delta_0$.
Since $Q_\nu\in L^1(0,1)$ and $\int_0^1 Q_\nu=0$, there exist $p_-\in(0,1)$ and $p_+\in(0,1)$ such that $Q_\nu(p_-)\le0\le Q_\nu(p_+)$. More specifically, we set
\[
p_-:=\sup\{p\in(0,1):Q_\nu(p)\le0\}\in(0,1),\qquad p_+:=\inf\{p\in(0,1):Q_\nu(p)\ge0\}\in(0,1).
\]
For $k\ge1$, we define
\[
\alpha_k:=\sup\left\{\alpha\in(0,\min(p_-, 2^{-k})]:\int_0^\alpha |Q_\nu(u)|\,du\le 2^{-k}\right\} \in (0, 2^{-k}],
\]
\[
\beta_k:=\sup\left\{\beta\in(0,\min(1-p_+,  2^{-k})]:\int_{1-\beta}^1 |Q_\nu(u)|\,du\le 2^{-k}\right\} \in (0, 2^{-k}].
\]

Then $(\alpha_k)_{k \in \NN}$ and $(\beta_k)_{k \in \NN}$ are non-increasing and satisfy $\alpha_k\downarrow0$, $\beta_k\downarrow0$. In particular, $Q_\nu(\alpha_k)\le0$ and $Q_\nu(1-\beta_k)\ge0$.
For $\alpha,\beta\in(0,1)$, we define
\[
D(\alpha):=\alpha\,Q_\nu(\alpha)-\int_0^\alpha Q_\nu(u)\,du, \qquad E(\beta):=\int_{1-\beta}^1 Q_\nu(u)\,du-\beta\,Q_\nu(1-\beta),
\]
and set $D_k:=D(\alpha_k)$, $E_k:=E(\beta_k)$, $s_k:=D_k-E_k$.
Since $Q_\nu$ is non-decreasing, for $u\in(0,\alpha_k]$ one has $Q_\nu(u)\le Q_\nu(\alpha_k)\le0$ and hence
\[
0\le D_k=\int_0^{\alpha_k}\bigl(Q_\nu(\alpha_k)-Q_\nu(u)\bigr)\,du\le \int_0^{\alpha_k}|Q_\nu(u)|\,du\le 2^{-k},
\]
and similarly $0\le E_k\le \int_{1-\beta_k}^1|Q_\nu(u)|\,du\le 2^{-k}$, so $|s_k|\le D_k+E_k\le 2^{1-k}$.
Define
\[
a_k:=Q_\nu(\alpha_k)-\frac{(s_k)_+}{\alpha_k},\qquad b_k:=Q_\nu(1-\beta_k)+\frac{(-s_k)_+}{\beta_k}.
\]
Let $\widehat\nu_k$ be the law with quantile
\[
Q_{\widehat\nu_k}(p)=
\begin{cases}
a_k, & p\in(0,\alpha_k],\\[1mm]
Q_\nu(p), & p\in(\alpha_k,\,1-\beta_k),\\[1mm]
b_k, & p\in[1-\beta_k,\,1).
\end{cases}
\]
Then, using $\int_0^1 Q_\nu=0$ and the identities
\[
D(\alpha_k)=\int_0^{\alpha_k}\bigl(Q_\nu(\alpha_k)-Q_\nu(u)\bigr)\,du,
\qquad
E(\beta_k)=\int_{1-\beta_k}^1\bigl(Q_\nu(u)-Q_\nu(1-\beta_k)\bigr)\,du,
\]
we compute
\begin{align*}
\int_0^1 Q_{\widehat\nu_k}(u)\,du & =\int_0^1 Q_\nu(u)\,du +\int_0^{\alpha_k}\bigl(a_k-Q_\nu(u)\bigr)\,du +\int_{1-\beta_k}^1\bigl(b_k-Q_\nu(u)\bigr)\,du\\
& = \Bigl(D_k- (s_k)_+\Bigr)+\Bigl(-E_k+(-s_k)_+\Bigr)\\
& = (D_k-E_k)-(s_k)_+ +(-s_k)_+\\
& = s_k-(s_k)_+ +(-s_k)_+=0,
\end{align*}
hence $\mean(\widehat\nu_k)=0$.
For each fixed $p\in(0,1)$ one has $Q_{\widehat\nu_k}(p)=Q_\nu(p)$ for all $k$ such that $\alpha_k<p<1-\beta_k$, hence $Q_{\widehat\nu_k}(p)\to Q_\nu(p)$.

Set $\varepsilon_k:=2^{-k}\wedge \tfrac14$ and define
\[
\overline\nu_k:=\widehat\nu_k * \text{Unif}_{[-\varepsilon_k,\varepsilon_k]}.
\]
Then $\overline\nu_k\ll\lambda$ and $\mean(\overline\nu_k)=0$.
If $X\sim\widehat\nu_k$ and $Z\sim\text{Unif}_{[-\varepsilon_k,\varepsilon_k]}$ are independent, then $X-\varepsilon_k\le X+Z\le X+\varepsilon_k$ a.s., hence for all $p\in(0,1)$,
\begin{equation}\label{eq:quantile_sandwich}
Q_{\widehat\nu_k}(p)-\varepsilon_k\le Q_{\overline\nu_k}(p)\le Q_{\widehat\nu_k}(p)+\varepsilon_k.
\end{equation}
Set $M_k:=\max\{|a_k|,|b_k|\}+\varepsilon_k$, so that $\supp(\overline\nu_k)\subset[-M_k,M_k]$, and define
\[
\delta_k:=\min\Bigl\{\frac18,\ \frac{2^{-k}}{8(M_k+1)}\Bigr\}\in\Bigl(0,\frac18\Bigr].
\]
Finally, define
\begin{equation}\label{eq:def_nuk}
\nu_k := (1-\delta_k)\,\overline\nu_k + \delta_k\,\text{Unif}_{[-M_k,M_k]}.
\end{equation}
Then $\nu_k\ll\lambda$, $\supp(\nu_k)=[-M_k,M_k]$ and, for every $x\in[-M_k,M_k]$,
\[
\rho_{\nu_k}(x)=(1-\delta_k)\rho_{\overline\nu_k}(x)+\delta_k\,\rho_{\text{Unif}_{[-M_k,M_k]}}(x)
\ge \delta_k\,\frac{1}{2M_k}>0,
\]
so \ref{prop:approx_nu_1:ac} and \ref{prop:approx_nu_1:lb} hold.

Fix $p\in(0,1)$. Since $\alpha_k\downarrow0$, $\beta_k\downarrow0$, and $\varepsilon_k\downarrow0$, from \eqref{eq:quantile_sandwich} and the fact that $Q_{\widehat\nu_k}(p)=Q_\nu(p)$ for all large $k$ we obtain $Q_{\overline\nu_k}(p)\to Q_\nu(p)$.
From \eqref{eq:def_nuk} we have $(1-\delta_k)F_{\overline\nu_k}(x)\le F_{\nu_k}(x)\le (1-\delta_k)F_{\overline\nu_k}(x)+\delta_k$, hence
\begin{equation}\label{eq:quantile_mixture_sandwich}
Q_{\overline\nu_k}\!\left(\frac{(p-\delta_k)_+}{1-\delta_k}\right)
\le Q_{\nu_k}(p)
\le Q_{\overline\nu_k}\!\left(\frac{p}{1-\delta_k}\right).
\end{equation}

Now, fix a continuity point $p\in(0,1)$ of $Q_\nu$. For every sufficiently small $\epsilon>0$ and all sufficiently large $k$,
\[
p-\epsilon\leq \frac{(p-\delta_k)_+}{1-\delta_k} \leq \frac{p}{1-\delta_k}\leq p+\epsilon.
\]
Hence, by \eqref{eq:quantile_mixture_sandwich} and the monotonicity of $Q_{\overline\nu_k}$,
\[
Q_{\overline\nu_k}(p-\epsilon)\leq Q_{\nu_k}(p) \leq Q_{\overline\nu_k}(p+\epsilon).
\]
Letting first $k\to\infty$ and then $\epsilon\downarrow0$ yields $Q_{\nu_k}(p)\to Q_\nu(p)$. Since $Q_\nu$ is continuous a.e., this proves~\ref{prop:approx_nu_1:pt}.

To prove \ref{prop:approx_nu_1:cx}, we use that for any $\eta,\xi\in\Prob_1(\RR)$,
\begin{equation}\label{eq:U_L1_Q}
|U_{\eta}(p)-U_{\xi}(p)|
=\left|\int_0^p\bigl(Q_\eta(u)-Q_\xi(u)\bigr)\,du\right|
\le \int_0^1 |Q_\eta(u)-Q_\xi(u)|\,du,\qquad p\in[0,1].
\end{equation}
On $(\alpha_k,1-\beta_k)$ we have $Q_{\widehat\nu_k}=Q_\nu$. On $(0,\alpha_k]$ and $[1-\beta_k,1)$, using the definitions of $a_k,b_k$,
\[
\int_0^{\alpha_k}|Q_\nu(u)-Q_{\widehat\nu_k}(u)|\,du \le D_k + (s_k)_+ \le D_k+|s_k|,
\]
\[
\int_{1-\beta_k}^1|Q_\nu(u)-Q_{\widehat\nu_k}(u)|\,du \le E_k + (-s_k)_+ \le E_k+|s_k|,
\]
hence $\int_0^1|Q_\nu(u)-Q_{\widehat\nu_k}(u)|du\le D_k+E_k+2|s_k|\le 6\cdot 2^{-k}$. Next, \eqref{eq:quantile_sandwich} gives pointwise
$|Q_{\widehat\nu_k}(u)-Q_{\overline\nu_k}(u)|\le \varepsilon_k$, hence
\[
\int_0^1|Q_{\widehat\nu_k}(u)-Q_{\overline\nu_k}(u)|du\le \varepsilon_k.
\]
Finally, by definition \eqref{eq:def_nuk}, $\nu_k$ is obtained from $\overline\nu_k$ by replacing a draw from $\overline\nu_k$ with probability $\delta_k$ by a draw from $\mathrm{Unif}_{[-M_k,M_k]}$: since both laws are supported in $[-M_k,M_k]$, this replacement changes the value by at most $2M_k$, hence
\[
\int_0^1|Q_{\overline\nu_k}(u)-Q_{\nu_k}(u)|du = \EE[|Q_{\overline\nu_k}(U)-Q_{\nu_k}(U)|]\le 2M_k\delta_k,\qquad \text{where } U \sim \text{Unif}_{[0,1]}.
\]
Combining the three bounds yields
\[
\int_0^1 |Q_{\nu_k}-Q_\nu|
\le \int_0^1 |Q_\nu-Q_{\widehat\nu_k}| + \int_0^1 |Q_{\widehat\nu_k}-Q_{\overline\nu_k}| + \int_0^1 |Q_{\overline\nu_k}-Q_{\nu_k}|
\le 6 \cdot 2^{-k}+\varepsilon_k+2M_k\delta_k
\le 8\cdot 2^{-k},
\]
where we used $\varepsilon_k\le 2^{-k}$ and $2M_k\delta_k\le 2M_k\cdot \frac{2^{-k}}{8(M_k+1)}\le 2^{-k}/4$.

Let $K:=\min\{m\ge1:8\cdot 2^{-m}\le \vartheta/2\}$ and redefine the final sequence by $\nu_k:=\nu_{k+K}$. We apply the same procedure to all relevant sequences. Then for every $k\ge1$,
\[
\int_0^1|Q_{\nu_k}-Q_\nu|\le \vartheta/2\quad\Rightarrow\quad
\sup_{p\in[0,1]}|U_{\nu_k}(p)-U_\nu(p)|\le \vartheta/2
\]
by \eqref{eq:U_L1_Q}.
Therefore, for every $i=1,\dots,n-1$,
\[
U_\mu(p_i^*)-U_{\nu_k}(p_i^*)
\ge \bigl(U_\mu(p_i^*)-U_\nu(p_i^*)\bigr)-|U_\nu(p_i^*)-U_{\nu_k}(p_i^*)|
\ge \vartheta-\frac{\vartheta}{2}=\frac{\vartheta}{2}>0,
\]
so  $(\mu,\nu_k)$ is irreducible for every $k$, proving \ref{prop:approx_nu_1:cx}.

To prove \ref{prop:approx_nu_1:dom}, define for $p\in(0,1)$
\begin{align*}
	Q(p):=2\bigl(|Q_\nu(p/2)|+|Q_\nu((1+p)/2)|+1\bigr)
+\sum_{k\ge1}(M_k+1)\bigl(\mathbf 1_{(0,4\delta_k]}(p)+\mathbf 1_{[1-4\delta_k,1)}(p)\bigr)\\
+2\left(\sum_{k\ge1}(|a_k|+1)\mathbf 1_{(0,4\alpha_k]}(p)
+\sum_{k\ge1}(|b_k|+1)\mathbf 1_{[1-4\beta_k,1)}(p)\right).
\end{align*}
Then $Q\in L^1(0,1)$ since $Q_\nu\in L^1(0,1)$ and $\sum_{k\ge1}4\delta_k(M_k+1)\le \sum_{k\ge1} 4\cdot\frac{2^{-k}}{8}<\infty$, and moreover
\[
\alpha_k(|a_k|+1)\le \alpha_k|Q_\nu(\alpha_k)|+(s_k)_+ +\alpha_k \le \int_0^{\alpha_k}|Q_\nu(u)|du + (D_k+E_k) +\alpha_k \le 4\cdot 2^{-k}
\]
\[
\beta_k(|b_k|+1)\le \int_{1-\beta_k}^1|Q_\nu(u)|du + (D_k+E_k) + \beta_k\le 4\cdot 2^{-k},
\]
which imply $\sum_k 4\alpha_k(|a_k|+1)<\infty$ and $\sum_k 4\beta_k(|b_k|+1)<\infty$.

Fix $k$ and $p\in(0,1)$.
If $p\in(0,4\delta_k]\cup[1-4\delta_k,1)$, then $|Q_{\nu_k}(p)|\le M_k\le Q(p)$.
Assume $p\in[4\delta_k,1-4\delta_k]$ and set
\[
r^-:=\frac{p-\delta_k}{1-\delta_k},\qquad r^+:=\frac{p}{1-\delta_k}.
\]
By \eqref{eq:quantile_mixture_sandwich} and $(p-\delta_k)_+=p-\delta_k$, we have $Q_{\overline\nu_k}(r^-)\le Q_{\nu_k}(p)\le Q_{\overline\nu_k}(r^+)$.
Moreover $p\ge4\delta_k$ gives $r^-\ge p-\delta_k\ge p/2$, while $p\le 1-4\delta_k$ and $\delta_k\le 1/8$ give $\frac{1}{1-\delta_k}\le 1+2\delta_k$ and hence $r^+\le p(1+2\delta_k) \le p+2\delta_k  \le (1+p)/2$.
Thus $r^-,r^+\in I_p:=[p/2,(1+p)/2]$.
Since $Q_{\overline\nu_k}$ is non-decreasing, we obtain
\[
|Q_{\nu_k}(p)|\le |Q_{\overline\nu_k}(r^-)|+|Q_{\overline\nu_k}(r^+)|.
\]
Fix $r\in I_p$. By \eqref{eq:quantile_sandwich} and $\varepsilon_k\le1$,
\[
|Q_{\overline\nu_k}(r)|\le |Q_{\widehat\nu_k}(r)|+1.
\]
If $r\in(\alpha_k,1-\beta_k)$ then $Q_{\widehat\nu_k}(r)=Q_\nu(r)$ and $r\in I_p$ implies $|Q_\nu(r)|\le |Q_\nu(p/2)|+|Q_\nu((1+p)/2)|$ by monotonicity of $Q_\nu$.
If $r\le \alpha_k$ then $Q_{\widehat\nu_k}(r)=a_k$ and $r\ge p/2$ implies $p\le 2\alpha_k<4\alpha_k$, hence $\mathbf 1_{(0,4\alpha_k]}(p)=1$.
If $r\ge 1-\beta_k$ then $Q_{\widehat\nu_k}(r)=b_k$ and $r\le (1+p)/2$ implies $p\ge 1-2\beta_k>1-4\beta_k$, hence $\mathbf 1_{[1-4\beta_k,1)}(p)=1$.
Combining these observations yields, for every $r\in I_p$,
\[
|Q_{\overline\nu_k}(r)|
\le |Q_\nu(p/2)|+|Q_\nu((1+p)/2)|+1+(|a_k|+1)\mathbf 1_{(0,4\alpha_k]}(p)+(|b_k|+1)\mathbf 1_{[1-4\beta_k,1)}(p).
\]
Applying this bound to $r=r^-$ and $r=r^+$ gives $|Q_{\nu_k}(p)|\le Q(p)$.
This proves \ref{prop:approx_nu_1:dom}.
\end{proof}

\begin{proposition}[Approximation of the reference measure]
\label{prop:approx_q_1}
Let $q\in \Prob(\RR)$ be such that $q \ll \lambda$.
Then there exists $(q_k)_{k\in\NN} \subseteq \Prob(\RR)$ such that
\begin{enumerate}[label=(\roman*)]
	\item $\supp(q_k)$ is a bounded interval  for all $k \in \NN$,
	\item $q_k \ll \lambda$ and $\rho_{q_k} \in L^\infty(\RR)$ for all $k \in \NN$,
	\item $\|\rho_{q_k}-\rho_q\|_{L^1(\RR)} \to 0$ as $k\to\infty$.
\end{enumerate}
\end{proposition}

\begin{proof}
Choose $R_0>0$ such that $q([-R_0,R_0])>0$, and set $R_k:=R_0+k$ for every $k\in\NN$.
Define
\[
g_k(x):=\mathbf{1}_{[-R_k,R_k]}(x)\,(\rho_q(x)\wedge k),
\qquad
\rho_{q_k}(x):=\left(1-\frac1k\right)\frac{g_k(x)}{\int_\RR g_k(x)\,dx}
+\frac1k\,\frac{\mathbf{1}_{[-R_k,R_k]}(x)}{2R_k}.
\]
Since $q([-R_0,R_0])>0$, one has $\int_\RR g_k(x)\,dx>0$ for every $k$, so $\rho_{q_k}$ is well-defined.
Moreover, $q_k\ll\lambda$, $\rho_{q_k}\in L^\infty(\RR)$, and
\[
\rho_{q_k}(x)\ge \frac{1}{2kR_k}>0
\qquad\text{for every }x\in[-R_k,R_k],
\]
while $\rho_{q_k}(x)=0$ for every $x\notin[-R_k,R_k]$. Hence
\[
\supp(q_k)=[-R_k,R_k].
\]

Now observe that, for every $x\in\RR$,
\[
g_k(x) \uparrow \rho_q(x)
\qquad\text{as }k\to\infty.
\]
Therefore, by the Monotone Convergence Theorem,
\[
\int_\RR g_k(x)\,dx \longrightarrow \int_\RR \rho_q(x)\,dx = 1.
\]
Also,
\[
\|g_k-\rho_q\|_{L^1(\RR)}
=\int_\RR (\rho_q(x)-g_k(x))\,dx
=1-\int_\RR g_k(x)\,dx
\longrightarrow 0.
\]

Finally,
\[
\left\|\frac{g_k}{\int_\RR g_k(x)\,dx}-\rho_q\right\|_{L^1(\RR)}
\le \left\|\frac{g_k}{\int_\RR g_k(x)\,dx}-g_k\right\|_{L^1(\RR)}+\|g_k-\rho_q\|_{L^1(\RR)}.
\]
Since $\|g_k\|_{L^1(\RR)}=\int_\RR g_k(x)\,dx$, we get
\[
\left\|\frac{g_k}{\int_\RR g_k(x)\,dx}-g_k\right\|_{L^1(\RR)}
=\left|\frac1{\int_\RR g_k(x)\,dx}-1\right|\int_\RR g_k(x)\,dx
=\longrightarrow 0.
\]
Therefore,
\[
\|\rho_{q_k}-\rho_q\|_{L^1(\RR)}
\le \left(1-\frac1k\right)
\left\|\frac{g_k}{\int_\RR g_k(x)\,dx}-\rho_q\right\|_{L^1(\RR)}
+\frac1k\left\|\frac{\mathbf{1}_{[-R_k,R_k]}}{2R_k}-\rho_q\right\|_{L^1(\RR)}.
\]
Since both $\frac{\mathbf{1}_{[-R_k,R_k]}}{2R_k}$ and $\rho_q$ have $L^1$-norm equal to $1$, we have
\[
\left\|\frac{\mathbf{1}_{[-R_k,R_k]}}{2R_k}-\rho_q\right\|_{L^1(\RR)}\le 2,
\]
so
\[
\|\rho_{q_k}-\rho_q\|_{L^1(\RR)}
\le
\left\|\frac{g_k}{\int_\RR g_k(x)\,dx}-\rho_q\right\|_{L^1(\RR)}+\frac{2}{k}
\longrightarrow 0.
\]
\end{proof}

\printbibliography

\end{document}